\documentclass[a4paper,10pt]{amsart}
\usepackage[T1]{fontenc}
\usepackage{lmodern}
\usepackage{amsmath}
\usepackage{amsfonts}
\usepackage{amssymb}
\usepackage{amsthm}
\usepackage{thmtools}
\usepackage{mathtools}
\usepackage{booktabs}
\usepackage{enumitem}
\usepackage{tikz-cd}
\usetikzlibrary{decorations.pathmorphing}
\tikzcdset{row sep/normal=1.35em}
\usepackage[normalem]{ulem}
\usepackage{multirow}
\usepackage{placeins}
\usepackage{array}
\usepackage{xr}
\usepackage[pdftex,
            pdfauthor={},
            pdftitle={Computing G-Crossed Extensions and Orbifolds of Vertex Operator Algebras},
            pdfsubject={},
            pdfkeywords={},
            pdfproducer={},
            pdfcreator={}]{hyperref}

\theoremstyle{plain}
\newtheorem{thm}{Theorem}[section]
\newtheorem{prop}[thm]{Proposition}
\newtheorem{cor}[thm]{Corollary}
\newtheorem{lem}[thm]{Lemma}

\newtheorem*{thm*}{Theorem}
\newtheorem*{conj*}{Conjecture}
\newtheorem*{prop*}{Proposition}

\theoremstyle{definition}
\newtheorem{defi}[thm]{Definition}

\newtheorem*{nota*}{Notation}
\newtheorem{rem}[thm]{Remark}
\newtheorem{ex}[thm]{Example}

\newtheorem{concl}[thm]{Conclusion}

\newtheorem{idea}{Idea}
\newtheorem{prob}[idea]{Problem}
\newtheorem*{prob*}{Problem}

\newtheoremstyle{introTheorems}
  {}
  {}
  {\itshape}
  {}
  {\bfseries}
  {}
  { }
  {\thmname{#1}
  \textnormal{\bf \thmnote{#3}.$\hspace{-.1cm}$}
  }
\theoremstyle{introTheorems}
\newtheorem{introTheorem}{Theorem}

\makeatletter
\newcommand*{\addFileDependency}[1]{
\typeout{(#1)}
\@addtofilelist{#1}
\IfFileExists{#1}{}{\typeout{No file #1.}}
}\makeatother

\newcommand{\Q}{\mathbb{Q}}
\newcommand{\Z}{\mathbb{Z}}

\newcommand{\N}{\mathbb{Z}_{\geq0}}
\newcommand{\C}{\mathbb{C}}
\newcommand{\R}{\mathbb{R}}

\renewcommand{\i}{\mathrm{i}}
\newcommand{\e}{\operatorname{e}} 

\newcommand{\Aut}{\operatorname{Aut}}

\newcommand{\Ind}{\operatorname{Ind}}
\newcommand{\rk}{\operatorname{rk}}

\newcommand{\sign}{\operatorname{sign}}
\newcommand{\voa}{vertex operator algebra}
\newcommand{\Voa}{Vertex operator algebra}
\newcommand{\VOA}{Vertex Operator Algebra}

\newcommand{\fpvosa}{fixed-point vertex operator subalgebra}

\newcommand{\mtc}{modular tensor category}
\newcommand{\mtcs}{modular tensor categories}

\newcommand{\id}{\operatorname{id}}
\newcommand{\eps}{\varepsilon}

\newcommand{\h}{\mathfrak{h}}

\newcommand{\Irr}{\operatorname{Irr}}
\newcommand{\V}{\mathcal{V}}

\newcommand{\strat}{strongly rational}

\renewcommand{\d}{\mathrm{d}}
\newcommand{\OO}{\operatorname{O}}

\newcommand{\Vect}{\operatorname{Vect}}
\newcommand{\no}{\,{\raise0.25em\hbox{$\mathop{\hphantom{\cdot}}\limits^{_{\circ}}_{^{\circ}}$}}\,}
\newcommand{\cC}{\mathcal{C}}
\newcommand{\cD}{\mathcal{D}}

\newcommand{\cB}{\mathcal{B}}

\newcommand{\cZ}{\mathcal{Z}}

\newcommand{\sslash}{/\mkern-6mu/}
\newcommand{\Rep}{\operatorname{Rep}}
\newcommand{\loc}{\mathrm{loc}}

\newcommand{\Heis}{\V_\h}
\newcommand{\SO}{\operatorname{SO}}

\newcommand{\auto}{{\bar{g}}}

\newcommand{\TY}{\mathcal{TY}}
\newcommand{\LM}{\mathcal{G\!LM}}
\newcommand{\X}{\mathsf{X}}
\newcommand{\SMatrix}{\mathcal{S}}
\newcommand{\TMatrix}{\mathcal{T}}
\newcommand{\eqi}{\varphi}
\newcommand{\DiscG}{{\bar{G}}}
\newcommand{\DiscGamma}{{\bar{\Gamma}}}
\newcommand{\DiscQ}{{\bar{Q}}}
\newcommand{\DiscS}{{\bar{\sigma}}}
\newcommand{\DiscO}{{\bar{\omega}}}
\newcommand{\DiscB}{B_{\DiscQ}}
\newcommand{\Discq}{\bar{q}}
\newcommand{\eval}{\operatorname{ev}} 
\newcommand{\coeval}{\operatorname{coev}} 
\newcommand{\basis}{\mathfrak{v}} 
\newcommand{\basisTwo}{\mathfrak{w}} 
\newcommand{\unit}{\mathbf{1}}
\newcommand{\Pic}{\operatorname{Pic}}

\begin{document}

\title[$G$-Crossed Extensions and VOA Orbifolds]{Computing $G$-Crossed Extensions and\\Orbifolds of Vertex Operator Algebras}
\author[César Galindo, Simon Lentner and Sven Möller]{César Galindo,\textsuperscript{\lowercase{a}} Simon Lentner\textsuperscript{\lowercase{b}} and Sven Möller\textsuperscript{\lowercase{b},\lowercase{c}}}

\thanks{\textsuperscript{a}{Universidad de los Andes, Bogotá, Colombia}}
\thanks{\textsuperscript{b}{Universität Hamburg, Hamburg, Germany}}
\thanks{\textsuperscript{c}{Research Institute for Mathematical Sciences, Kyoto University, Kyoto, Japan}}

\thanks{\SMALL Email: \href{mailto:cn.galindo1116@uniandes.edu.co}{\nolinkurl{cn.galindo1116@uniandes.edu.co}}, \href{mailto:simon.lentner@uni-hamburg.de}{\nolinkurl{simon.lentner@uni-hamburg.de}}, \href{mailto:math@moeller-sven.de}{\nolinkurl{math@moeller-sven.de}}}

\begin{abstract}
In this article, we develop tools for computing $G$-crossed extensions of braided tensor categories. Their equivariantisations appear as categories of modules of fixed-point subalgebras (or orbifolds) of \voa{}s and are often difficult to determine.

As the first tool, we show how the seminal work of Etingof, Nikshych and Ostrik on the uniqueness of $G$-crossed extensions can be used to determine the category of modules of orbifold \voa{}s. As an application, we determine the modular tensor category of the orbifold of a lattice \voa{} under a lift of $-\!\id$ for a lattice with odd-order discriminant form. In that case, the de-equivariantisation is of Tambara-Yamagami type.

As the second tool, we describe how $G$-crossed extensions and condensations by commutative algebras commute in a suitable sense. This leads to an effective approach to compute new $G$-crossed extensions. As one application, we produce the coherence data that are then used in \cite{GLM24b} to define a generalisation of the Tambara-Yamagami categories with more than one simple object in the twisted sector. This also yields the modular tensor category of the orbifold of an arbitrary lattice \voa{} under a lift of $-\!\id$. Finally, we sketch how to categorically approach the general problem of lattice orbifolds under lifts of arbitrary lattice involutions.
\end{abstract}

\maketitle
\setcounter{tocdepth}{1}
\tableofcontents
\setcounter{tocdepth}{2}


\section{Introduction}


\subsection{Motivation}

Braided tensor categories and specialisations, such as modular tensor categories, have emerged as vitally important structures in low-dimensional physics, such as $2$-dimensional conformal field theories and topological phases of matter. In the context of the former, they arise as local representations of fields of observables, i.e.\ as categories of modules of \voa{}s.


The main object of interest in this article is the categorical notion of a braided $G$\nobreakdash-crossed extension \cite{ENO10} (see \autoref{def_Gcrossed}). Let $\cB$ be a braided tensor category, together with an action of a (finite) group $G$ by braided tensor autoequivalences. Then a \emph{braided $G$-crossed extension} $\cC\supset\cB$ is a $G$-graded tensor category
\begin{equation*}
\cC=\bigoplus_{g\in G}\cC_g
\end{equation*}
endowed with a $G$-action satisfying $g_*\cC_h=\cC_{ghg^{-1}}$ and with a $G$-crossed braiding $\cC_g\otimes \cC_h\to \cC_{ghg^{-1}}\otimes\cC_g$, where the identity component is $\cC_1=\cB$ with the given braiding and $G$-action. For a braided $G$-crossed extension $\cB\subset\cC$, the $G$-equi\-vari\-an\-ti\-sa\-tion $\cC\sslash G$, also referred to as gauging of $\cB$ by $G$, is again a braided tensor category extending $\cB\sslash G$ (see, e.g., \cite{CGPW16}).

\medskip

One important application of the notion of braided $G$-crossed extensions, and the one we shall pursue in the present work, is orbifold theory of \voa{}s, i.e.\ the study of the \fpvosa{} $\smash{\V^G}$ of a \voa{} $\V$ under the action of a (finite) group $G$.

\Voa{}s and their representations axiomatise $2$-dimensional conformal field theories in physics. Under suitable regularity assumptions, the category of modules $\cB\coloneqq\Rep(\V)$ is a (pseudo-unitary) modular tensor category. If the group $G$ acts on $\V$, then the category of modules $\cB=\cC_1$ together with the categories $\cC_g=\smash{\Rep^g(\V)}$ of $g$-twisted modules for every $g\in G$ often forms a braided $G$-crossed tensor category $\cC=\smash{\bigoplus_{g\in G}\cC_g}=\smash{\Rep^G(\V)}$, and the $G$-equi\-vari\-an\-ti\-sa\-tion $\cC\sslash G$ is the modular tensor category $\smash{\Rep(\V^G)}$ of modules of the vertex operator subalgebra $\smash{\V^G}$ fixed by $G$, the so-called $G$-orbifold \cite{McR21b}. (We remark that a formulation of the above statement in the conformal-net axiomatisation of conformal field theory is given in \cite{Mue05}.)

The aim of this article is to devise methods that facilitate the computation of braided $G$-crossed extensions, and thereby of the full modular tensor categories (in particular, beyond just listing simple objects and fusion rules) of modules of \voa{} orbifolds, which is a notoriously difficult problem even in comparably easy examples (see, e.g., \cite{ADL05,Moe16,EMS20a,Lam20,DNR21b}). In contrast to earlier work, we shall rely very little on explicit constructions of twisted modules or twisted Zhu algebras of \voa{}s.

\medskip


At the centre of our approach (see also \cite{GJ19}) is a fundamental result by Etingof, Nikshych and Ostrik \cite{ENO10} and Davydov and Nikshych \cite{DN21} on the existence and uniqueness of braided $G$-crossed extensions, which we now state (in the version recorded in \autoref{thm_ourENO_braided}): \emph{for a given modular tensor category $\cB$ with an action of a finite group $G$, there is some obstruction $\mathrm{O}_4\in H^4(G,\C^\times)$, and if it vanishes, there exists a braided $G$-crossed extension $\cC\supset\cB$, unique up to a choice $\omega\in H^3(G,\C^\times)$ modifying the associator by a scalar depending on the $G$-degree.}

Importantly, this means that if for a given $\cB$ with $G$-action we are able to explicitly construct some braided $G$-crossed extension $\cC$, then we have essentially already determined all such extensions (see also, e.g., Section~3.2 in \cite{Bis18}). This leads to the following straightforward but nonetheless effective approach, which we want to advertise and which is contemporaneously employed in \cite{GR24}:
\begin{idea}\label{idea_1}
Given a \voa{} $\V$ with a (faithful) action of some finite group $G$, compute the category $\cC=\smash{\Rep^G(\V)}$ of $g$-twisted modules for $g\in G$ (and implicitly determine the orbifold category $\smash{\Rep(\V^G)}\cong\cC\sslash G$) by considering all possible braided $G$-crossed extensions of $\cB=\Rep(\V)$, indexed by $\omega\in H^3(G,\C^\times)$. Then, identify the correct $\omega$ by taking, e.g., the $G$-ribbon twist into account.
\end{idea}

\smallskip

However, it is still a highly nontrivial task to construct the braided $G$-crossed tensor category $\cC\supset\cB$, even in small examples. To this end, we establish in this article (see \autoref{thm_currentExtVsCrossedExt}, see also \cite{BJLP19}) the following tool. Recall that for a commutative algebra $A$ in a braided tensor category $\cB$ there is the tensor category $\cB_A$ of $A$-modules with $\otimes_A$ and a braided tensor subcategory of local modules $\cB_A^\loc$. We call $\cB_A^\loc$ the \emph{condensation} of $\cB$ by $A$ and denote it by $\cB\leadsto\cB_A^\loc$.
\begin{idea}\label{idea_2}
Taking braided $G$-crossed extensions $\cC\supset\cB$ commutes in a certain sense with condensing both categories to the categories of modules $\cC_A$ and $\cB_A$.

More precisely, let $\cC\supset\cB$ be a braided $\bar{G}$-crossed extension, and $A$ a commutative algebra in $\cB$ that can be endowed with the structure of a $G$-equivariant algebra $(A,\eqi)\in \cC\sslash G$ for some central extension $G$ of $\bar{G}$. We give a suitable notion of $G$-local modules $\cC^{\loc}_A$ over $(A,\eqi)$ in the braided $\bar{G}$-crossed extension $\cC$ such that this produces a braided $G$-crossed extension of the condensation $\cB^{\loc}_A$ of $\cB$ by $A$ in the usual sense:
\begin{equation*}
\begin{tikzcd}
\cB
\arrow[hookrightarrow]{rrr}{\text{$\bar{G}$-crossed ext.}}
\arrow{d}{\text{$\otimes_A$}}
\arrow[bend right=50,rightsquigarrow]{dd}{}
&&&
\cC
\arrow{d}{\text{$\otimes_A$}}
\arrow[bend left=50,rightsquigarrow]{dd}{}
\\
\cB_A
&&&
\cC_A
\\
\cB_A^\loc
\arrow[hookrightarrow]{u}{}
\arrow[hookrightarrow]{rrr}{\text{$G$-crossed ext.}}
&&&
\cC_A^\loc
\arrow[hookrightarrow]{u}{}
\end{tikzcd}
\end{equation*}
\end{idea}

In the \voa{} setting, a commutative algebra corresponds to a conformal extension, such as a simple-current extension, and condensation describes the category of modules over the extended \voa{} \cite{HKL15}. Hence, \autoref{idea_2} for \voa{}s states that taking conformal extensions commutes with taking $G$-orbifolds in a suitable sense. In this context, the central extension consists of \emph{lifts} of the group of automorphisms $\bar{G}$ of the underlying \voa{} to a group of automorphisms $G$ of the extended \voa{}.

For example, this approach in principle reduces the determination of ${G}$-orbifold categories of lattice \voa{}s $\V_L$ to the computation of $\bar{G}$-orbifold categories of Heisenberg \voa{}s $\Heis$ plus a purely categorical computation:
\begin{equation*}
\begin{tikzcd}
\Rep(\Heis)
\arrow[rightsquigarrow]{d}{}
\arrow[hookrightarrow]{rrr}{\text{$\bar{G}$-crossed ext.}}
&&&
\Rep^G(\Heis)
\arrow[rightsquigarrow]{d}{}
\\
\Rep(\V_L)
\arrow[hookrightarrow]{rrr}{\text{$G$-crossed ext.}}
&&&
\Rep^G(\V_L)
\end{tikzcd}
\end{equation*}
Here, we ignore for now the subtleties of infinite-order commutative algebra extensions (but see \cite{CMY22} and \autoref{sec_latticeDiscriminantForm}), and of infinite-dimensional spaces of intertwining operators that prevent us from defining $\smash{\Rep^G(\Heis)}$ rigorously (but see \autoref{sec_VOAIdea2_Els} on how to circumvent this for $G\cong\Z_2$).

\medskip


We apply our methods to the following initial class of examples, where we can already produce new results. Namely, we consider the case when the underlying category $\cB$ is pointed, i.e.\ when
\begin{equation*}
\cB\coloneqq\Vect_\Gamma^Q\quad(\text{simple objects }\C_a,a\in\Gamma)
\end{equation*}
is the braided tensor category of $\Gamma$-graded vector spaces for some finite, abelian group $\Gamma$, with associator $\omega$ and braiding $\sigma$ specified by a quadratic form $Q\colon\Gamma\to\C^\times$, as discussed in \autoref{ex_braidedVect} following \cite{JS93,DL93,EGNO15}. If $Q$ is nondegenerate, the pair $(\Gamma,Q)$ is called discriminant form, in which case $\smash{\Vect_\Gamma^Q}$ has a nondegenerate braiding and is hence a (pseudo-unitary) modular tensor category.

We first demonstrate \autoref{idea_1} directly: if $\Gamma$ has odd order and $G=\langle g\rangle\cong\Z_2$ acts by multiplication by $-1$ on $\Gamma$, then the braided $\Z_2$-crossed extension, as a tensor category, is a Tambara-Yamagami category \cite{TY98}
\begin{equation*}
\cC=\cC_1\oplus\cC_g=\Vect_\Gamma\oplus\Vect,
\end{equation*}
where the unique simple object $\X\in\cC_g$ has tensor product $\X\otimes \X=\smash{\bigoplus_{a\in\Gamma}}\C_a$. The possible $G$-crossed braidings and $G$-ribbon structures were determined in \cite{Gal22}. Based on this, we present in \autoref{thm_TYribbon} and \autoref{cor_TY} a description of the two possible braided $G$-crossed extensions (with positive quantum dimensions)
\begin{equation*}
\cC=\Vect_\Gamma^Q[\Z_2,\eps]
\end{equation*}
for $\eps=\pm 1$, and discuss how to distinguish the two possibilities based on the eigenvalues of the $\Z_2$-ribbon twist.

The $\Z_2$-equivariantisation $\cC\sslash\Z_2$ is again a modular tensor category (see \autoref{prop:TY_modular}). It was already described in \cite{GNN09} and is a special case of the more general results in \autoref*{GLM2sec_equiv} of \cite{GLM24b}.

Now, if $L$ is a positive-definite, even lattice, then the corresponding lattice \voa{} $\V_L$ has the category of modules $\Rep(\V_L)\cong\smash{\Vect_\Gamma^Q}$, where $\Gamma=L^*\!/L$ is the discriminant form of $L$ equipped with the nondegenerate quadratic form $Q\colon\Gamma\to\Q/\Z\to\C^\times$ defined by $Q(\lambda+L)=\e(\langle\lambda,\lambda\rangle/2)$ for $\lambda+L\in\Gamma$. If $\bar{G}=\Z_2$ is a group of isometries of $L$ that lifts to an action of $G=\Z_2$ as automorphisms on $\V_L$ and whose induced action on $\Gamma$ is multiplication by $-1$, then our approach produces the precise modular tensor category $\smash{\Rep(\V_L^G)}$ of modules of the orbifold \voa{} $\smash{\V_L^G}$ (see \autoref{thm_latTY} and \autoref{cor:latTY}), including all coherence data, provided that $|L^*\!/L|$ is odd. Note that typically in the \voa{} literature, only the isomorphism classes of simple objects and the fusion rules are explicitly known \cite{ADL05,Els17}, and this is already quite a nontrivial computation. A notable exception is \cite{Bis18,EG23}, which use the language of conformal nets and have obtained in these cases similar descriptions.

We point out that the two possible $\Z_2$-extensions $\smash{\Vect_\Gamma^Q}[\Z_2,\eps]$ for $\eps=\pm 1$ cannot be distinguished by $\Gamma$ and the action of $G$ on it but depend via the $G$-ribbon structure (i.e.\ the $L_0$-weights of the unique irreducible twisted $\V_L$-module) on the dimensions of the eigenspaces of the action of $G$ on the lattice $L$ (cf.\ \cite{Moe16,EMS20a} for $\cB=\Vect$). This is the expected behaviour in general: while the possible braided $G$-crossed extensions $\cC$ in \cite{ENO10,DN21} only depend on the action of $G$ on $\cB$, the exact $\omega\in H^3(G,\C^\times)$ is fixed by additional data (see \autoref{idea_1}).

\medskip

We now sketch how we propose to use \autoref{idea_2} to compute extensions of $\smash{\Vect_\Gamma^Q}$ by arbitrary $G$-actions, before we give again concrete results for $G=\Z_2$: the braided tensor category $\smash{\cB_A^\loc=\Vect_\Gamma^Q}$ is the category of modules of a commutative algebra~$A$ in the infinite braided tensor category $\cB=\smash{\Vect_{\R^d}^\DiscQ}$ associated with $(\R^d,\DiscQ)$ where $\DiscQ$ is the standard Euclidean quadratic form (or any nondegenerate, positive-definite quadratic form). This is in particular visible in the construction of lattice \voa{}s from the underlying Heisenberg \voa{}s (see, e.g., \cite{FLM88}). Using \autoref{idea_2} and our corresponding results, the problem of arbitrary braided $G$-crossed extensions of $\smash{\Vect_\Gamma^Q}$ reduces to the following fundamental problem, disregarding problems of infinite categories and groups:
\begin{prob}\label{problem_geoOrbifold}
The algebraic group $\bar{G}=\SO_n(\R)$, or any subgroup, acts as isometries on $(\R^d,\DiscQ)$ with the standard Euclidean quadratic form $\DiscQ$. What are the braided $G$-crossed extensions $\cC=\bigoplus_{g\in\bar{G}}\cC_g$ of $\cB=\smash{\Vect_{\R^d}^\DiscQ}$ with this action of $\bar{G}$?
\end{prob}

Note that $\cC_g$ only depends on $g$, not the ambient group $\bar{G}$, and there is little doubt that $\cC_g$ should be $\smash{\Vect_{(\R^d)^g}}$ for the fixed subspace $\smash{(\R^d)^g}$, but determining all involved structures is highly nontrivial.

For $\bar{G}=\Z_2$ acting on $(\R^d,\DiscQ)$, there is an orthogonal decomposition into copies of $\R$ with trivial action and copies with action $-\!\id$. So, \autoref{problem_geoOrbifold} is solved by taking Deligne products of $\smash{\Vect_\R}$ and (infinite) Tambara-Yamagami categories. Hence, in this case, \autoref{idea_2} produces the data for the definition of the braided $\Z_2$-crossed extensions $\smash{\Vect_\Gamma^Q[\Z_2,\eps]}$ of $\smash{\Vect_\Gamma^Q}$, where $\Gamma$ may now be of even order. This is a new result and one of the main contributions of this paper. On the other hand, the rigorous proof that these data actually define braided $G$-crossed tensor categories is quite involved and is given in \cite{GLM24b} (see \autoref{thm_evenTY_here}).

Applying this to lattice \voa{}s, we obtain the modular tensor categories $\smash{\Rep(\V_L^{\Z_2})}$ of lattice orbifolds under lifts of $-\!\id$ (see \autoref{thm_latmain} and \autoref{cor_latmain}), again including all coherence data (cf.\ \cite{ADL05}).

Moreover, \autoref{idea_2} produces the braided tensor categories of modules of arbitrary $\Z_2$-orbifolds of lattice \voa{}s $\V_L$ (see \autoref{sec_VOAIdea2}). We can circumvent the problems caused by $\R^d$ being infinite by only descending to the finite-index sublattice $L_+ \oplus L_- \subset L$ consisting of the intersection of $L$ with the $\Z_2$-eigenspaces. In fact, this is exactly how this case is treated in the \voa{} literature \cite{BE15,Els17}, where again only the fusion rules are determined.


\subsection{Extended Outline}

In \autoref{sec_ENO}, we review the notion of a braided $G$-crossed tensor category. Then we recall in \autoref{thm_ourENO_braided} the fundamental result from \cite{ENO10,DN21}, the classification of braided $G$-crossed extensions, in the version that we intend to use in this work.

In the quoted results, $G$-extensions $\cC=\bigoplus_{g\in G}\cC_g$ of a tensor category $\cB=\cC_1$ are classified by group homomorphisms from $G$ to the Brauer-Picard group of $\cB$, which is the group of $\cB$-bimodule categories $g\mapsto \cC_g$. The Brauer-Picard group in turn is shown to be equivalent to the braided autoequivalences of the centre $\cZ(\cB)$, by comparing the left and right action on the centre as bimodule category autoequivalences. This equivalence is however not easy to make explicit. For $\cB$ modular, the equivalence $\cZ(\cB)\cong \cB\boxtimes \cB^\mathrm{rev}$ can be used to define so-called alpha-induction $\Aut(\cB)\to \Aut(\cZ(\cB))$ (equivalently one can turn the module category into a bimodule category using the braiding), which produces the $G$-extensions with a compatible $G$-crossed braiding. We add to this the assertion, as maybe obvious to the experts, that a given $G$-action $G\to\cZ(\cB)$ uniquely determines the homomorphism to the Brauer-Picard group. Also note that from the start we assume a categorical action of the group (as this will be the case in the examples coming from \voa{}s), so the additional cohomological obstructions and choices in \cite{ENO10} do not appear.

We restrict ourselves to the case of (finite and semisimple) modular tensor categories and finite groups. Note, however, that the results on $G$-extensions are also true for nonsemisimple tensor categories with nondegenerate braiding \cite{Shi18} by \cite{DN21}. Moreover, it would be extremely helpful to have a version for infinite tensor categories and for algebraic groups $G$ (see, e.g., \autoref{problem_geoOrbifold}):
\begin{prob}
Give a classification of braided $G$-crossed extensions for $G$ an algebraic group acting on a braided tensor category $\cB$.
\end{prob}

We refer to \cite{Gal22} and \autoref*{GLM2sec_GRibbon} in \cite{GLM24b} for the definition of $G$-ribbon structures for braided $G$-crossed tensor categories and the fact that these correspond under equivariantisation to usual ribbon structures. We also recall (see \autoref*{GLM2sec_pseudo} in \cite{GLM24b}) the property of being pseudo-unitary (i.e.\ all quantum or categorical dimensions being positive), which is always true for (sufficiently nice) \voa{} examples and reduces the number of possible $G$-ribbon structures to at most one. In our examples, we can always directly give a ribbon structure, but for the future a corresponding version of the above result would be valuable:
\begin{prob}\label{problem_RibbonENO}
Give a classification of braided $G$-crossed extensions with (pseudo-unitary) $G$-ribbon structure for $G$ acting on a modular tensor category $\cB$ by ribbon autoequivalences.
\end{prob}

\smallskip

In \autoref{sec_TY}, we discuss a minimal nontrivial example of braided $G$-crossed extensions. Let $\cB=\smash{\Vect_\Gamma^Q}$ denote the braided tensor category of $\Gamma$-graded vector spaces, $\Gamma$ some finite, abelian group, with associator and braiding determined by a quadratic form $Q$ on $\Gamma$ \cite{JS93,EGNO15}. Let $G$ be a group acting (faithfully) by isometries on $(\Gamma,Q)$ and acting correspondingly (and strictly) on the category $\smash{\Vect_\Gamma^Q}$. An interesting open problem in this context is the computation of the unique braided $G$-crossed extension of $\cB$, up to a $3$-cocycle.

First, if $\Gamma$ has odd order and $G=\langle g\rangle\cong\Z_2$ acts by multiplication by $-1$, then the braided $\Z_2$-crossed extension, as a tensor category, is a Tambara-Yamagami category \cite{TY98}
\begin{equation*}
\cC=\cC_1\oplus\cC_g=\Vect_\Gamma\oplus\Vect,
\end{equation*}
where the unique simple object $\X\in \cC_g$ has tensor product $\X\otimes \X=\bigoplus_{a\in\Gamma} \C_a$ (cf.\ \cite{Bis18,EG23}). The possible $G$-crossed braidings and $G$-ribbon structures on $\cC$ have been determined in \cite{Gal22}. From this, in \autoref{thm_TYribbon} and \autoref{cor_TY}, we derive a description of the two braided $\Z_2$-crossed extensions $\smash{\Vect_\Gamma^Q}[\Z_2,\eps]=\cC$ and discuss how to distinguish them based on the eigenvalues of the $\Z_2$-ribbon twist. We also recall some facts about the appearing Gauss sums and the notion of the signature $\sign(\Gamma,Q)$ of a finite, abelian group with a quadratic form, which is again closely related to the lattice \voa{} application.

\medskip

In \autoref{sec_GExtensionVsCurrentExtension} we establish the relation between braided $G$-crossed extensions of a braided tensor category $\cB$ and condensations by a commutative algebra (see \autoref{idea_2}).

To this end, in \autoref{sec_CurrentExtension} we introduce the category of modules $\cB_A$ and of local modules $\smash{\cB_A^\loc}$ of a commutative algebra $A$ in a braided tensor category $\cB$, with tensor product $\otimes_A$ and associators.

Then, in \autoref{sec_DiscForm} we give as example the condensation of $\cB=\smash{\Vect_\DiscGamma^\DiscQ}$ by an algebra of the form $A=\bigoplus_{\lambda \in I} \C_\lambda$ for an isotropic subgroup $I\subset \DiscGamma$, resulting in the tensor category $\cB_A=\smash{\Vect_{\DiscGamma/I}^\omega}$ and the braided tensor category $\cB_A^\loc=\smash{\Vect_\Gamma^Q}$ for $\Gamma=I^\perp/I$ and $Q$ obtained from $\DiscQ$ restricted to $I^\perp$ and factored to $I^\perp/I$. While the result is certainly well-known both in a physics and mathematics context, we spend some effort to compute all structures of the resulting braided tensor category from the definitions because these computations serve as a template for the new cases considered in \autoref{sec_idea2_infTY} and because we could not find such a treatment in literature.

In \autoref{sec_ThmGExtensionVsCurrentExtension} we define $G$-local modules of a $G$-equivariant algebra $A\in \cB\sslash G$ in a braided $G$-crossed tensor category $\cC=\bigoplus_{g\in G}\cC_g$ with $\cC_1=\cB$. Then we prove the main result of this section (see also \cite{BJLP19}):
\begin{introTheorem}[\ref{thm_currentExtVsCrossedExt}]
The $G$-local modules $\cC_A^\loc$ over $A$ inside $\cC$ are a braided $G$-crossed extension of the local modules $\cB_A^\loc$ over $A$ inside $\cB$, and there is a commuting square
\begin{equation*}
\begin{tikzcd}
\cB
\arrow[hookrightarrow]{rrrr}{\text{$G$-crossed extension}}
\arrow{d}{\text{$\otimes_A$}}
\arrow[bend right=50,rightsquigarrow]{dd}{}
&&&&
\cC
\arrow{d}{\text{$\otimes_A$}}
\arrow[bend left=50,rightsquigarrow]{dd}{}
\\
\cB_A
&&&&
\cC_A
\\
\cB_A^\loc
\arrow[hookrightarrow]{u}{}
\arrow[hookrightarrow]{rrrr}{\text{$G$-crossed extension}}
&&&&
\cC_A^\loc
\arrow[hookrightarrow]{u}{}
\end{tikzcd}
\end{equation*}
\end{introTheorem}

\smallskip

In \autoref{sec_TYtoTY} we show as a first demonstration that the condensation from $\smash{\Vect_\DiscGamma^\DiscQ}$ to $\smash{\Vect_\Gamma^Q}$ with $\Gamma=I^\perp/I$ in \autoref{sec_DiscForm} comes with a corresponding condensation of the Tambara-Yamagami category $\smash{\Vect_\DiscGamma^\DiscQ}[\Z_2,\bar\eps]$ (from \autoref{sec_TY}) to another Tambara-Yamagami category $\smash{\Vect^{Q}_{\Gamma}}[\Z_2,\eps]$, and we find that this relates the two braided $\Z_2$-crossed extensions for the same choice $\eps=\bar\eps$ (see \autoref{prop_TYtoTY}):
\begin{equation*}
\begin{tikzcd}
\Vect_\DiscGamma^\DiscQ
\arrow[rightsquigarrow]{d}{}
\arrow[hookrightarrow]{rrr}{\text{$\Z_2$-crossed ext.}}
&&&
\Vect_\DiscGamma^\DiscQ[\Z_2,\bar\eps]
\arrow[rightsquigarrow]{d}{}
\\
\Vect_{\Gamma}^Q
\arrow[hookrightarrow]{rrr}{\text{$\Z_2$-crossed ext.}}
&&&
\Vect^{Q}_{\Gamma}[\Z_2,\eps]
\end{tikzcd}
\end{equation*}

\smallskip

Then, in \autoref{sec_TYEven} and \cite{GLM24b} we use our approach to obtain an interesting new result, the construction of the braided $\Z_2$-crossed extension of $\smash{\Vect_\Gamma^Q}$ with a certain $\Z_2$-action given by $-\!\id$ on $\Gamma$ and with $\Gamma$, in contrast to \autoref{sec_TY}, possibly of even order. While we produce the coherence data for the construction in this paper, the actual consistency proof is given in \cite{GLM24b}, \autoref*{GLM2thm_evenTY}.
\begin{introTheorem}[\ref{thm_evenTY_here}]
For $\eps\in\{\pm1\}$, the data given in \autoref*{GLM2sec_evenTY} of \cite{GLM24b} (and produced in \autoref{sec_idea2_infTY} of this paper) define a $\Z_2$-crossed ribbon fusion category
\begin{equation*}
\Vect_\Gamma^Q[\Z_2,\eps]=\Vect_\Gamma\oplus\Vect_{\Gamma/2\Gamma},
\end{equation*}
which is a braided $\Z_2$-crossed extension of $\smash{\Vect_\Gamma^Q}$ for a discriminant form $(\Gamma,Q)$ with the above categorical $\Z_2$-action. It is equipped with a natural choice of $\Z_2$-ribbon structure for which all quantum dimensions are positive.
\end{introTheorem}
In \autoref*{GLM2sec_equiv} of \cite{GLM24b} we compute the $\Z_2$-equivariantisation $\smash{\Vect_\Gamma^Q}[\Z_2,\eps]\sslash\Z_2$, which is again a modular tensor category.

Morally, the idea of the definition and the proof is that for the infinite, abelian group $\R^d$ with action $-\!\id$ there is an infinite version of the Tambara-Yamagami category, and proceeding as in \autoref{sec_TYtoTY},
\begin{equation*}
\begin{tikzcd}
\Vect_{\R^d}^\DiscQ
\arrow[rightsquigarrow]{d}{}
\arrow[hookrightarrow]{rrr}{\text{$\Z_2$-crossed ext.}}
&&&
\text{``inf.\ TY cat.''}
\arrow[rightsquigarrow]{d}{}
\\
\Vect_{\Gamma}^Q
\arrow[hookrightarrow]{rrr}{\text{$\Z_2$-crossed ext.}}
&&&
\Vect^{Q}_{\Gamma}[\Z_2,\eps]
\end{tikzcd}
\end{equation*}
produces the category in question with all structures (see \autoref{sec_idea2_infTY}). Rigorously, to avoid problems with infinite condensations, we then turn around the argument and take the resulting data as input for the definition of a finite braided $\Z_2$-crossed extension $\smash{\Vect^{Q}_{\Gamma}}[\Z_2,\eps]$ of the finite category $\smash{\Vect_\Gamma^Q}$, as done in \cite{GLM24b}.

\medskip

In \autoref{sec_VOA} we then turn to vertex operators algebras. Roughly speaking, a \voa{} $\V$ is a formal power series version of an associative, commutative algebra, and in good cases, it has attached to it a modular tensor category of modules $\Rep(\V)$. Moreover, for \voa{} modules there are associated analytic functions (called $q$-characters or graded dimensions) that form a vector-valued modular form under some action of the modular group $\operatorname{PSL}_2(\Z)$ on the Grothendieck ring. This action in turn is visible also on a purely categorical level in terms of the $\SMatrix$-matrix (matrix of traces of all double braidings) and $\TMatrix$-matrix (diagonal matrix of all twist eigenvalues).

In \autoref{sec_VOAIntro}, we briefly review the main concepts related to \voa{}s relevant for this work.

In \autoref{sec_LatticeVOA}, we introduce two important examples: the rank-$d$ Heisenberg \voa{} $\Heis$ associated with a vector space $\h=\R^d$ (viewed as abelian Lie algebra) equipped with a nondegenerate, symmetric bilinear form and the \voa{} $\V_L$ associated with a positive-definite, even lattice $L$. In fact, the category of (suitable) modules of $\Heis$ is $\Rep(\Heis)\cong\smash{\Vect_\h}$, and the category of modules of $\V_L$, which is $\Rep(\V_L)\cong\smash{\Vect_\Gamma^Q}$ with $(\Gamma,Q)=L^*\!/L$, is the (infinite) $A$-condensation by a commutative algebra $A$ related to $L$ (see \autoref{sec_latticeDiscriminantForm}).

In \autoref{sec_VOALatticeMinusOne}, we briefly review orbifolds of \voa{}s $\V$ by a finite group $G$ of \voa{} automorphisms: the fixed-point vertex operator subalgebra $\smash{\V^G}$ is called \emph{orbifold}, and computing the modular tensor category $\smash{\Rep(\V^G)}$ is an important and challenging task (see, e.g., \cite{DVVV89,DPR90} for early examples). Under favourable conditions, there is an associated braided $G$-crossed tensor category $\cC=\smash{\Rep^G(\V)}=\smash{\bigoplus_{g\in G}}\cC_g$ (referred to as the category of $g$-twisted $\V$-modules for $g\in G$) with $\Rep(\V)=\cC_1$, and $\smash{\Rep(\V^G)}=\cC\sslash G$ is the $G$-equivariantisation of $\cC$. As equivariantisation is a relatively straightforward and reversible operation (although it increases the complexity of the category), the main challenge lies in determining the braided $G$-crossed tensor category $\cC$.

Then, we demonstrate \autoref{idea_1} on the \voa{} side. We already briefly discussed the special case of odd discriminant forms (see \autoref{thm_latTY} and \autoref{cor:latTY}) above; here we summarise the general case (which also uses \autoref{idea_2}). We consider the lattice \voa{} $\V_L$ for a positive-definite, even lattice $L$ (for simplicity satisfying the stronger evenness condition in \autoref{sec:prel}, but now with $\Gamma=L^*\!/L$ possibly of even order) and compute the modular tensor category $\smash{\Rep(\V_L^G)}$ associated with a lift $G\cong\Z_2$ of the action of $\langle-\id\rangle$ on $L$.
\begin{introTheorem}[\ref{thm_latmain}]
The $\Z_2$-crossed ribbon fusion category (with positive quantum dimensions) of untwisted and twisted modules of the lattice \voa{} $\V_L$ under the above action of $G\cong\Z_2$ is given by
\begin{equation*}
\Vect_\Gamma^Q[\Z_2,\eps]
\end{equation*}
with $\eps=+1\in\{\pm1\}$ from \autoref{thm_evenTY_here}, constructed in \cite{GLM24b}.
\end{introTheorem}
Then the modular tensor category $\smash{\Rep(\V_L^G)}$ is given by the equivariantisation $\smash{\Vect_\Gamma^Q}[\Z_2,\eps]\sslash\Z_2$ (see \autoref{cor_latmain}), described in \autoref*{GLM2sec_equiv} of \cite{GLM24b}. In \autoref*{GLM2sec_latticedata} of \cite{GLM24b}, we also give a description of $\smash{\Rep^G(\V_L)}\cong\smash{\Vect_\Gamma^Q}[\Z_2,\eps]$ in terms of the lattice data. The fusion rules and simple objects were already determined in \cite{ADL05}; here we determine the complete category including the coherence data.

\medskip

Finally, in \autoref{sec_VOAIdea2} we demonstrate \autoref{idea_2} on the \voa{} side. Besides the infinite condensations, on which \autoref{thm_latmain} is (nonrigorously) based (see \autoref{sec_VOAevenTV}), a main example of condensations by commutative algebras are simple-current extensions, like conformal extensions of lattice \voa{}s $\V_L\subset\V_{L'}$ for positive-definite, even lattices $L\subset L'$ of the same rank.

In \autoref{sec_VOAIdea2_TYtoTY} we apply this idea to two Tambara-Yamagami categories, as was described categorically above (see \autoref{sec_TYtoTY}).

In \autoref{sec_VOAIdea2_Els} we briefly mention how to use the fact that for any lattice $L$ with an action of $\Z_2$ as isometries there is a finite-index sublattice
\begin{equation*}
L_+\oplus L_- \subset L
\end{equation*}
such that $\Z_2$ acts on $L_\pm$ by $\pm 1$, where $L_\pm=\mathfrak{h}_\pm \cap L$ are the intersections of $L$ with the eigenspaces of $\Z_2$ acting on $\mathfrak{h}= L\otimes_\Z\R$. For this sublattice, by the previous results, we can describe the braided $\Z_2$-crossed extension
\begin{align*}
\Rep^G(\V_{L_+\oplus L_-})&\cong
\bigl(\Vect_{\Gamma_+}\oplus\Vect_{\Gamma_+}\bigr)\boxtimes_{\Z_2}\Vect_{\Gamma_-}^{Q_-}[\Z_2,\eps]\\
&\cong\bigl(\Vect_{\Gamma_+}^{Q_+}\boxtimes\Vect_{\Gamma_-}^{Q_-}\bigr)\oplus\bigl(\Vect_{\Gamma_+}\boxtimes\Vect\bigr)
\end{align*}
with $\Gamma_+\coloneqq L_+^*/L_+$ and $\Gamma_-\coloneqq L_-^*/L_-$. Hence, by \autoref{thm_currentExtVsCrossedExt} or \autoref{idea_2}, the braided $G$-crossed tensor category $\smash{\Rep^G(\V_L)}$ we are looking for can be obtained as the $G$-local modules over the algebra in $\smash{\Vect_{\Gamma_+}\boxtimes\Vect_{\Gamma_-}}$ associated with~$L$.

In fact, this is morally also the idea behind the \voa{} results in \cite{BE15,Els17}, where the fusion rules for the $\Z_2$-twisted modules for an arbitrary $\Z_2$-action on a lattice are computed (more precisely, on the level of the equivariantisation). Our result, in principle, gives directly the braided $\Z_2$-crossed extensions of $\Rep(\V_L)=\smash{\Vect_\Gamma^Q}$, $\Gamma=L^*\!/L$, including all coherence data.


\subsection{Notation}

All vector spaces and (vertex, Lie, etc.) algebras will be over the base field $\C$, unless stated otherwise. All categories are enriched over $\Vect=\smash{\Vect_\C}$. (For the categorical statements, we could equally well take any algebraically closed field of characteristic zero.) By $\Z_n$ we shall always mean the cyclic group $\Z/n\Z$. We use the notation $\e(x)=e^{2\pi\i x}$.

Throughout, we denote by $G$ a (usually finite) group, written multiplicatively.
In general, $g$ denotes an arbitrary element of $G$, but in the special case of $G=\langle g\rangle\cong\Z_2$ we denote by $g$ the nontrivial element of $G$. By contrast, $\Gamma$ always denotes an abelian group that is written additively.
Quadratic forms on $\Gamma$ are usually written multiplicatively with values in $\C^\times$; they are related to the usual notion of quadratic forms with values in $\Q/\Z$ or $\R/\Z$ via $\e(\cdot)$.


\subsection{Preliminaries}\label{sec:prel}

By a (rational) \emph{lattice}, we mean a free abelian group $L$ of finite rank equipped with a nondegenerate bilinear form $\langle\cdot,\cdot\rangle\colon L\times L\to\Q$. The lattice $L$ is called \emph{integral} if $\langle v,w\rangle\in\Z$ for all $v,w\in L$. If $\langle v,v\rangle\in 2\Z$ for all $v\in L$, then $L$ is called \emph{even}. An integral lattice that is not even is called \emph{odd}. Given a lattice $L$, we denote by $L^*=\{v\in L\otimes_\Z\Q\mid\langle v,w\rangle\in\Z\text{ for all }w\in L\}$ the \emph{dual} of $L$, embedded via $\langle\cdot,\cdot\rangle$ into the ambient space of $L$.

Further, we call a lattice \emph{strongly even}, as in \autoref*{GLM2ass_strongeven} of \cite{GLM24b}, if $\langle v,w\rangle\in 2\Z$ for all $v,w\in L$.

If $L$ is even, then $L\subset L^*$ and the quotient $L^*\!/L$ is a finite, abelian group endowed with a nondegenerate quadratic form $Q\colon L^*\!/L\to\C^\times$ given by $Q(v+L)=\e(\langle v,v\rangle/2)$ for all $v\in L^*$. In that situation, we call $L^*\!/L$ the \emph{discriminant form} of $L$.

More generally, we call any pair $(\Gamma,Q)$ of a finite, abelian group $\Gamma$ together with a nondegenerate quadratic form $Q\colon\Gamma\to\C^\times$ a \emph{discriminant form} (sometimes called \emph{metric group}). The function $Q$ being quadratic means that $\smash{Q(na)=Q(a)^{n^2}}$ for all $a\in\Gamma$ and $n\in\Z$ and that $B_Q\colon\Gamma\times\Gamma \to\C^\times$ with $B_Q(a,b)\coloneqq Q(a+b)Q(a)^{-1}Q(b)^{-1}$ for $a,b\in\Gamma$ is bimultiplicative (and nondegenerate). We call $B_Q$ the \emph{associated bimultiplicative form}. Any discriminant form $(\Gamma,Q)$ can be realised as $L^*\!/L$ for some even lattice~$L$ \cite{Nik80}.

Given a discriminant form $(\Gamma,Q)$, a subgroup $I$ is called \emph{isotropic} if $Q(i)=1$ for all $i\in I$. In that case, $I^\perp\coloneqq\{a \in \Gamma\mid B_Q(a,i) = 1\text{ for all }i \in I\}$ is a subgroup of $\Gamma$ containing $I$, and the quotient $I^\bot/I$ naturally inherits the structure of a discriminant form, namely equipped with the nondegenerate quadratic form induced by $Q$.


\subsection*{Acknowledgements}

We are happy to thank Yoshiki Fukusumi, Terry Gannon, Hannes Knötzele, Christian Reiher, Andrew Riesen and Christoph Schweigert for valuable discussions. César Galindo expresses his gratitude for the hospitality and excellent working conditions provided by the Department of Mathematics at the University of Hamburg, where part of this research was conducted as a fellow of the Humboldt Foundation. Simon Lentner would like to express his gratitude for the hospitality of Thomas Creutzig and Terry Gannon at the University of Alberta.

César Galindo was partially supported by Grant INV-2023-162-2830 from the School of Science of Universidad de los Andes.  Simon Lentner was supported by the Humboldt Foundation. Sven Möller acknowledges support from the DFG through the Emmy Noether Programme and the CRC 1624 \emph{Higher Structures, Moduli Spaces and Integrability}, project numbers 460925688 and 506632645. Sven Möller was also supported by a Postdoctoral Fellowship for Research in Japan and Grant-in-Aid KAKENHI 20F40018 by the \emph{Japan Society for the Promotion of Science}.


\section{\texorpdfstring{$G$}{G}-Crossed Extensions of Braided Tensor Categories}\label{sec_ENO}

In this section, we introduce the main categorical notions used in this text, particularly braided $G$-crossed tensor categories and braided $G$-crossed extensions. We also state the results from \cite{ENO10,DN21} on the classification of braided $G$-crossed extensions. For a detailed discussion of ribbon structures and pseudo-unitarity in the context of braided $G$-crossed tensor categories, we refer the reader to \cite{GLM24b} (see also \cite{Gal22}).


\subsection{Tensor Categories}

In this article, tensor categories are $\C$-linear abelian monoidal categories, similar to \cite{DGNO10}. Further, we consider braided tensor categories, rigid tensor categories and ribbon categories. A fusion category is a tensor category that is finite, semisimple, rigid and has a simple tensor unit, as in \cite{EGNO15}. A modular tensor category is a ribbon fusion category with nondegenerate braiding. We define module and bimodule categories as in \cite{EGNO15}.

A rigid tensor category is called \emph{pointed} if every simple object is $\otimes$-invertible. In the following, we consider the typical description of pointed (braided) fusion categories associated with finite (abelian) groups and some group cohomological data (see, e.g., \cite{JS93,EGNO15}).

\begin{ex}[Pointed Fusion Categories]\label{ex_tensorVect}
Suppose that $G$ is a finite group and $\omega\in Z^3(G,\C^\times)$ a $3$-cocycle. We denote by $\smash{\Vect_G^\omega}$ the pointed fusion category of $G$-graded vector spaces, where the simple objects are $\C_g$, $g\in G$, and the tensor product is given by $\C_g\otimes\C_h=\C_{gh}$ with associator
\begin{equation*}
(\C_g\otimes\C_h)\otimes\C_k\overset{\omega(g,h,k)}{\longrightarrow}\C_g\otimes(\C_h\otimes\C_k)
\end{equation*}
for all $g,h,k\in G$. The pentagon identity for the associator holds exactly because $\omega$ is a $3$-cocycle.
\end{ex}

\begin{ex}[Pointed Braided Fusion Categories]\label{ex_braidedVect}
Let $\Gamma$ be a finite, abelian group and $(\sigma,\omega)\in Z_{\mathrm{ab}}^3(\Gamma,\C^\times)$ an abelian $3$-cocycle on $\Gamma$ \cite{EM50}. We denote by $\smash{\Vect_\Gamma^{\sigma,\omega}}$ the pointed braided fusion category given by the fusion category $\smash{\Vect_\Gamma^\omega}$ equipped with the braiding
\begin{equation*}
\C_a\otimes\C_b\overset{\sigma(a,b)}{\longrightarrow}\C_b\otimes\C_a
\end{equation*}
for $a,b\in\Gamma$. The hexagon identity for the braiding corresponds to the defining property of an abelian $3$-cocycle, i.e.
\begin{equation}\label{eq: abelian cocycle equations}
\frac{\omega(b,a,c)}{\omega(a,b,c)\omega(b,c,a)}=\frac{\sigma(a,b+c)}{\sigma(a, b)\sigma(a,c)},\quad\frac{\omega(a,b,c)\omega(c,a,b)}{\omega(a,c,b)}=\frac{\sigma(a+b,c)}{\sigma(a,c)\sigma(b,c)}
\end{equation}
for $a,b,c\in\Gamma$. For $\omega=1$, it states that $\sigma$ is bimultiplicative; otherwise, $\omega$ measures its deviation from bimultiplicativity.

Abelian $3$-coboundaries are given by $(\sigma_\kappa,\d\kappa)$ with $\sigma_\kappa(a,b)\coloneqq\kappa(a,b)\kappa(b,a)^{-1}$ for any function $\kappa\colon\Gamma\times\Gamma\to\C^\times$, and they correspond to braided tensor equivalences $\smash{\Vect_\Gamma^{\sigma,\omega}}\cong\smash{\Vect_\Gamma^{\sigma\sigma_\kappa,\omega\d\kappa}}$ that map each object $\C_a$ to itself but with a possibly nontrivial tensor structure given by $\kappa$. Thus, cohomology classes of abelian $3$-cocycles on $\Gamma$ correspond to equivalence classes of braided tensor categories on the abelian category $\smash{\Vect_\Gamma}$ with the given tensor product.

As shown in \cite{EM50}, the abelian $3$-cocycles $(\sigma,\omega)$, up to coboundaries, correspond bijectively to quadratic forms $Q\colon\Gamma\to\C^\times$. Recall that $B_Q\colon\Gamma\times\Gamma\to\C^\times$ with $B_Q(a,b)=Q(a+b)Q(a)^{-1}Q(b)^{-1}$ for $a,b\in\Gamma$ denotes the associated bimultiplicative form. In this correspondence, $B_Q(a,b)=\sigma(a,b)\sigma(b,a)$ is identified as the double-braiding and $Q(a)=\sigma(a,a)$ as the self-braiding. If the quadratic form $Q$ is nondegenerate, meaning that $B_Q$ is, then the braiding is nondegenerate. If no explicit choices of $(\sigma,\omega)$ are made, we denote the corresponding equivalence class of braided tensor categories by $\smash{\Vect_\Gamma^Q}$.

There is a rigid structure $\C_a^*=\C_{-a}$, $a\in\Gamma$, with the obvious evaluation and coevaluation. There is a canonical choice of a ribbon structure $\theta_{\C_a}=Q(a)$. As is discussed in more detail in \autoref*{GLM2sec_GRibbon} of \cite{GLM24b}, this is the unique choice of ribbon structure for which all simple objects $\C_a$ have quantum dimension~$1$, aligning with the Frobenius-Perron dimension. This is known as the pseudo-unitary choice. All possible ribbon structures differ from this by a homomorphism $\Gamma\to\{\pm1\}$.

In the particular case when $|\Gamma|$ is odd, the cohomology class associated with the quadratic form $Q$ can be represented by a distinguished abelian $3$-cocycle $(\sigma,\omega)$ with $\omega=1$ and $\sigma(a,b)=\smash{B_Q^{1/2}(a,b)}$ for $a,b\in\Gamma$. Here, we recall that when the abelian group $\Gamma$ has odd order, every bimultiplicative map $B\colon\Gamma\times\Gamma\to\C^\times$ has a unique bimultiplicative square root $B^{1/2}$. (An analogous statements holds for quadratic forms.)
\end{ex}

\medskip

We also consider group actions on tensor categories:
\begin{defi}\label{def_Gcrossed}
Let $\cC$ be a tensor category with the associativity constraint given by $(X\otimes Y)\otimes Z\overset{\alpha}{\to}X\otimes (Y\otimes Z)$ and $G$ a finite group. Then a \emph{$G$-action} on $\cC$ consists of the following data:
\begin{enumerate}[label=(\roman*)]
\item For every element $g\in G$, a functor
\begin{equation*}
g_*\colon\cC\to\cC.
\end{equation*}
We denote the image of an object $X\in\cC$ by $g_*(X)$ and the image of a morphism $f$ by $g_*(f)$.
\item For every pair of elements $g,h\in G$, a natural isomorphism
\begin{equation*}
T_2^{g,h}(X)\colon(gh)_*(X)\to g_*(h_*(X))
\end{equation*}
such that associativity holds: for all $g,h,l\in G$,
\begin{equation*}
T_2^{g,h}(l_*(X))\circ T_2^{gh,l}(X)=g_*(T_2^{h,l}(X))\circ T_2^{g,hl}.
\end{equation*}
\item A tensor structure $\tau^g$ on each functor $g_*$. That is, for every $g\in G$ and every pair of objects $X,Y\in\cC$, a natural isomorphism
\begin{equation*}
\tau^g_{X,Y}\colon g_*(X\otimes Y)\to g_*(X)\otimes g_*(Y)
\end{equation*}
such that the diagrams
\begin{equation*}
\begin{tikzcd}[column sep=small]
g_*((X \!\otimes\! Y) \!\otimes\! Z) \ar[r, "\tau^g"] \ar[d, "g_*(\alpha)"'] & g_*(X \!\otimes\! Y) \!\otimes\! g_*(Z) \ar[r, "\tau^g"] & (g_*(X) \!\otimes\! g_*(Y)) \!\otimes\! g_*(Z) \ar[d, "\alpha"] \\
g_*(X \!\otimes\! (Y \!\otimes\! Z)) \ar[r, "\tau^g"'] & g_*(X) \!\otimes\! g_*(Y \!\otimes\! Z) \ar[r, "\tau^g"'] & g_*(X) \!\otimes\! (g_*(Y) \!\otimes\! g_*(Z))
\end{tikzcd}
\end{equation*}
and
\begin{equation*}
\begin{tikzcd}
(gh)_*(X \otimes Y) \ar[r, "\tau^{gh}"] \ar[d, "T_2^{g,h}(X \otimes Y)"'] & (gh)_*(X) \otimes (gh)_*(Y) \ar[d, "T_2^{g,h}(X) \otimes T_2^{g,h}(Y)"] \\
g_*(h_*(X \otimes Y)) \ar[r, "\tau^g \tau^h"'] & g_*(h_*(X)) \otimes g_*(h_*(Y))
\end{tikzcd}
\end{equation*}
commute for all $g,h\in G$.
\end{enumerate}
A $G$-action is called \emph{strict} if all the monoidal functors $(g_*, \tau^g)$ are strict and the natural transformations $T_2$ are equalities.
\end{defi}

\begin{rem}
The data of an action of $G$ on $\cC$ is equivalent to having a tensor functor from the discrete tensor category $\smash{\underline{G}}$ (where the set $G$ serves as objects and the tensor product is defined by the multiplication in $G$) to the tensor category of tensor functors $\smash{\underline{\Aut}_\otimes(\cC)}$ along with monoidal natural isomorphisms.
\end{rem}

\begin{defi}
Let $\cC$ be a tensor category and $G$ a finite group acting on $\cC$. The \emph{$G$-equivariantisation} of $\cC$, denoted by $\cC\sslash G$, is the tensor category with objects pairs $(X,\eqi)$, where $X\in\cC$ and $\eqi$ is a family of isomorphisms, $\eqi_g\colon g_*(X) \to X$ for each $g \in G$, satisfying the following condition:
\begin{equation*}
\eqi_{gh}=\eqi_g\circ g_*(\eqi_h)\circ T_2^{g,h}.
\end{equation*}
Morphisms $f\colon(X,\eqi)\to(Y,\eqi')$ in $\cC\sslash G$ satisfy
\begin{equation*}
\eqi'_g\circ g_*(f)=f\circ\eqi_g.
\end{equation*}
The tensor product is $(X,\eqi)\otimes(Y,\eqi')=(X\otimes Y,\eqi'')$ with
\begin{equation*}
\eqi''_g=\eqi_g\otimes\eqi'_g\circ\tau^g_{X,Y},
\end{equation*}
and the tensor unit is $(\unit,\id_\unit)$.
\end{defi}

The simple objects of $\cC\sslash G$ can be parametrised as follows. The group $G$ acts on the set $\Irr(\cC)$ of equivalence classes of simple objects. If we fix a representative $X_i$ for every orbit and denote the corresponding stabiliser subgroup by $G_i$, we can choose isomorphisms $t_g\colon g_*(X_i)\to X_i$. The action restricted to $G_i$ defines a $2$-cocycle $\chi_i\in Z^2(G_i,\C^\times)$ by
\begin{equation}\label{eq: equivariant}
\chi_i(g,h)\id_{X_i}=t_{gh}^{-1}\circ t_g\circ g_*(t_h)\circ T_2^{g,h}.
\end{equation}
The cohomology class of $\chi_i$ does not depend on the choice of $\{t_g\mid g\in G_i\}$. The simple objects of $\cC\sslash G$, up to isomorphism, are in bijective correspondence with isomorphism classes of irreducible $\chi_i$-projective representations of $G_i$ for every $i$. For more details on the correspondence, see \cite{BN13}.
\begin{ex}\label{Ex: classification equivariant simple order 2}
Let $G=\langle g\rangle$ be a cyclic group of order~$2$ acting on a tensor category~$\cC$. The classification of simple objects of $\cC\sslash G$ in this case is explicit and straightforward. If a simple object $X$ has a trivial stabiliser, then $X^g\coloneqq X\oplus g_*(X)$ is an equivariant object with $\eqi_g=\id_{g_*(X)}\oplus T_2^{g,g}(X)$. If the stabiliser is $G$, then take $t_g\colon g_*(X)\to X$ and, using the equation \eqref{eq: equivariant}, define $\gamma=\chi(g,g)^{-1/2}$. The isomorphisms $\eqi_g^{(\pm)}=\pm\gamma t_g\colon g_*(X)\to X$ endow $X$ with two nonisomorphic equivariant structures.
\end{ex}


\subsection{\texorpdfstring{$G$}{G}-Crossed Tensor Categories}\label{sec_Gcrossed}

Given a tensor category $\cC$ and a group $G$, a (faithful) $G$-grading on $\cC$ is a decomposition $\cC=\bigoplus_{g\in G}\cC_g$ such that the tensor product $\otimes$ maps $\cC_g\times\cC_h$ to $\cC_{gh}$, the unit object is in $\cC_1$ and $\cC_g\neq0$ for all $g\in G$.

The central object in this text are braided $G$-crossed tensor categories (see, e.g., \cite{Tur00,EGNO15}), which we introduce in the following.
\begin{defi}
A \emph{braided $G$-crossed tensor category} is a tensor category $\cC$ endowed with the following structures:
\begin{enumerate}[label=(\roman*)]
\item an action of $G$ on $\cC$,
\item a faithful $G$-grading $\cC=\oplus_{g\in G}\cC_g$,
\item isomorphisms, called the \emph{$G$-braiding},
\begin{equation*}
c_{X,Y}\colon X\otimes Y\to g_*(Y)\otimes X
\end{equation*}
for $g\in G$, $X\in\cC_g$, $Y\in\cC$, natural in $X$ and $Y$.
\end{enumerate}
These structures must satisfy the following conditions:
\begin{enumerate}[label=(\alph*)]
\item $g_*(\cC_h)\subset\cC_{ghg^{-1}}$ for all $g, h \in G$.
\item The diagrams
\begin{equation*}\label{axiom 1 trenza}
\begin{tikzcd}
g_*(X\otimes Y) \ar{dd}{\tau^g_{X,Y}} \ar{rr}{g_*(c_{X,Y})} && g_*(h_*(Y)\otimes X)\ar{dd}{T_2^{ghg^{-1},h}(T_2^{g,h})^{-1}\tau^g_{h_*Y,X}}\\\\
g_*(X)\otimes g_*(Y) \ar{rr}{c_{g_*(X),g_*(Y)}} && (ghg^{-1})_*g_*(Y)\otimes g_*(X)
\end{tikzcd}
\end{equation*}
commute for all $g,h\in G$, $X\in\cC_h$, $Y\in\cC$.
\item The diagrams
\begin{equation*}\label{trenzas GG}
\begin{tikzcd}
X \otimes Y\otimes Z \ar{rr}{c_{X,Y\otimes Z}} \ar{d}{c_{X,Y}\otimes\id_{Z}} && g_*(Y\otimes Z)\otimes X \ar{d}{\tau^g_{Y,Z}} \\
g_*(Y)\otimes X\otimes Z \ar{rr}{\id_{g_*(Y)}\otimes c_{X,Z}}&& g_*(Y)\otimes g_*(Z) \otimes X
\end{tikzcd}
\end{equation*}
commute for all $g\in G$, $X\in\cC_g$, $Y,Z\in\cC$ and the diagrams
\begin{equation*}\label{trenzas 2}
\begin{tikzcd}
X \otimes Y\otimes Z \ar{rr}{c_{X\otimes Y, Z}} \ar{d}{\id_X\otimes c_{Y,Z}} && (gh)_*(Z)\otimes X\otimes Y \ar{d}{T_2^{g,h}} \\
X\otimes h_*(Z)\otimes Y\ar{rr}{c_{X,h_*(Z)}\otimes \id_Y}&& g_*h_*(Z)\otimes X \otimes Y
\end{tikzcd}
\end{equation*}
commute for all $g,h\in G$, $X\in\cC_g$, $Y\in\cC_h$, $Z\in\cC$.
\end{enumerate}
Here, for simplicity, we omitted the associativity constraint isomorphisms.
\end{defi}
The definition of equivalence of braided $G$-crossed tensor categories is given in \cite{G17}, Section~5.2.

We call $\cC$ a \emph{braided $G$-crossed extension} of a braided tensor category $\cB$ if $\cC_1=\cB$.

\medskip

A main application of braided $G$-crossed tensor categories is the following, which can be found in the present context in \cite{Tur10b} (appendix by Michael Müger):
\begin{prop}[\cite{DGNO10}]
Let $\cC$ be a braided $G$-crossed tensor category. The equivariantisation $\cC \sslash G$ is a braided tensor category that contains a fusion subcategory braided equivalent to the symmetric category $\Rep(G)$. This corresponds to those equivariant objects $(V,(\eqi_g)_{g\in G})$ where $V$ is in the tensor subcategory generated by the unit object of $\cC$.

Conversely, for every braided tensor category $\cD$ that contains the symmetric fusion category $\Rep(G)$ with trivial braiding, we can define a braided $G$-crossed tensor category $\cC\coloneqq\cD_G$ as the \emph{de-equivariantisation} by $G$. This is the tensor category of modules over $\smash{\C^G}$, the algebra of functions on $G$, which is an algebra object in $\Rep(G)$ and thus in $\cD$.
\end{prop}

\begin{rem}
The equivariantisation of a braided $G$-crossed tensor category is a nondegenerate braided tensor category if and only if the neutral component is a nondegenerate braided tensor category. For the existence of a ribbon element, see \autoref*{GLM2sec_GRibbon} in \cite{GLM24b}.
\end{rem}

\begin{ex}\label{ex_DrinfeldCenter}~
\begin{enumerate}[wide]
\item The archetypal example of a braided $G$-crossed tensor category is $\smash{\Vect_G}$, the pointed fusion category of $G$-graded vector spaces with the obvious $G$-grading. Up to equivalence, it has a unique braided $G$-crossed structure where the $G$-action is strict and given by $g_*(\C_h)=\C_{ghg^{-1}}$ for $g,h\in G$, with the $G$-braiding being the identity.

In this case, the equivariantisation $\smash{\Vect_G}\sslash G$ is canonically braided equivalent to $\mathcal{Z}(\smash{\Vect_G})$, the Drinfeld centre of $\smash{\Vect_G}$. See \cite{NNW09} for more details.

\item For an arbitrary $G$-graded extension of $\Vect$, we need to \emph{twist} the previous braided $G$-crossed tensor category and consider $\smash{\Vect_G^\omega}$ for a $3$-cocycle $\omega\in Z^3(G,\C^\times)$, as in \autoref{ex_tensorVect}. Given $g,h\in G$, define the maps
\begin{equation}\label{eq: maps for twisting g-crossed}
\begin{split}
\gamma_{g,h}(x)&\coloneqq\frac{\omega(g,h,x)\omega(ghx(gh)^{-1},g,h)}{\omega(g, hxh^{-1},h)},\\
\mu_g(x,y)&\coloneqq\frac{\omega(gxg^{-1},g,y)}{\omega(gxg^{-1},gy,g)\omega(g,x,y)}
\end{split}
\end{equation}
for $x,y\in G$. The $G$-action on $\smash{\Vect_G^\omega}$ on objects is the same as before, and the $G$-braiding is again the identity. However, the action now is nonstrict: for each $g \in G$, the constraints are
\begin{equation*}
\tau^g_{x,y}=\mu_g(x,y)\id_{g_*(\C_{xy})},\quad T_2^{g, h}(\C_x)=\gamma_{g,h}(x)\id_{g_*(h_*(\C_x))}.
\end{equation*}

The equivariantisation $\smash{\Vect_G^\omega}\sslash G$ is canonically braided equivalent to the twisted Drinfeld centre $\mathcal{Z}(\smash{\Vect_G^\omega})$. See \cite{NNW09} for more details.

\item Now, assume that $G$ is abelian; hence the $G$-action is trivial on objects. In this case, $\mu_a\in Z^2(G,\C^\times)$ is a $2$-cocycle for each $a\in G$. The cohomologies $\mu_a$ for all $a\in G$ measure whether the category $\smash{\Vect_G^\omega} \sslash G$ is pointed or not. If $\gamma_a=\delta(l_a)$ for certain $l_a\in C^1(G, \C^\times)$, then the simple objects of $\smash{\Vect_G^\omega}\sslash G$ correspond to elements in $G\times\widehat{G}$. Specifically, the pair $(a,\gamma)$ corresponds to the object $\C_a$ with the equivariant structure $\eqi_b=\gamma(b)l_a(b)\id_{\C_a}$. The tensor product is determined by the extension of $G$ by $\widehat{G}$ given by the $2$-cocycle $\gamma_{a,b}\frac{l_{ab}}{l_a l_b}\in Z^2(G,\widehat{G})$. See \cite{GJ16} for details.
\end{enumerate}
\end{ex}


\subsection{Classification of Braided \texorpdfstring{$G$}{G}-Crossed Extensions}\label{sec_ENOENO}

In this section, we recall from \cite{ENO10}, Theorem~7.12, the classification of all braided $G$-crossed fusion categories $\cC$ for a given braided fusion category $\cC_1$ with an action of a finite group~$G$. This is thoroughly discussed also in \cite{DN21}, and we remark that there the assumption of semisimplicity and rigidity is dropped, but finiteness in the sense of \cite{EGNO15} is retained (that is, they are working in the setting of finite braided tensor categories).

Given a finite group $G$, we denote by $\smash{\uuline{G}}$ the discrete monoidal $2$-category whose objects are elements of $G$ and where the $1$-morphisms and $2$-morphisms are identities, with the tensor product given by the product of $G$. Given a braided tensor category $\cB$, we denote by $\smash{\uuline{\Pic}}(\cB)$ the monoidal $2$-category of invertible module categories over $\cB$, where 1-morphisms are module equivalences and 2-morphisms are natural isomorphisms between module functors. The product in this monoidal structure is the relative tensor product of modules over $\cB$; see Definition~4.5 of \cite{ENO10}.

One of the main results of \cite{ENO10}, Theorem~7.7, states that faithful braided $G$-crossed extensions of a given fusion category $\cB$ can be classified by monoidal $2$-functors from $\smash{\uuline{G}}$ to $\smash{\uuline{\Pic}}(\cB)$.

Explicitly computing $\smash{\uuline{\Pic}}(\cC)$ can be a demanding task, even for pointed braided fusion categories. However, $\smash{\underline{\Pic}(\cC)}$ (the first truncation) as a categorical group is equivalent to $\smash{\underline{\Aut}_{\otimes,\mathrm{br}}(\cB)}$, as demonstrated in Theorem~1.1 of \cite{ENO10}. This result allows us to show, for example, that for $\smash{\Vect_\Gamma^Q}$, as in \autoref{ex_braidedVect}, with $\Gamma$ a finite, abelian group and $Q$ a nondegenerate quadratic form on it, $\Pic(\smash{\Vect_\Gamma^Q})$ is the associated orthogonal group, which is the group of automorphisms of $\Gamma$ that preserve $Q$.

\begin{thm}[\cite{DN21}, Section~8.3; \cite{ENO10}, Theorem~7.12]\label{thm_ourENO_braided}
Let $\cB$ be a finite tensor category with nondegenerate braiding, with an action of a finite group $G$ on $\cB$ by braided autoequivalences. Then there is a certain obstruction $\mathrm{O}_4\in H^4(G,\C^\times)$, and if and only if this obstruction vanishes, there exists a braided $G$-crossed tensor category $\cC$, where $\cC_1=\cB$ as a braided tensor category with $G$-action. The equivalence classes of extensions associated with the $G$-action form a torsor over $\omega\in H^3(G,\C^\times)$; see \autoref{rem: torsor H3}.
\end{thm}
\begin{proof}[Sketch of Proof]
By Section~8.3 in \cite{DN21}, braided $G$-crossed extensions of $\cB$ are classified up to equivalence by monoidal $2$-functors from $\smash{\uuline{G}}$ to $\smash{\uuline{\Pic}}(\cB)$. Now, since $\cB$ is nondegenerate, the $1$-truncation defines an equivalence
\begin{equation*}
\underline{\Pic}(\cB)\cong\underline{\Aut}_{\otimes,\mathrm{br}}(\mathcal{Z}(\cB)).
\end{equation*}
The obstruction to lifting the action $\smash{\underline{G}}\to\smash{\underline{\Aut}_{\otimes, \mathrm{br}}}(\mathcal{Z}(\cB))\cong\smash{\underline{\Pic}(\cB)}$ to $\smash{\uuline{\Pic}}(\cB)$ corresponds exactly to the $H^4(G,\C^\times)$-obstruction $\mathrm{O}_4$; see Proposition~2.23 in \cite{DN21}. The equivalence classes of these liftings form a torsor over $H^3(G,\C^\times)$; see Theorem~8.9 in \cite{ENO10}.
\end{proof}

\begin{rem}\label{rem: torsor H3}~
\begin{enumerate}[wide]
\item If $\cC$ is a braided $G$-crossed tensor category, then for $\omega\in Z^3(G, \C^\times)$, we can define a \emph{twisted} $\cC^\omega$ by modifying the associativity constraint with $\omega$ and \emph{twisting} the $G$-action using the maps in equation \eqref{eq: maps for twisting g-crossed}. See \cite{EG18b}, Proposition~2.2, for details.
\item The braided $G$-crossed extensions and their $G$-equivariantisations of the trivial modular tensor category $\cB=\Vect$ correspond to $\smash{\Vect_G^\omega}$ and its Drinfeld centre, respectively; see \autoref{ex_DrinfeldCenter}.
\end{enumerate}
\end{rem}


\section{Example: \texorpdfstring{$\Z_2$}{Z\_2}-Crossed Extensions of Tambara-Yamagami Type}\label{sec_TY}

We give a first example of a nontrivial braided $G$-crossed extension, namely one for which the monoidal structure is of Tambara-Yamagami type.

Throughout this section, let $G=\langle g\rangle\cong\Z_2$. Tambara and Yamagami classified all $\Z_2$-graded fusion categories in which all but one of the simple objects are invertible \cite{TY98}. The braided $G$-crossed structures on these were classified in \cite{Gal22}. We remark that compared with \cite{Gal22} we replace the quadratic form $q$ by $q^{-1}$ in order to match the usual convention for the associated bimultiplicative form $\chi$ (and we replace $\chi$ by $\sigma$ and $\tau$ by $\eps|\Gamma|^{-1/2}$ in order to improve readability).
\begin{defi}
The \emph{Tambara-Yamagami category} $\TY(\Gamma,\sigma,\eps)$ for a finite, abelian group $\Gamma$, a symmetric, nondegenerate bimultiplicative form $\sigma\colon \Gamma\times \Gamma\to \C^\times$ and a sign choice $\eps\in\{\pm1\}$ is a semisimple $\Z_2$-graded tensor category
\begin{equation*}
\TY(\Gamma,\sigma,\eps) = \Vect_\Gamma \oplus \Vect
\end{equation*}
with simple objects $\C_a$, $a\in \Gamma$, and $\X$, fusion rules
\begin{equation*}
\C_a\otimes \C_b=\C_{a+b},\quad
\C_a\otimes \X=\X\otimes \C_a =\X,\quad
\X\otimes \X=\bigoplus_{t\in\Gamma} \C_t
\end{equation*}
for $a,b\in\Gamma$ and the following nontrivial associators
{\allowdisplaybreaks
\begin{align*}
(\C_a\otimes \X)\otimes \C_b
&=X
\overset{\sigma(a,b)}{\xrightarrow{\hspace*{2.97cm}}}
\X=\C_a\otimes (\X\otimes \C_b),\\
(\X\otimes \C_a)\otimes \X
&=\bigoplus_{t\in \Gamma} \C_t
\overset{\sigma(a,t)}{\xrightarrow{\hspace*{1.6cm}}}
\bigoplus_{t\in \Gamma} \C_t
=\X\otimes (\C_a\otimes \X),\\
(\X\otimes \X)\otimes \X
&=\bigoplus_{t\in \Gamma} \X
\overset{\eps|\Gamma|^{-1/2}\sigma(t,r)^{-1}}{\xrightarrow{\hspace*{1.775cm}}}
\bigoplus_{r\in \Gamma}\X
=\X\otimes (\X\otimes \X).
\end{align*}
}%
\end{defi}

There is a rigid structure with (left) dual objects $\C_a^*=\C_{-a}$, $\X^*=\X$ and $\coeval_\X=\iota_{\C_0}$ and $\smash{\eval_\X=\eps|\Gamma|^{1/2}\pi_{\C_0}}$, where $\iota_{\C_t}$ and $\pi_{\C_t}$ are the canonical embeddings and projections, respectively, for the direct sum $\X\otimes\X=\bigoplus_{t\in\Gamma}\C_t$, and otherwise the obvious choices. The two zigzag identities relating $\eval_\X$ and $\coeval_\X$ hold because both the associator and inverse associator on $\X\otimes \X\otimes\X$ at $t=0$ are $\eps|\Gamma|^{-1/2}$.

\begin{thm}[\cite{Gal22}, Lemma~4.8, Theorem~4.9, Proposition~4.12]\label{thm_TYribbon}
Consider the Tambara-Yamagami category $\TY(\Gamma,\sigma,\eps)$.
\begin{enumerate}
\item There are two actions of $G = \langle g \rangle \cong \Z_2$ on $\TY(\Gamma,\sigma,\eps)$ such that the action on $\smash{\Vect_\Gamma}$ restricts to the action $g_*\C_a = \C_{-a}$ with trivial tensor structure $\tau^g_{\C_a,\C_b} = \id$ and trivial composition $T_2^{g,g}(\C_a) = \id$, $a,b\in\Gamma$. Namely, the actions $g_*\X = \X$ with trivial tensor structures $\tau^g_{\C_a,\X} = \tau^g_{\X,\C_a} = \tau^g_{\X,\X} = \id$ and with composition structures $T_2^{g,g}(\C_a) = \id$ and $T_2^{g,g}(\X) = \pm \id$.
\smallskip
\item The nonstrict action $T_2^{g,g}(\X) = -\id$ does not admit a $\Z_2$-braiding.
\smallskip
\item For the strict $\Z_2$-action, the $\Z_2$-braidings are in bijection with pairs $(q,\alpha)$ of a (nondegenerate) quadratic form $q\colon \Gamma \to \C^\times$ with associated bimultiplicative form $\sigma(a,b) = q(a+b) q(a)^{-1} q(b)^{-1}$ and a choice $\alpha$ of square root of
\begin{equation*}
\alpha^2 = \eps |\Gamma|^{-1/2} \sum_{a \in \Gamma} q(a)^{-1}.
\end{equation*}
Then the $\Z_2$-braiding is given by
{\allowdisplaybreaks
\begin{align*}
\C_a \otimes \C_b
&= \C_{a+b}
\overset{\sigma(a,b)}{\xrightarrow{\hspace*{1.27cm}}}
\C_{a+b} = \C_a \otimes \C_b,\\
\C_a \otimes \X
&= \X
\overset{q(a)^{-1}}{\xrightarrow{\hspace*{2.32cm}}}
\X = \X \otimes \C_a,\\
\X \otimes \C_a
&= \X
\overset{q(a)^{-1}}{\xrightarrow{\hspace*{2.32cm}}}
\X = \C_{-a} \otimes \X,\\
\X \otimes \X
&= \bigoplus_{t \in \Gamma} \C_t
\overset{\alpha \, q(t)}{\xrightarrow{\hspace*{0.86cm}}}
\bigoplus_{t \in \Gamma} \C_t = \X \otimes \X
\end{align*}
}%
for $a,b \in \Gamma$.
\smallskip
\item The $\Z_2$-ribbon structures for a given $(q,\alpha)$ are in correspondence with a choice of $\beta = \pm \alpha^{-1}$. Then the ribbon twist is given by
\begin{align*}
\C_a
&\overset{q(a)^{2}}{\xrightarrow{\hspace*{1cm}}}
\C_a,\\
\X
&\overset{\beta}{\xrightarrow{\hspace*{1cm}}}
\X
\end{align*}
for $a\in\Gamma$.
\end{enumerate}
We denote by $\TY(\Gamma,\sigma,\eps\,|\,q,\alpha,\beta)$ the $\Z_2$-crossed ribbon Tambara-Yamagami category associated with these data.
\end{thm}

\begin{rem}\label{rem_TYext}
$\TY(\Gamma,\sigma,\eps\,|\,q,\alpha,\beta)$ is a braided $\Z_2$-crossed extension of the braided tensor category $\smash{\Vect_\Gamma^{q^2}}\cong\smash{\Vect_{\Gamma}^{\sigma,1}}$ (with $\sigma$ symmetric) discussed in \autoref{ex_braidedVect}, with ribbon structure given by $\theta_{\C_a}=q(a)^2$. (But note that the quadratic form $q^2$ on $\Gamma$ may be degenerate if $|\Gamma|$ is even.)
\end{rem}

\begin{rem}\label{rem_alphaEquivalence}
The two choices of $\alpha$ yield equivalent $G$-crossed ribbon categories. Indeed, take $F=\id$ as an autoequivalence of the Tambara-Yamagami category as a linear category, endow $F$ with the structure of a tensor functor with the trivial tensor structure $F(X\otimes Y)\to F(X)\otimes F(Y)$ and then with the structure of a functor between categories with $G$-action with a \emph{nontrivial} constraint $F\circ g_*\to g_*\circ F$, namely
\begin{equation*}
\id_{\C_a}\colon \id(g_*(\C_a)) \longrightarrow g_*(\id(\C_a))
\quad\text{and}\quad
-\id_\X\colon \id(g_*(\X)) \longrightarrow g_*(\id(\X)).
\end{equation*}
It can be verified that this defines a braided $G$-crossed functor between the Tambara-Yamagami categories with the two different choices of $\alpha$. That is, these choices lead to equivalent braided $G$-crossed tensor categories and consequently also to equivalent $G$-equivariantisations, where the equivalence interchanges equivariant objects for eigenvalues $\pm1$.

Take as the simplest example $\Gamma = \{0\}$ and $\eps = 1$, i.e.\ $\smash{\Vect_{\Z_2}}$. Then the two choices $\alpha = \pm 1$ give $\Z_2$-equivariantisations with simple objects $1^\pm, \X^\pm$ with braiding $c_{\X^{\pm},\X^{\pm'}} = \pm$ or $\mp$, respectively. However, these are equivalent braided tensor categories (namely the Drinfeld centre of $\smash{\Vect_{\Z_2}}$) by sending $\X^\pm\mapsto\X^\mp$.
\end{rem}

\begin{rem}\label{rem_TY}
A symmetric, nondegenerate bimultiplicative form $\sigma\colon \Gamma\times \Gamma\to\C^\times$ admits $|\Gamma/2\Gamma|$ many choices of quadratic form $q\colon \Gamma\to\C^\times$ with associated bimultiplicative form $\sigma$.
\end{rem}

\begin{rem}\label{rem_gauss}
The Gauss sum in \autoref{thm_TYribbon} takes the simple form
\begin{equation*}
G(\Gamma,q^{-1})\coloneqq|\Gamma|^{-1/2}\sum_{a\in \Gamma} q(a)^{-1}=\e(\sign(\Gamma,q^{-1})/8)=\e(-\sign(\Gamma,q)/8)
\end{equation*}
where $\sign$ denotes the signature (a number in $\Z_8$) of a discriminant form. Hence, $\alpha^2=\eps \,\e(-\sign(\Gamma,q)/8)$. This shows that $\alpha$ and $\beta$ are $16$-th roots of unity so that $|\alpha|=|\beta|=1$. Moreover, $\alpha\beta\in\{\pm1\}$ by definition.
\end{rem}

\begin{prop}\label{cor_TYpseudounitary}
The quantum dimensions of the above $\Z_2$-crossed ribbon fusion category $\TY(\Gamma,\sigma,\eps\,|\,q,\alpha,\beta)$ coincide with the Frobenius-Perron dimensions if and only if $\alpha \beta = \eps$.
\end{prop}
\begin{proof}
Recall that the neutral component of $\TY(\Gamma,\sigma,\eps\,|\,q,\alpha,\beta)$ corresponds to $\smash{\Vect_\Gamma}$ with braiding given by $\sigma$. The canonical ribbon is given for $a\in\Gamma$ by
\begin{equation*}
\sigma(a,a) = q(a^2) / q(a)^2 = q^4(a) / q(a)^2 = q^2(a)=\theta_{\C_a}.
\end{equation*}
Therefore, the quantum dimension of the neutral component of $\TY(\Gamma,\sigma,\eps\,|\,q,\alpha,\beta)$ coincides with the Frobenius-Perron dimension.

Now, the quantum dimension of $\X$ is
\begin{align*}
\dim(\X)
&= \eval_\X \circ c_{\X,\X^*} \circ (\theta_\X \otimes \id_{\X^*}) \circ \coeval_\X \\
&= \eps |\Gamma|^{1/2} \cdot (\alpha \, q(0)) \beta \id_\X \\
&= \eps (\alpha \beta) |\Gamma|^{1/2} \id_\X.
\end{align*}
Since the Frobenius-Perron dimension of $\X$ is $|\Gamma|^{1/2}$, to have equality between the quantum and Frobenius-Perron dimensions is equivalent to $\alpha \beta = \eps$.
\end{proof}
By the definition of $\alpha$ and $\beta$, the condition $\alpha\beta=\eps$ is equivalent to $\beta/\alpha=\e(\sign(\Gamma,q)/8)$.

\medskip

The special case of $\Gamma$ of odd order will be of particular interest to this text. Indeed, note that in the \voa{} context it is natural to demand that the equivariantisation of a braided $G$-crossed tensor category be modular (see \autoref{sec_VOAIntro}).
\begin{prop}[\cite{GNN09}, Proposition~5.1]\label{prop:TY_modular}
The equivariantisation of the braided $\Z_2$-crossed tensor category $\TY(\Gamma,\sigma,\eps\,|\,q,\alpha,\beta)\sslash\Z_2$ has a nondegenerate braiding and hence is a modular tensor category if and only if $|\Gamma|$ is odd.
\end{prop}
Hence, in the following, we assume that $\Gamma$ has odd order. Then, as mentioned in \autoref{ex_braidedVect}, given any nondegenerate quadratic form $Q$ on $\Gamma$, the equivalence class $\smash{\Vect_\Gamma^Q}$ has a distinguished representative $\smash{\Vect_\Gamma^{\sigma,1}}$, i.e.\ there exists a distinguished representing abelian $3$-cocycle $(\sigma,\omega)$ with trivial associator $\omega=1$ and braiding given by the symmetric bimultiplicative form $\smash{\sigma=B_Q^{1/2}}$.

That is, by \autoref{rem_TYext}, the Tambara-Yamagami categories $\TY(\Gamma,\sigma,\eps\,|\,q,\alpha,\beta)$ are braided $\Z_2$-crossed extensions of $\smash{\Vect_\Gamma^Q}$ (concretely represented by $\smash{\Vect_\Gamma^{\sigma,1}}$). We note that, since $\Gamma$ is odd, there is a unique choice of quadratic form $q$ with associated bimultiplicative form $B_q=\sigma$ (see \autoref{rem_TY}), and this quadratic form must coincide with the unique square root of $Q$, i.e.\ $q=Q^{1/2}$.

Since $|\Gamma|$ is odd, $\smash{\Vect_\Gamma^Q}$ has a unique ribbon structure (cf.\ \autoref*{GLM2sec_GRibbon} in \cite{GLM24b}). The ribbon twist is given by $\theta_{\C_a}=Q(a)=q(a)^2$ for $a\in\Gamma$, and with this ribbon structure the quantum dimensions agree with the Frobenius-Perron dimensions (in particular, $\smash{\Vect_\Gamma^Q}$ is pseudo-unitary). Hence, the $\Z_2$-crossed ribbon structure on $\TY(\Gamma,\sigma,\eps\,|\,q,\alpha,\beta)$ is an extension of the ribbon structure on $\smash{\Vect_\Gamma^Q}$, and yields coinciding quantum and Frobenius-Perron dimensions if and only if $\beta=\eps/\alpha$ by \autoref{cor_TYpseudounitary}.

We summarise this in the following statement:
\begin{cor}\label{cor_TY}
For $|\Gamma|$ odd, the modular tensor category $\smash{\Vect_\Gamma^Q}$ with strict $\Z_2$-action $g_*\C_a=\C_{-a}$ has two braided $\Z_2$-crossed extensions with positive quantum dimensions, namely, for $\eps=\pm 1$,
\begin{equation*}
\Vect_\Gamma^Q[\Z_2,\eps] \coloneqq \TY(\Gamma,B_Q^{1/2},\eps \mid Q^{1/2}, \alpha,\eps/\alpha).
\end{equation*}
Here, $\alpha$ is one of the solutions of $\alpha^2=\eps\,G(\Gamma,q^{-1})=\eps\,\e(-\sign(\Gamma,q)/8)$, both giving equivalent extensions. The equivariantisations $\Vect_\Gamma^Q[\Z_2,\eps]\sslash\Z_2$ are again modular.
\end{cor}
By \autoref{thm_ourENO_braided} these are the unique two braided $\Z_2$-crossed extensions of $\smash{\Vect_\Gamma^Q}$. The equivariantisation $\smash{\Vect_\Gamma^Q}[\Z_2,\eps]\sslash\Z_2$ was computed in \cite{GNN09} and will appear as a special case of the more general result in \autoref*{GLM2sec_equiv} of \cite{GLM24b}.

\begin{rem}\label{rem_Gauss}
If $|\Gamma|$ is odd, the Gauss sum $G(\Gamma,q^{-1})$ over $q^{-1}=Q^{-1/2}$ can be expressed in a simple way in terms of the signature of the discriminant form $\sign(\Gamma,Q)$ (see, e.g., \cite{Sch09}):
\begin{align*}
G(\Gamma,q^{-1})&=|\Gamma|^{-1/2}\sum_{a\in \Gamma}q(a)^{-1}
=\e(\sign(\Gamma,q^{-1})/8)\\
&=\left(\frac{2}{|\Gamma|}\right)G(\Gamma,Q^{-1})=\left(\frac{2}{|\Gamma|}\right)\e(-\sign(\Gamma,Q)/8),
\end{align*}
where $(\frac{\cdot}{\cdot})$ is the Kronecker symbol, and for $n$ odd $(\frac{2}{n})=(\frac{n}{2})=(-1)^{(n^2-1)/8}$.

If $(\Gamma,Q)$ is obtained as the discriminant form $L^*\!/L$ of an even lattice $L$, then the signature of $(\Gamma,Q)$ is the signature of the underlying vector space $L\otimes_\Z\R$ modulo~$8$ (see \autoref{sec_LatticeVOA}). With the above formula, the ribbon twist eigenvalue on $\X$ (up to an irrelevant sign) reads
\begin{equation*}
\theta_\X=\beta=\eps/\alpha=\pm \left(\eps\left(\frac{2}{|\Gamma|}\right)\e(\sign(\Gamma,Q)/8)\right)^{1/2}.
\end{equation*}
\end{rem}

\begin{ex}
Let $\Gamma=\Z_n$ with odd $n$ and an arbitrary nondegenerate quadratic form $Q_k(a)=\e(k\,a^2/n)$ for some $k\in\Z$ with $(k,n)=1$. Then
\begin{equation*}
\e(\sign(\Gamma,Q_k)/8)=\left(\frac{k}{n}\right)\epsilon_n,\quad
\epsilon_n=
\begin{cases}
1 &\text{if }n=1\pmod{4},\\
\i &\text{if }n=3\pmod{4},
\end{cases}
\end{equation*}
where $\epsilon_n\coloneqq\e(\sign(\Gamma,Q_1)/8)$ is the signature of the standard discriminant form $(\Gamma,Q_1)$ and
\begin{equation*}
G(\Gamma,Q_k^{-1/2})^{-1}=\left(\frac{2}{n}\right)\e(\sign(\Gamma,Q_1)/8)
=\begin{cases}
1 &\text{if }n=1\pmod{8},\\
-\i &\text{if }n=3\pmod{8},\\
-1 &\text{if }n=5\pmod{8},\\
\i &\text{if }n=7\pmod{8},
\end{cases}
\end{equation*}
which we notice is precisely $\i^{-(n-1)/2}=\e(-(n-1)/8)$.
\end{ex}


\section{Condensations by Commutative Algebras}

In this section we review the condensation of a braided tensor category $\cB$ by a commutative algebra $A\in\cB$ to a new braided tensor category $\cB_A^{\loc}$ of local $A$-modules, $\cB\leadsto\cB_A^{\loc}$. As an example, we discuss the condensation of a quadratic space or discriminant form by an isotropic subgroup. We also discuss how automorphisms of $A$ lift to automorphisms of the braided tensor category $\cB_A^{\loc}$.

A particular instance of this example will be the categories of modules associated with lattice \voa{}s. This is in principle well-known (see \cite{DL93}), but we compute all structures of the resulting braided tensor category in detail as these computations serve as template for new results treated below. Moreover, we are not aware of a treatment including these details in the literature.


\subsection{Commutative Algebra Condensations}\label{sec_CurrentExtension}

We briefly recall certain notions of algebras in (braided) tensor categories and how the corresponding categories of modules yield a construction, which we call condensation, to obtain new (braided) tensor categories \cite{Par95,KO02,FFRS06,DMNO13} (and \cite{LW23,SY24} in the nonsemisimple setting). They correspond to conformal extensions of \voa{}s \cite{HKL15,CKM17,CMSY24} and appear in the context of anyon condensation \cite{FSV13,Kon14}.

\medskip

In a tensor category $\cB$, there is the notion of an algebra $A$. In a braided tensor category $\cB$, there is the notion of a commutative algebra $A$. There is also the concept of modules over an algebra $A$ in a tensor category $\cB$, which are objects $M$ in $\cB$ with an action morphism $\rho\colon A \otimes M \to M$ in $\cB$. This produces a category~$\cB_A$. If $\cB$ is braided and $A$ commutative, then $\cB_A$ carries the structure of a tensor category with a tensor product $M \otimes_A N$ defined as the coequaliser of the tensor product $M \otimes N$ under the morphism
\begin{equation*}
\begin{tikzcd}
(M \otimes A)\otimes N
\ar{d}{c_{A,M}^{-1}\otimes \id}
\ar{r}{\text{ass.}}
&M \otimes (A\otimes N)
\ar{d}{\rho}
\\
(A \otimes M)\otimes N
\ar{r}{\rho\otimes \id}
&
M\otimes N
\end{tikzcd}
\end{equation*}
A \emph{local $A$-module} $M$ is an $A$-module such that
\begin{equation*}
\rho=\rho \circ c_{M,A}\circ c_{A,M},
\end{equation*}
and we denote the category of local $A$-modules in $\cB$ by $\cB_A^\loc$. Note that for local modules the choice between $c_{M,A}$ and $c_{A,M}^{-1}$ in the definition of $M\otimes_A N$ disappears. As it turns out, this definition produces again a braided tensor category $\cB_A^\loc$, which we call the \emph{condensation} of $\cB$ by $A$. We also write $\cB\leadsto\cB_A^{\loc}$.

Note that, e.g., Theorem~B in \cite{SY24} gives precise conditions on $A$ under which $\cB_A^{\loc}$ is finite, ribbon or modular. In this article, we only work with simple-current extensions; hence if $\cB$ is ribbon and the twist $\theta_A=\id_A$, then $\cB_A^{\loc}$ is also ribbon (see Lemma~6.1 in \cite{SY24} and \autoref{rem_simplecurrentribbon} below).


\subsection{Condensations of Discriminant Forms}\label{sec_DiscForm}

In the following, as a main example, we study the condensation of $\cB = \smash{\Vect_\DiscGamma^\DiscQ}$ by an algebra corresponding to an isotropic subgroup $I \subset \DiscGamma$ of the discriminant form $(\DiscGamma, \DiscQ)$.

\medskip

Let $(\DiscGamma,\DiscQ)$ be an abelian group together with a nondegenerate quadratic form $\DiscQ\colon\DiscGamma\to\C^\times$ and let $\DiscB(v,w)=\DiscQ(v+w)\DiscQ(v)^{-1}\DiscQ(w)^{-1}$ be the associated nondegenerate bimultiplicative form. As in \autoref{ex_braidedVect}, let $\cB=\smash{\Vect_\DiscGamma^\DiscQ}$ be the braided tensor category associated with this data, and choose an explicit abelian $3$-cocycle $(\DiscS,\DiscO)$ on $\DiscGamma$ associated with $\DiscQ$ in the sense that $\DiscS(v,v)=\DiscQ(v)$ and $\DiscS(v,w)\DiscS(w,v)=\DiscB(v,w)$ holds for $v,w\in \DiscGamma$. Then $\DiscS(v,w)$ describes the braiding and $\DiscO(v,w,z)$ the associator in $\cB$. In principle, $\smash{\DiscGamma}$ can be an infinite group (see \cite{CMY22,AR18} and \autoref{sec_latticeDiscriminantForm}), but we assume for now that $\DiscGamma$ is finite, in which case $(\DiscGamma,\DiscQ)$ is a discriminant form (see \autoref{sec:prel}).

\medskip

Let $I$ be a subgroup of $\DiscGamma$. For any $2$-cochain $\epsilon\colon I \times I \to \C^\times$ with $\d\epsilon = \DiscO$, we define an algebra $A$ in $\cB$ by
\begin{equation*}
A \coloneqq \C_\epsilon[I] = \bigoplus_{i \in I} \C_i,
\end{equation*}
with basis elements $e_i \in \C_i$ for $i \in I$ and multiplication given by
\begin{equation*}
e_i \cdot e_j \coloneqq \epsilon(i,j)\,e_{i+j}
\end{equation*}
for $i, j \in I$. The algebra $A$ is associative if $\d\epsilon = \DiscO$, as assumed. The isomorphism class of $A$ depends on $\epsilon$ only up to a coboundary $\d\gamma$ for some $\gamma\colon I \to \C^\times$. The map $\beta(i, j) \coloneqq \DiscS(i, j)\epsilon(i, j)\epsilon(j, i)^{-1}$, $i,j\in I$, is bimultiplicative and satisfies $\DiscQ(i) = \beta(i, i)$. This map $\beta$ depends only on the isomorphism class of $A$. The algebra $A$ is commutative if and only if $\beta(i, j) = 1$ for all $i, j\in I$, or equivalently, if $\DiscQ(i) = 1$ for all $i\in I$, meaning the subgroup $I$ is isotropic. Note that if $\DiscGamma$ is finite and of odd order, we may choose $\DiscO = 1$ and $\epsilon = 1$.

In the remainder of this section, we aim to verify the following well-known statement. For simplicity, we restrict our exposition to the case of a strict base category~$\cB$, i.e.\ $\DiscO = 1$, and consequently $\epsilon$ a $2$-cocycle because of $\d\epsilon = \DiscO$. This restriction suffices for the examples discussed later, though the main assertions do not depend on this simplification.

Recall from \autoref{sec:prel} the definition $I^\perp=\{v \in \DiscGamma\mid\DiscB(v,i) = 1\text{ for all }i \in I\}$ and that $\Gamma\coloneqq I^\perp/I$ is again a discriminant form, with the quadratic form $Q$ given by $\DiscQ$ restricted to $I^\perp$ and factored through $I^\perp/I$.

\begin{prop}\label{prop_quadext}
The tensor category of $A$-modules and the braided tensor category of local $A$-modules are given by
\begin{align*}
\cB_A &\cong \Vect_{\DiscGamma/I},\\
\cB_A^{\loc} &\cong \Vect_{\Gamma}^{Q},
\end{align*}
respectively. If we denote the cosets in $\DiscGamma/I$ and $I^\perp/I$ by $a\in\DiscGamma/I$, choose an explicit set of representatives $\hat{a}\in\DiscGamma$ and define the corresponding $2$-cocycle $u\colon\DiscGamma/I\times \DiscGamma/I\to I$ by $u(a, b) = \hat{a} + \hat{b} - \widehat{a+b}$, then the proof produces for $\DiscO = 1$ the following explicit $3$-cocycle $\omega$ for $\cB_A$ and abelian $3$-cocycle $(\sigma,\omega)$ for $\cB_A^{\loc}$:
\begin{equation*}
\sigma(a, b) = \DiscS(\hat{a}, \hat{b}),
\quad
\omega(a, b, c) = \DiscS(\hat{a}, u(b, c)) \frac{\epsilon(u(b, c), u(a, b+c))}{\epsilon(u(a, b), u(a+b, c))}.
\end{equation*}
\end{prop}

\begin{rem}\label{rem_simplecurrentribbon}
The condition $\DiscQ(i) = 1$ additionally ensures that the (standard choice of) ribbon twist $\theta_{\C_v} = \DiscQ(v)$ acts trivially on $A$, and hence lifts to a ribbon twist on $\cB_A$.
\end{rem}

\begin{proof}
We construct the structures step by step and verify the assertions of \autoref{prop_quadext}. We obtain $A$-modules in $\cB_A$ as induced modules of $\C_v$ for any $v\in \DiscGamma$. More explicitly,
\begin{equation*}
\Ind_A\C_v = A\otimes \C_v= \bigoplus_{i+v\in v+I} \C_{i+v}
\end{equation*}
as objects in $\cB$, with basis denoted by $e_{i+v}\in \C_{i+v}$ and with $A$-action
\begin{equation*}
\rho\colon e_i\otimes e_{j+v}\longmapsto
\epsilon(i,j)\, e_{(i+j)+v}.
\end{equation*}
Note that a different choice of coset representatives $v+I=v'+I$ leads to a different action on the same objects, but there is a nontrivial $A$-module isomorphism
\begin{align*}
\Ind_A\C_v&\cong \Ind_A\C_{v'},\\
e_{i+v}&\longmapsto \epsilon(i,v-v') e_{(i+(v-v'))}.
\end{align*}
We shall denote the isomorphism class of this simple $A$-module from now on by $\C_{a}$ for cosets $a=v+I\in \DiscGamma/I$.

It is not difficult to see that these modules are simple $A$-modules and that these are all simple modules. Hence, as abelian categories
\begin{equation*}
\cB_A\cong \Vect_{\DiscGamma/I}.
\end{equation*}

The local modules $\Ind_A\C_v$ are those for which for every $i,j\in I$
\begin{equation*}
\epsilon(i,j)=\DiscS(i,j+v)\DiscS(j+v,i)\,\epsilon(i,j),
\end{equation*}
which is equivalent to the condition $\DiscB(i,v)=1$ for all $i\in I$, which of course only depends on the coset $v+I$. Hence, as abelian categories
\begin{equation*}
\cB_A^{\loc}\cong \Vect_{I^\perp/I}.
\end{equation*}

\smallskip

The tensor product $\otimes_A$ of two $A$-modules $(\C_{a},\rho_a)$ and $(\C_{b},\rho_b)$ was defined as the coequaliser of the following two maps
\begin{equation*}
\begin{tikzcd}
\bigoplus_{\substack{i+v\in a+I,\\ j+w\in w+I}} &
(\C_{i+v} \otimes \C_x)\otimes \C_{j+w}
\ar{d}{c_{A,M}^{-1}\otimes \id}
\ar{r}{\text{ass.}}
&\C_{i+v} \otimes (\C_x\otimes \C_{j+w})
\ar{d}{\id\otimes \rho_b}
\\
&
(\C_x \otimes \C_{i+v})\otimes \C_{j+w}
\ar{r}{\rho_a\otimes \id}
&
\C_{i+v} \otimes \C_{j+w}
\end{tikzcd}
\end{equation*}
for $v+I,w+I\in \DiscGamma/I$ and for all $x\in I$. Setting equal the images of the generators $(e_{i+v}\otimes e_x)\otimes e_{j+w}$ under the upper and the lower morphism gives the coequaliser as equivalence relation
\begin{align*}
\epsilon(x,j)
\,(e_{i+v}\otimes e_{(x+j)+w})
\sim \DiscS(i+v,x)\epsilon(x,i)
\,(e_{(i+x)+v}\otimes e_{j+w})
\end{align*}
such that the tensor product over $A$ is
\begin{equation*}
\C_a\otimes_A \C_b =\C_a\otimes \C_b /\sim.
\end{equation*}
The $A$-action on this tensor product is given by the action on the first factor as
\begin{equation*}
e_x\otimes (e_{i+v}\otimes e_{j+w})
\longmapsto \epsilon(x,i)\,(e_{(x+i)+v}\otimes e_{j+w}).
\end{equation*}
Clearly, the set of equivalence classes of this equivalence relation is as object $\bigoplus_{k\in I} \C_{k+v+w}$, and on the set of equivalence class representatives $e_{k+v}\otimes e_w$ the action on the first factor agrees with the action on $\Ind_A\C_{v+w}$. Thus, there is an explicit isomorphism of $A$-modules \begin{align*}
\iota_{v,w}\colon\C_{v+I}\otimes_A \C_{w+I} &\overset{\sim}{\longrightarrow} \C_{(v+w)+I},\\
e_{i+v}\otimes e_{j+w}&\longmapsto \DiscS(i+v,j)\epsilon(j,i)e_{(i+j)+(v+w)}
\end{align*}
for cosets $a=v+I$ and $b=w+I$ in $\DiscGamma/I$. Note, however, that in general this module isomorphism is \emph{not} compatible with the choice of $A$-module isomorphism above changing representatives $\C_{v+I}\cong \C_{v'+I}$ and $\C_{w+I}\cong \C_{w'+I}$, not even for $\epsilon=1$. Similarly, the braiding of $\C_{v+I}\otimes_A \C_{w+I}$ written as a morphism on $\C_{(v+w)+I}$ depends on these choices.

Hence, we need to introduce a preferred system of coset representatives $\hat{a}\in \DiscGamma$ for the $I$-cosets $a\in\DiscGamma/I$. We assume $\hat{0}=0$. Define the corresponding normalised $2$-cocycle $u\colon\DiscGamma/I\times \DiscGamma/I\to I$ as
\begin{equation*}
u(a,b)=\hat{a}+\hat{b}-\widehat{a+b},
\end{equation*}
which corresponds to the extension of abelian groups $I\to \DiscGamma\to \DiscGamma/I$ and the choice of coset representatives. Further, denote the composition of $\iota$ with the corresponding change of representative from $\hat{a}+\hat{b}$ to $\widehat{a+b}$ by
\begin{align*}
\hat{\iota}_{\hat{a},\hat{b}}\colon
\Ind_A\C_{\hat{a}}\otimes\Ind_A\C_{\hat{b}}
&\longrightarrow
\Ind_A\C_{\widehat{a+b}},\\
e_{i+\hat{a}}\otimes e_{j+\hat{b}}
&\longmapsto
\DiscS(i+\hat{a},j)\epsilon(j,i)\epsilon(i+j,u(a,b))\;
e_{(u(a,b)+i+j)+\widehat{a+b}}.
\end{align*}

Then, the associator in the tensor category $\cB_A$ is given by the dashed arrow in
\begin{equation*}
\begin{tikzcd}
\bigl(\C_{i+\hat{a}} \otimes \C_{j+\hat{b}}\bigr)\otimes \C_{k+\hat{c}}
\ar{d}{\hat{\iota}_{\hat{a},\hat{b}} \otimes \id}
\ar{r}{\DiscO=1}
&
\C_{i+\hat{a}} \otimes \bigl(\C_{j+\hat{b}}\otimes \C_{k+\hat{c}} \bigr)
\ar{d}{\id\otimes \hat{\iota}_{\hat{a},\widehat{b+c}} \otimes \id} \\
\C_{(u(a,b)+i+j)+\widehat{a+b}} \otimes \C_{c+k}
\ar{d}{\hat{\iota}_{\widehat{a+b},c} \otimes \id}
&
\C_{i+\hat{a}} \otimes \C_{(u(b,c)+j+k)+\widehat{b+c}}
\ar{d}{\id\otimes \hat{\iota}_{\hat{b},\hat{c}} \otimes \id} \\
\C_{(u(a,b)+u(a+b,c)+i+j+k)+\widehat{a+b+c}}
\arrow[dashed, r]{}{}
&
\C_{(u(a,b+c)+u(b,c)+i+j+k)+\widehat{a+b+c}}
\end{tikzcd}
\end{equation*}
in any degree $i+\hat{a},j+\hat{b},k+\hat{c}\in \DiscGamma$ in cosets $a,b,c\in \DiscGamma/I$, and the result should be a morphism on the full module, so it should turn out to be independent of $i,j,k$. Following the basis element $e_{(i+j+k)+\widehat{a+b+c}}$ through these isomorphisms, and using that $\DiscS(v,w)$ is bimultiplicative because $\DiscO(v,w,z)=1$, we compute:
{\allowdisplaybreaks
\begin{align*}
&e_{(u(a,b)+u(a+b,c)+i+j+k)+\widehat{a+b+c}}\\
&\mapsto \bigl(\DiscS(\hat{a}+\hat{b}+i+j,k)\epsilon(k,u(a,b)+i+j)\epsilon(u(a,b)+i+j+k,u(a+b,c))\bigr)^{-1}\\
&\quad\cdot\bigl(\DiscS(\hat{a}+i,j)\epsilon(j,i)\epsilon(i+j,u(a,b))\bigr)^{-1}\\
&\quad\cdot \bigl(\DiscS(\hat{b}+j,k)\epsilon(k,j)\epsilon(j+k,u(b,c))\bigr)\\
&\quad \cdot \bigl(\DiscS(\hat{a}+i,u(b,c)+j+k)\epsilon(u(b,c)+j+k,i)\epsilon(u(b,c)+i+j+k,u(a,b+c))\bigr)\\
&\quad\cdot e_{(u(a,b+c)+u(b,c)+i+j+k)+\widehat{a+b+c}} \\
&=\DiscS(\hat{a},u(b,c))
\frac{\epsilon(u(b,c),u(a,b+c))}{\epsilon(u(a,b),u(a+b,c))}
 e_{(u(a,b+c)+u(b,c)+i+j+k)+\widehat{a+b+c}}.
\end{align*}
}%
The last step holds for $i=j=k=0$, and we may verify the expected independence of $i$, $j$ and $k$ by cancelling all $\DiscS$ except $\DiscS(i,u(b,c))=\epsilon(i,u(b,c))\epsilon(u(b,c),i)^{-1}$ and then repeatedly applying the $2$-cocycle condition of $\epsilon$.

Similarly, the associated braiding is indicated by the dashed arrow in
\begin{equation*}
\begin{tikzcd}
\C_{i+\hat{a}} \otimes \C_{j+\hat{b}}
\ar{d}{\hat\iota_{\hat{a},\hat{b}}}
\ar{rr}{\DiscS(i+\hat{a},j+\hat{b})}
&&
\C_{j+\hat{b}} \otimes \C_{i+\hat{a}}
\ar{d}{\hat\iota_{\hat{b},\hat{a}}} \\
\C_{(u({a},{b})+i+j)+\widehat{a+b}}
\arrow[dashed, rr]{}{}
&&
\C_{(u({b},{a})+i+j)+\widehat{b+a}},
\end{tikzcd}
\end{equation*}
in any degree $i+\hat{a},j+\hat{b}\in I^ \perp$ in cosets $a,b\in I^\perp/I$. Following the basis element and using in addition to the previous arguments $\DiscS(i,x)\DiscS(x,i)=\DiscB(i,x)=1$ for $x\in I^ \perp$ and $i\in I$, we compute
{\allowdisplaybreaks
\begin{align*}
&e_{(u(a,b)+i+j)+\widehat{a+b}}\\
&\mapsto
\bigl(\DiscS(\hat{a}+i,j)\epsilon(j,i)\epsilon(i+j,u(a,b))\bigr)^{-1} (e_{\hat{a}+i}\otimes e_{\hat{b}+j}) \\
&\mapsto \bigl(\DiscS(\hat{a}+i,j)\epsilon(j,i)\epsilon(i+j,u(a,b))\bigr)^{-1} \DiscS(\hat{a}+i,\hat{b}+j) (e_{\hat{b}+j}\otimes e_{\hat{a}+i})\\
&\mapsto \bigl(\DiscS(\hat{a}+i,j)\epsilon(j,i)\epsilon(i+j,u(a,b))\bigr)^{-1} \DiscS(\hat{a}+i,\hat{b}+j)\\
&\quad\cdot \bigl(\DiscS(\hat{b}+j,i)\epsilon(i,j)\epsilon(j+i,u(b,a))\bigr)\,e_{\widehat{b+a}+(u(b,a)+j+i)}\\
&=\DiscS(\hat{a},\hat{b})\;e_{(u(b,a)+j+i)+\widehat{b+a}}.
\end{align*}
}%
This concludes the analysis of the tensor category $\cB_A$ and the braided tensor category $\cB_A^{\loc}$, in particular of a new abelian $3$-cocycle on $I^\perp/I$.
\end{proof}


\subsection{Example: Infinite Condensation Associated with Lattice}\label{sec_latticeDiscriminantForm}

We discuss an interesting application of the condensation of an infinite pointed braided fusion category by an infinite commutative algebra object. This parallels the construction of a lattice \voa{} together with its category of modules as extension of a Heisenberg \voa{} (cf.\ \cite{DL93} and \autoref{sec_LatticeVOA}).

\medskip

Let $L$ be an even lattice of rank $d=\rk(L)$, with the nondegenerate, symmetric bilinear form $\langle\cdot,\cdot\rangle\colon L\times L\to\Z$. The lattice $L$ can be embedded into its ambient vector space $\DiscGamma=\R^d=L\otimes_\Z\R$, which is equipped with the corresponding extension of $\langle\cdot,\cdot\rangle$. Let $\cB=\smash{\Vect_{\R^d}^{\DiscQ}}$ be the infinite braided tensor category with simple objects $\C_v$, $v\in\R^d$, and quadratic form
\begin{equation*}
\DiscQ(v)=\e(\langle v,v\rangle/2)
\end{equation*}
for $v\in\R^d$, and with an explicit representing abelian $3$-cocycle
\begin{align*}
\DiscS(u,v)&=\e(\langle u,v\rangle/2),\\
\DiscO(u,v,w)&=1
\end{align*}
for $u,v,w\in\R^d$.

We view $L\subset\R^d$ as a subgroup, and because we assumed the lattice to be even (i.e.\ $\langle v,v\rangle\in2\Z$ for all $v\in L$), this is an isotropic subgroup. Now, the construction discussed in this section produces a condensation by the infinite commutative algebra $A=\C_\epsilon[L]$ for a choice of $2$-cocycle $\epsilon\colon L\times L\to\C^\times$ (meaning $\d\epsilon = \DiscO=1$) whose skew form equals $\epsilon(u,v)\epsilon(v,u)^{-1}=\DiscS(u,v)=\e(\langle u,v\rangle/2)$ for all $u,v\in L$. However, some care must be taken because $A$ is infinite. One rigorous approach to this construction is to consider the tensor category of comodules over the group algebra $\C[\R^n]$ and the Hopf subalgebra $\C[L]$, and then apply the de-equivariantisation of Hopf algebras with the braided central structure defined by $\langle\cdot,\cdot\rangle$, as discussed in \cite{AGP14}. Another rigorous approach is to use the direct-limit completions discussed in \cite{CMY22}.

As a final result, we obtain, as expected, the pointed tensor category
\begin{equation*}
\cB_A = \Vect_{\R^d / L}
\end{equation*}
and the pointed braided tensor category
\begin{equation*}
\cB_A^{\loc} = \Vect_\Gamma^Q,
\end{equation*}
where we recall that dual lattice is $L^* = \{v \in \R^d \mid \langle v, w \rangle \in \Z \text{ for all } w \in L\}$ and the discriminant form $(\Gamma, Q) \coloneqq L^*\!/L$, with the nondegenerate quadratic form $Q$ on $\Gamma$ induced from $\DiscQ$ on $\DiscGamma = \R^d$ (see \autoref{sec:prel}).

We fix a choice $\hat{a}$ of representatives for the $L$-cosets $a$ in $\Gamma = L^*\!/L$. Then we set
\begin{equation*}
Q(a) = \DiscQ(\hat{a})
\end{equation*}
for all $a \in \Gamma$. We can also describe the representing abelian $3$-cocycle $(\sigma, \omega)$ on $\Gamma$ in terms of the abelian $3$-cocycle $(\DiscS, \DiscO)$ on $\DiscGamma = \R^d$.
To this end, define the following $2$-cocycle (depending on the choices of $\hat{a}$) with values in $L$:
\begin{equation*}
u(a, b) \coloneqq \hat{a} + \hat{b} - \widehat{a+b}
\end{equation*}
for $a, b \in \Gamma$. Then the formulae above yield the following well-known representing abelian $3$-cocycles on the discriminant form $\Gamma = L^*\!/L$ of an even lattice (cf.\ \cite{DL93}):
\begin{align*}
\sigma(a, b) &= \DiscS(\hat{a}, \hat{b}) = \e(\langle\hat{a}, \hat{b}\rangle / 2),\\
\omega(a, b, c) &= \DiscS(\hat{a}, u(b, c)) \frac{\epsilon(u(b, c), u(a, b+c))}{\epsilon(u(a, b), u(a+b, c))}
\end{align*}
for $a, b, c \in \Gamma$. The choices of the $2$-cocycle $\epsilon$ and of the $L$-coset representatives $\hat{a}$ are not essential, as they only affect the $2$-cocycle $u$ and the abelian $3$-cocycle $(\sigma, \omega)$ by a coboundary.


\subsection{Groups Acting on Discriminant Forms}\label{sec_DiscFormAut}

In the following, continuing in the setting of \autoref{sec_DiscForm}, we describe how automorphism groups of the commutative algebra $A$ lift to automorphism groups of the braided tensor category $\cB_A^{\loc}$.

\medskip

As in \autoref{sec_DiscForm}, let $\cB=\smash{\Vect_\DiscGamma^\DiscQ}$ be the braided tensor category associated with the discriminant form $(\DiscGamma,\DiscQ)$ and choose a representing abelian $3$-cocycle $(\DiscS,\DiscO)$. Let $I\subset\DiscGamma$ be an isotropic subgroup, that is $Q(i)=1$ for all $i\in I$. Consider a commutative algebra $A=\C_\epsilon[I]$ with basis $e_i$, $i\in I$, for $\epsilon(i,j)\epsilon(j,i)^{-1}=\DiscS(i,j)$. Then $\cB_A=\smash{\Vect_{\DiscGamma/I}}$ as tensor category and $\cB_A^\loc =\smash{\Vect_{\Gamma}^{Q}}$ as braided tensor category with $\Gamma=I^\perp/I$.

We now assume that $\DiscG$ is a finite group acting on $\DiscGamma$, with the action denoted by $v \mapsto \bar{g}.v$, and that this action preserves the quadratic form, i.e.\ $\DiscQ(\bar{g}.v) = \DiscQ(v)$ for $v\in\DiscGamma$ and $\bar{g}\in\DiscG$. For simplicity, we also assume that the action preserves a chosen representing abelian $3$-cocycle $(\DiscS, \DiscO)$, rather than merely preserving it up to a coboundary. This action induces a strict action of the group $\DiscG$ on the braided tensor category $\cB$, meaning it has both a trivial tensor structure $\tau^{\bar{g}}$ and a trivial composition structure $T_2^{\bar{g}, \bar{h}}$ (see \autoref{def_Gcrossed}). The image $\bar{g}_*A$ under the braided tensor functor is again a commutative algebra in $\cB$.

We would now like to assume that $A$ is fixed under the action of $\DiscG$. In the categorical context, this means specifying additional data: namely, for each $\bar{g} \in \DiscG$, we fix an algebra isomorphism $\eqi_{\bar{g}}\colon\bar{g}_* A \overset{\sim}{\longrightarrow} A$. However, in general, we cannot choose a set of isomorphisms $\eqi_{\bar{g}}$, $\bar{g}\in\DiscG$, that is multiplicative in $\bar{g}$. Instead, we introduce a central extension of groups
\begin{equation*}
1 \to C \to G \overset{-}{\to} \DiscG \to 1,
\end{equation*}
where $C$ is a subgroup of automorphisms of $A$ as an algebra and object in $\cB$. Then, we can choose a set of algebra isomorphisms $\eqi_g$, $g\in G$, that are multiplicative:
\begin{align*}
\eqi_g\colon \, \bar{g}_* A &\overset{\sim}{\longrightarrow} A,\\
\eqi_g \circ \bar{g}_*(\eqi_h) &= \eqi_{gh}
\end{align*}
for $g,h\in G$, or in other words, $(A, \eqi)$ is a commutative algebra in the equivariantisation:
\begin{equation*}
(A, \eqi) \in \cB \sslash G.
\end{equation*}

In the lattice context, it is common to refer to $G$ acting on $A$ in this way as a \emph{lift} of $\bar{G}$ acting on $\cB$. Note that the central extension $G$ also acts on $(\DiscGamma, \DiscQ)$ via its quotient to $\bar{G}$. Thus, in hindsight, we could begin with a nonfaithful action of $G$ and assume, without loss of generality, that $G = \bar{G}$.

Let us now be more explicit: for $g\in G$, a morphism $\bar{g}_*A \to A$ in $\cB$ is of the form
\begin{align*}
\eqi_g\colon e_{\bar{g}.i} &\longmapsto \eta_g(i)e_{\bar{g}.i}
\end{align*}
for some scalar function $\eta_g\colon I \to \C^\times$. The multiplicativity condition over the central extension reads $\eta_g \eta_h = \eta_{gh}$ for $g,h\in G$. The compatibility of each $\eqi_g$ with the algebra structure $\mu_A$ implies that the following diagram commutes:
\begin{equation*}
\begin{tikzcd}
\bar{g}_*\Bigl(\bigl(\bigoplus_{i\in I} \C_i\bigr) \otimes \bigl(\bigoplus_{j\in I} \C_j\bigr)\Bigr)
\ar{dd}{g_*(\mu_A)}
\ar{r}{\tau^{\bar{g}}}
&
\bar{g}_*\bigl(\bigoplus_{i\in I} \C_v\bigr) \otimes \bar{g}_*\bigl(\bigoplus_{j\in I} \C_w\bigr)
\ar{d}{\eqi_g\otimes \eqi_g}
\\
&
\bigl(\bigoplus_{i\in I} \C_{\bar{g}.i}\bigr) \otimes \bigl(\bigoplus_{j\in I} \C_{\bar{g}.j}\bigr)
\ar{d}{\mu_A}
\\
\bar{g}_*\bigl(\bigoplus_{i+j\in I} \C_{i+j}\bigr)
\ar{r}{\eqi_g}
&
\bigl(\bigoplus_{c\in I} \C_{\bar{g}.(i+j)}\bigr).
\end{tikzcd}
\end{equation*}
We assumed the strictness $\tau^{\bar{g}}=1$, so that this commuting diagram yields the condition
\begin{equation*}
\frac{\eta_g(i+j)}{\eta_g(i)\eta_g(j)}
=\frac{\epsilon(\bar{g}.i,\bar{g}.j)}{\epsilon(i,j)}.
\end{equation*}
Clearly, if we have two different solutions $\eta_g$ and $\eta_g'$, they differ by a group homomorphism, and it is sufficient to consider $\epsilon$ and $\eta_g$ with values in $\{\pm 1\}$.

We argue that there exists a solution $\eta_g$ to this equation: since $\epsilon(i, j)$ and $\epsilon(\bar{g}.i, \bar{g}.j)$ are $2$-cochains with coboundaries $\DiscO(i, j, k)$ and $\DiscO(\bar{g}.i, \bar{g}.j, \bar{g}.k)$, which agree by our simplifying assumption (otherwise it would have to be compensated by $\tau^{\bar{g}}$), the quotient $\epsilon(\bar{g}.i, \bar{g}.j)/\epsilon(i, j)$ is a $2$-cocycle. Also, its skew form $\DiscS(\bar{g}.i, \bar{g}.j)/\DiscS(i, j)$ is trivial, again by our simplifying assumption; so it is a $1$-coboundary. Hence, there is a $1$-cochain $\eta_g\colon I \to \C^\times$ with $\d\eta_g = \epsilon(\bar{g}.i, \bar{g}.j)/\epsilon(i, j)$.

\begin{prop}\label{prop_ActionOnCondensation}
Let $\DiscG$ be a finite group acting on $(\DiscGamma,\DiscQ)$ and additionally assume for simplicity that $\DiscG$ preserves the chosen abelian $3$-cocycle $(\DiscS,\DiscO)$, so that it induces a strict action on $\smash{\Vect_{\DiscGamma}^{\DiscQ}}$ as braided tensor equivalences. Let $A=\C_\epsilon[I]$ be a commutative algebra built from an isotropic subgroup $I\subset \DiscGamma$ with condensation $\cB_A^{\loc}=\smash{\Vect^Q_\Gamma}$ for $\Gamma=I^\perp/I$ as in \autoref{sec_DiscForm}. Let $\eqi_g\colon \bar{g}_*A\to A$, $g\in G$, be a $G$-equivariant structure for a central extension $G$ of $\bar{G}$ described by $\eta_g$, as discussed in the previous paragraph. Then this data induces an action of $G$ on the tensor category $\cB_A=\smash{\Vect_{\DiscGamma/I}}$ and on the braided tensor category $\cB_A^{\loc}=\smash{\Vect^Q_\Gamma}$ for $\Gamma=I^\perp/I$. Explicitly, the action is given for $g\in G$ on objects by
\begin{equation*}
g_{A*}\colon \C_a \longmapsto \C_{\bar{g}.a}
\end{equation*}
for any coset $a\in \DiscGamma/I$ or $a\in I^\perp/I$, respectively.

If we assume $\epsilon=1$ and $\eta_g=1$ for all $g\in G$, then $G=\DiscG$, and if we continue to assume $\DiscO=1$ as in \autoref{sec_DiscForm}, then the tensor structure on $g_{A*}$ is given by
\begin{equation*}
\tau^{g_A}_{a,b}\colon\; g_{A*}(\C_a \otimes_A \C_b)
\xrightarrow{\DiscS(\hat{a},\bar{g}.\hat{b}-\widehat{\bar{g}.b})}
g_{A*}(\C_a) \otimes_A g_{A*}(\C_b).
\end{equation*}
\end{prop}
\begin{proof}
It is apparent that a (braided) tensor functor $\bar{g}_*\colon \cB\to \cB$ for $g\in G$ maps a (commutative) algebra $A$ to a (commutative) algebra $\bar{g}_*A$ and by functoriality also induces an equivalence of tensor categories $\cB_{A}\to \cB_{\bar{g}_*A}$, namely by sending any $A$-module $(M,\rho)$ to the $\bar{g}_*A$-module $(\bar{g}_*M,\bar{g}_*\rho)$. Moreover, for a given equivariant structure $\eqi_g\colon \bar{g}_*A\overset{\sim}{\to} A$ we can precompose the action of a $\bar{g}_*A$-module with $\eqi^{-1}_g$ to get an $A$-action, which altogether gives a tensor functor
\begin{equation*}
g_{A*}\colon \cB_A \overset{\bar{g}_*}{\longrightarrow}\cB_{g_*A} \overset{\eqi^{-1}_{g*}}{\longrightarrow} \cB_A.
\end{equation*}
For a braided tensor functor and a commutative algebra this preserves the subcategory of local modules.

To compute the tensor structure, under the additional assumption $\DiscO=1$, we explicitly choose coset representatives $a=I+\hat{a}\in\DiscGamma$ for $a\in\DiscGamma/I$ and consider from \autoref{sec_DiscForm} the explicit realisation of $\C_a$, $a\in\DiscGamma/I$, in $\cB_A$ as induced modules $\Ind_A\C_{\hat{a}}$. Then, the first functor appearing in the definition of $g_{A*}$ maps
\begin{equation*}
\Ind_A\C_{\hat{a}}=\bigoplus_{j+\hat{a}\in I+\hat{a}} \C_{i+\hat{a}}
\overset{\bar{g}_*}{\longmapsto}
\bigoplus_{j+\hat{a}\in I+\hat{a}} \C_{\bar{g}.j+\bar{g}.\hat{a}}
\end{equation*}
with the $\bar{g}_*A$-action
\begin{equation*}
\bar{g}_*e_i \otimes e_{\bar{g}.j+\bar{g}.\hat{a}}
\overset{\bar{g}_*\rho}{\longmapsto}
\epsilon(i,j)
e_{\bar{g}.(i+j)+\bar{g}.\hat{a}}.
\end{equation*}
The second functor appearing in the definition of $g_{A*}$ precomposes this with $\eqi^{-1}_g$; so it produces on the same object the following $A$-action
\begin{align*}
e_{\bar{g}.i} \otimes e_{\bar{g}.j+\bar{g}.\hat{a}}
&\overset{\eqi^{-1}_{g*}}{\longmapsto}
\eta_g(i)^{-1}
\bar{g}_*e_{i} \otimes e_{\bar{g}.j+\bar{g}.\hat{a}}\\
&\overset{\bar{g}_*}{\longmapsto}
\eta_g(i)^{-1}\epsilon(i,j)
e_{\bar{g}.(i+j)+\bar{g}.\hat{a}}\\
&=\eta_g(j)\eta_g(i+j)^{-1} \epsilon(\bar{g}.i,\bar{g}.j)e_{\bar{g}.(i+j)+\bar{g}.\hat{a}}.
\end{align*}
After substituting $\bar{g}^{-1}.i$ and $\bar{g}^{-1}.j$, this is the action of $\Ind_A\C_{\bar{g}.\hat{a}}$ up to an additional factor $\eta_g(\bar{g}^{-1}.j)\eta_g(\bar{g}^{-1}.i+\bar{g}^{-1}.j)^{-1}$, and accordingly we can construct an explicit isomorphism of $A$-modules
\begin{align*}
\eqi^{-1}_{\bar{g}*}\bar{g}_*\Ind_A\C_{\hat{a}}
&\longrightarrow \Ind_A\C_{\bar{g}.\hat{a}},\\
e_{i+\hat{a}}
&\longmapsto
\eta_g(\bar{g}^{-1}.i)e_{i+\bar{g}.\hat{a}}.
\end{align*}
We also have to take into account that $\bar{g}.\hat{a}$ is not necessarily the chosen representative $\widehat{\bar{g}.a}$, so we postcompose with the respective isomorphism from \autoref{sec_DiscForm}
\begin{align*}
h\colon
\eqi^{-1}_{g*}\bar{g}_*\Ind_A\C_{\hat{a}}
&\longrightarrow \Ind_A\C_{\widehat{\bar{g}.a}},\\
e_{i+\hat{a}}
&\longmapsto
\eta_g(\bar{g}^{-1}.i)
\epsilon(i,\bar{g}.\hat{a}-\widehat{\bar{g}.a})
e_{(i+(\bar{g}.\hat{a}-\widehat{\bar{g}.a}))+\widehat{\bar{g}.a}}.
\end{align*}

The tensor structure $\tau^g_A$ now arises from the incompatibility of these isomorphisms $h$ with the isomorphisms $\hat{\iota}_{\hat{a},\hat{b}}$ used in the tensor product in \autoref{sec_DiscForm}:
\begin{equation*}
\begin{tikzcd}[column sep=small]
\eqi^{-1}_{g*}\bar{g}_*(\C_{i+\hat{a}} \otimes \C_{j+\hat{b}})
\ar{d}{\eqi^{-1}_{g*}\bar{g}_*(\hat\iota_{\hat{a},\hat{b}})}
\ar{r}{\id}
&
\eqi^{-1}_{g*}\bar{g}_*\C_{i+\hat{a}} \otimes \eqi^{-1}_{g*}\bar{g}_*\C_{j+\hat{b}}
\ar{d}{h\otimes h} \\
\eqi_{g*}^{-1}\bar{g}_*\C_{(u({a},{b})+i+j)+\widehat{a+b}}
\ar{d}{h}
&
\C_{(i+(\bar{g}.\hat{a}-\widehat{\bar{g}.a}))+\widehat{\bar{g}.a}} \otimes \C_{(j+(\bar{g}.\hat{b}-\widehat{\bar{g}.b}))+\widehat{\bar{g}.b}}
\ar{d}{\hat\iota_{\widehat{\bar{g}.a},\widehat{\bar{g}.b}}}
\\
\C_{(u({a},{b})+i+j+(\bar{g}.\hat{a+b}-\widehat{\bar{g}.(a+b)}))+\widehat{\bar{g}.(a+b)}}
\arrow[dashed,r]{}{}
&
\C_{(u(\bar{g}.a,\bar{g}.b)+i+j+(\bar{g}.\hat{a}-\widehat{\bar{g}.a}+(\bar{g}.\hat{b}-\widehat{\bar{g}.b})+\widehat{\bar{g}.(a+b)}}
\end{tikzcd}
\end{equation*}
Following again a basis element $e_{i+\hat{a}}\otimes e_{j+\hat{b}}$ gives for the dashed morphism the following scalar factor, under the assumption $\DiscO=1$ and also the simplifying assumption $\epsilon=1$ and $\eta_g=1$ for all $g\in G$ (which in particularly implies $\DiscS(x,y)=1$ for $x,y\in I$):
\begin{equation*}
\DiscS(i+\hat{a},j)^{-1}
\DiscS\bigl(i+(\bar{g}.\hat{a}-\widehat{\bar{g}.a})+\hat{a},j+(\bar{g}.\hat{b}-\widehat{\bar{g}.b})\bigr)=\DiscS(\hat{a},\bar{g}.\hat{b}-\widehat{\bar{g}.b}),
\end{equation*}
which we can see to be again independent of $i$ and $j$, as it should. This is the asserted formula for the tensor structure $\tau^g_A$.
\end{proof}

\begin{prob}
The considerations in this section should be repeated without the simplifying assumptions, i.e.
\begin{enumerate}
\item for $\bar{g}$ only fixing $Q$ but not the chosen abelian $3$-cocycle $(\DiscS,\DiscO)$ (so that the action on $\cB$ will be nonstrict),
\item for $\DiscO\neq 1$,
\item for $\epsilon\neq 1$ and $\eta_g\neq 1$ a nontrivial lift.
\end{enumerate}
\end{prob}
We remark that in the lattice setting below (\autoref{sec_idea2_infTY}), where $\smash{\Vect_{\R^d}^{\DiscQ}}$ admits a distinguished representing abelian $3$-cocycle with $\DiscS$ symmetric and $\DiscO=1$, the first two simplifications may still be assumed, but the third is a proper restriction that forces us to restrict to what we call strongly even lattices in \cite{GLM24b}, \autoref*{GLM2ass_strongeven} (see \autoref{sec:prel}).

We return to the example discussed in \autoref{sec_latticeDiscriminantForm} of an infinite condensation $\cB_A^{\loc}=\smash{\Vect_\Gamma^Q}$, $\Gamma=L^*\!/L$, associated with an even lattice $L$ in $(\R^d,\DiscQ)$.

\begin{ex}\label{exm_counterexampleA1}
Let $\bar{g}$ be the $(-1)$-involution $\bar{g}.v = -v$, $v\in\R^d$, which is an isometry of $(\R^d, \DiscQ)$ and thus acts strictly on $\smash{\Vect_{\R^d}^{\sigma,1}}$. Any lattice $L$ is invariant under $\bar{g}$. If we assume that all scalar products in $L$ are even (as per \autoref*{GLM2ass_strongeven} in \cite{GLM24b}, see \autoref{sec:prel}), then $\epsilon = 1$, and we may choose the trivial lift $\eta_g = 1$. Under these conditions, the nontrivial tensor structure for the action of $G = \DiscG = \Z_2$ is given by \autoref{prop_ActionOnCondensation}.

Let us remark that the case where $\Gamma=L^*\!/L$ is a group of odd order, which is the other extreme, is not covered by our simplifying assumptions since $\epsilon \neq 1$. However, in this case, we can also choose a trivial lift $\eta_g=1$, and the resulting tensor structure is trivial. This matches the strict $\Z_2$-action on $\smash{\Vect_\Gamma^Q}$ for odd-order $\Gamma$ as discussed in \autoref{sec_TY}.

Let $L$ be the (even) $A_1$-lattice, i.e.\ $L = \alpha\Z$ with $\langle \alpha, \alpha \rangle = 2$. Its dual lattice is $L^* = (\alpha/2)\Z$, and the discriminant form is $\Gamma = L^*\!/L \cong \Z_2$. We consider both the trivial and nontrivial actions of $\Z_2 = \langle g \rangle$ on $L$, where $g.\alpha = \pm\alpha$. In both cases, the induced actions of $\Z_2$ on $\Gamma$, which parametrises the objects of $\smash{\Vect_\Gamma^Q}$, are trivial because $\alpha/2 + \alpha\Z = -\alpha/2 + \alpha\Z$.

For the trivial action, the tensor structure is $\tau^g_A = \id$, while for the nontrivial action, the tensor structure is nontrivial, as indicated by \autoref{prop_ActionOnCondensation}. Specifically, let $a = \alpha/2 + \alpha\Z$ be the nonzero coset in $\Gamma$ and choose the representative $\hat{a} = \alpha/2$. Then, $g.\hat{a} - \widehat{g.a} = -\alpha$, and $\DiscS(\hat{a}, g.\hat{a} - \widehat{g.a}) = \e(\langle \alpha/2, \alpha \rangle / 2) = -1$. Thus, the tensor structure is given by
\begin{equation*}
\tau^g_A\colon g_{A*}(\C_a \otimes_A \C_a)
\xrightarrow{\;-1\;}
g_{A*}(\C_a) \otimes_A g_*(\C_a).
\end{equation*}
This demonstrates that different actions on the lattice $L$, and hence on the commutative algebra $A = \C_\epsilon[L]$, can result in the same action on objects but different tensor structures.
\end{ex}


\section{\texorpdfstring{$G$}{G}-Crossed Extensions Versus Condensations}\label{sec_GExtensionVsCurrentExtension}

In this section, we establish a connection between braided $G$-crossed extensions of a braided tensor category and condensations by a commutative algebra. As explained in \autoref{idea_2}, these processes commute in a certain sense. As a first example, we consider the braided $\Z_2$-crossed extensions $\smash{\Vect^\DiscQ_\DiscGamma}[\Z_2, \bar{\eps}]$ of Tambara-Yamagami type and show that \autoref{idea_2} reproduces a braided $\Z_2$-crossed extension $\smash{\Vect^Q_\Gamma}[\Z_2, \eps]$ of the condensation, for the choice $\eps = \bar{\eps}$.


\subsection{Main Statement}\label{sec_ThmGExtensionVsCurrentExtension}

In the following, we prove that taking braided $G$-crossed extensions of braided tensor categories commutes, in a suitable sense, with passing to local modules of an algebra. For \voa{}s, this result implies that taking orbifolds and conformal extensions commute in a similar sense (see \autoref{sec_VOAIdea2}). In the context of anyon condensation, the result has appeared in \cite{BJLP19}. We state the main assertion:

\begin{thm}\label{thm_currentExtVsCrossedExt}
Let $G$ be a finite group and $\cC = \bigoplus_{g \in G} \cC_g$ a braided $G$-crossed tensor category with $G$-braiding $c^G$, and let $\cB = \cC_1$.

Let $A$ be a commutative algebra object in $\cB$, and consider the braided tensor category $\cB_A^\loc$. Suppose that $A$ is invariant under the $G$-action in the sense that there is a $G$-equivariant structure $\eqi = (\eqi_g)_{g \in G}$ with algebra isomorphisms
\begin{equation*}
\eqi_g\colon g_*(A) \to A, \quad \eqi_{gh} = \eqi_g \circ g_*(\eqi_h),
\end{equation*}
or in other words, let $(A, \eqi)$ be an algebra in $\cB \sslash G$. Then:
\begin{enumerate}[label=(\alph*)]
\item\label{item:ass0} The category of $A$-modules $(M,\rho)$ in $\cC=\bigoplus_{g\in G}\cC_g$ is a tensor category with a $G$-grading and a $G$-action with neutral component $\cB_A$.
\item\label{item:ass1} The subcategory of $A$-modules that are $G$-local with respect to $\eqi$ in the sense that the following two morphisms coincide
\begin{equation}\begin{split}\label{formula_Glocal}
A\otimes M&\overset{\rho}{\longrightarrow} M,\\
A\otimes M&\overset{c_{A,M}^G}{\longrightarrow}
M\otimes A\overset{c_{M,A}^G}{\longrightarrow}
g_*A\otimes M\overset{\eqi_g}{\longrightarrow} A\otimes M
\overset{\rho}{\longrightarrow} M
\end{split}\end{equation}
is a braided $G$-crossed extension of $\cB_A^\loc$.
\item\label{item:ass2} The $G$-equivariantisation of this braided $G$-crossed extension is equivalent to the category of local modules $(\cC\sslash G)^\loc_{(A,\eqi)}$.
\end{enumerate}
\end{thm}

The statement of the theorem is illustrated by the following diagram:
\begin{equation*}
\begin{tikzcd}
\cB
\arrow[hookrightarrow]{rrrr}{\text{$G$-crossed extension}}
\arrow{d}{\text{$\otimes_A$}}
\arrow[bend right=50,rightsquigarrow]{dd}{}
&&&&
\cC
\arrow{d}{\text{$\otimes_A$}}
\arrow[bend left=50,rightsquigarrow]{dd}{}
\\
\cB_A
&&&&
\cC_A
\\
\cB_A^\loc
\arrow[hookrightarrow]{u}{}
\arrow[hookrightarrow]{rrrr}{\text{$G$-crossed extension}}
&&&&
\cC_A^\loc
\arrow[hookrightarrow]{u}{}
\end{tikzcd}
\end{equation*}
Note that, in contrast to \autoref{sec_DiscFormAut}, we implicitly set $G$ acting on $A$ equal to its quotient $\bar{G}$ acting on $\cB$. This is possible without loss of generality by allowing $G$ to act nonfaithfully on $\cB$. The reason is that the two braided crossed extensions should be over the same group.

\begin{proof}
\begin{enumerate}[label=(\alph*),wide]
\item The $G$-braiding $\{c_{A,X}\colon A \otimes X \to X \otimes A\}_{X\in \cC}$ is a half-braiding and defines a commutative central structure for $A$, meaning that $A \in \mathcal{Z}(\cC)$ is a commutative algebra. This allows us to apply standard constructions to endow the category of $A$-modules in $\cC$ with a tensor structure $\otimes_A$ and a unit object $A$. For an algebra $A$ in the neutral component $\cC$ of a $G$-graded tensor category $\bigoplus_{g \in G} \cC_g$, the category of representations naturally inherits a $G$-graded tensor structure.

For an algebra $(A,\eqi)$ with $G$-equivariant structure $\eqi$ in a tensor category $\cC$ with $G$-action, the category of $A$-representations carries a $G$-action, as in \autoref{sec_DiscFormAut}, by sending an $A$-module $(M,\rho)$ to the object $g_*M$ and using the equivariant structure
\begin{equation*}
g_{A*}\colon \cB_A \overset{g_*}{\longrightarrow}\cB_{g_*A} \overset{\eqi^{-1}_*}{\longrightarrow} \cB_A.
\end{equation*}
The composition structure $T_2$ and the tensor structure for the category of $A$-modules are derived from the corresponding structures of the underlying $G$-action on $\cC$. This can be verified directly, assuming the $G$-action is strict, which can be done without loss of generality due to the coherence theorem for braided $G$-crossed tensor categories.

\item We now consider the subcategory of $G$-local $A$-modules, as defined in the theorem statement, which depends on the $G$-equivariant structure $\eqi$. The $G$-crossed braiding of these $A$-modules is induced from the underlying $G$-crossed braiding of $\cC$, similar to the noncrossed case. This braiding is a morphism of $A$-modules due to the $G$-locality condition, aligning with the standard case and satisfying (iii) of \autoref{def_Gcrossed}.

\item We unpack the definitions of $(\cC \sslash G)^\loc_{(A, \eqi)}$. By definition, $(A, \eqi)$ is a commutative algebra in $\cB \sslash G = \cC_1 \sslash G$ and therefore also in $\cC \sslash G$. By definition, a module $N = \bigoplus_{g \in G} N_g$ over $(A, \eqi)$ in $\cC \sslash G$ is a module $(N, \rho) = \bigoplus_{g \in G} (N_g, \rho_g)$ in $\cC = \bigoplus_{g \in G} \cC_g$ over $A$ in $G$-degree 1, together with an equivariant structure $\eqi^N_g\colon g_*(N) \to N$ such that the action is compatible with the equivariant structures, in the sense that the following diagram commutes:
\begin{equation*}
\begin{tikzcd}
g_*(A)\otimes g_*(N)
\arrow{r}{(\tau^g)^{-1}}
\arrow[swap]{d}{\eqi_g\otimes \eqi^N_g}
& g_*(A\otimes N)
\arrow{r}{g_*(\rho)}
& g_*(N) \arrow[swap]{d}{\eqi_g^N} \\
A\otimes N \arrow{rr}{\rho}
&& M
\end{tikzcd}
\end{equation*}
A module $(N,\rho,\eqi^N)$ over $(A,\eqi)$ is local if and only if each module $(N_g,\rho_g)$ in $\cC_g$ fulfils the $G$-locality condition~\eqref{formula_Glocal}. We know from \autoref{sec_CurrentExtension}:
\begin{enumerate}[label=(\arabic*)]
\item The set of $(A,\eqi)$-modules $(\cC\sslash G)_{(A,\eqi)}$ is a tensor category. The tensor product in $\cC \sslash G$ is given by
\begin{equation*}
(N',\eqi^{N'})\otimes (N'',\eqi^{N''})=(N'\otimes N'',\eqi^{N'\otimes N''})
\end{equation*}
with the equivariant structure
\begin{equation*}
\eqi^{N'\otimes N''}\colon g_*(N'\otimes N'')
\overset{\tau^g}{\longrightarrow} g_*N'\otimes g_*N''
\overset{\eqi^{N'}_g\otimes \eqi^{N''}_g\otimes \id}{\longrightarrow} N'\otimes g_*.
\end{equation*}
The tensor product $\otimes_A$ is given by a coequaliser $N'\otimes N''\to N'\otimes_A N''$ as discussed in \autoref{sec_CurrentExtension}. The action of $A$ is given by the action on the left factor (which, of course, includes an associator). This action is again an equivariant $(A,\eqi)$-action.
\item
The set of local $(A, \eqi)$-modules $(\cC \sslash G)_{(A, \eqi)}$ forms a braided tensor category. The braiding is induced from the braiding in $\cC \sslash G$, and is defined component-wise by
\begin{equation*}
N_g'\otimes N_h''
\overset{c^G}{\longrightarrow}g_*N_h''\otimes N'_g
\overset{\eqi^{N'}\otimes \id}{\longrightarrow}N_{ghg^{-1}}''\otimes N'_g.
\end{equation*}
\end{enumerate}
From this description it is clear that the braided tensor category coincides with the equivariantisation of the braided $G$-crossed tensor category in (b).\qedhere
\end{enumerate}
\end{proof}

\begin{rem}
In favourable cases, we would expect that any braided $G$-crossed extension of $\cB_A^\loc$ can be obtained from a braided $G$-crossed extension of $\cB$. This is straightforward when $A$ describes a simple-current extension, but generally, a strategy is needed to construct the twisted sectors of $\cC$ from the twisted sectors of $\cC_A^\loc$. One approach, although potentially tedious, involves realising all twisted components as modules over algebras in $\cC_A^\loc$, then taking $A$-invariants, and finally checking compatibility with all additional structures.
\end{rem}


\subsection{Example: From Tambara-Yamagami to Tambara-Yamagami}\label{sec_TYtoTY}

In the following, we demonstrate \autoref{thm_currentExtVsCrossedExt} using the example discussed in \autoref{sec_DiscForm} and \autoref{sec_TY}. This example also serves as a template for the new results presented in \autoref{sec_idea2_infTY}.

\medskip

Let $(\DiscGamma, \DiscQ)$ be a discriminant form, i.e.\ a finite, abelian group equipped with a nondegenerate quadratic form $\DiscQ\colon\DiscGamma\to\C^\times$, with associated bimultiplicative form $\DiscB$. We consider the associated ribbon fusion category $\cB = \smash{\Vect_\DiscGamma^\DiscQ}$, characterised by a representing abelian $3$-cocycle $(\DiscS, \DiscO)$. For simplicity, we restrict to the case that $|\DiscGamma|$ is odd and $(\DiscS, \DiscO) = (\DiscB^{1/2}, 1)$.

As in \autoref{sec_DiscForm}, let $I \subset \DiscQ$ be an isotropic subgroup, and consider the associated commutative algebra $A = \C_\epsilon[I]$ with $\epsilon(a, b)\epsilon(b, a)^{-1} = \DiscS(a, b)$, such that $\cB_A^\loc=\smash{\Vect_\Gamma^Q}$, where $\Gamma = I^\perp/I$ and $Q$ is the quadratic form $\DiscQ$ restricted to $I^\perp$ and factored through $I^\perp/I$. Since $\smash{\DiscB(i, j)}= \pm 1$ for $i, j \in I$, and under the additional assumption that $|\DiscGamma|$ is odd, it follows that $\DiscB(i, j) = 1$. Thus, we may choose $\epsilon = 1$.

As in \autoref{sec_DiscFormAut}, consider a group $G$ acting on $\DiscGamma$ in a way that preserves both $\DiscQ$ and the chosen representing abelian $3$-cocycle $(\DiscS, \DiscO)$. This action induces a strict action on $\cB$. As an example, consider the case where $G = \langle g \rangle \cong \Z_2$ acts on $\DiscGamma$ as $-\!\id$, which preserves the abelian $3$-cocycle $(\smash{\DiscB}^{1/2}, 1)$ when $|\DiscGamma|$ is odd. This results in a strict action on $\smash{\Vect_\DiscGamma^\DiscQ}$ via strict tensor functors:
\begin{equation*}
g_* \C_v = \C_{-v},\quad \tau^g_{\C_v, \C_w} = \id
\end{equation*}
for all $v, w \in \DiscGamma$. We need to choose an equivariant structure $\eqi_g\colon g_*A \to A$, defined on the basis $e_i \in \C_i$ of $A$ by $g_*e_i \mapsto \eta_g(i)e_{-i}$, where $\eta_g(i)$ is a $1$-cochain on $I$ satisfying $\d\eta_g(i, j) = \epsilon(g.i, g.j)\epsilon(i, j)^{-1}$ for all $i,j\in I$. Since $|\DiscGamma|$ is odd, we can choose $\epsilon = 1$, allowing us to set $\eqi_g = 1$. In this scenario, $G = \bar{G}$.

\medskip

We are now in a position to demonstrate \autoref{thm_currentExtVsCrossedExt}. Consider from \autoref{cor_TY}, for a choice of sign $\bar\eps$, the Tambara-Yamagami category (with positive quantum dimensions)
\begin{equation*}
\Vect^\DiscQ_\DiscGamma[\Z_2,\bar\eps]=\cC_1\oplus \cC_g
\end{equation*}
as a braided $G$-crossed extension of $\cB=\cC_1$. Then the statement of \autoref{thm_currentExtVsCrossedExt} is that condensation by the $\Z_2$-equivariant commutative algebra $(A, \eqi)$ results in another braided $\Z_2$-crossed extension (with positive quantum dimensions) of the braided tensor category $\smash{\cB_A^{\loc}}=\smash{\Vect_{\Gamma}^{Q}}$. By the uniqueness of such extensions, this extension must again be of the form $\smash{\Vect_{\Gamma}^{Q}}[\Z_2,\eps]$ for some sign $\eps$:
\begin{equation*}
\begin{tikzcd}
\Vect_\DiscGamma^\DiscQ
\arrow[rightsquigarrow]{d}{}
\arrow[hookrightarrow]{rrr}{\text{$\Z_2$-crossed ext.}}
&&&
\Vect^\DiscQ_\DiscGamma[\Z_2,\bar\eps]
\arrow[rightsquigarrow]{d}{}
\\
\Vect_{\Gamma}^{Q}
\arrow[hookrightarrow]{rrr}{\text{$\Z_2$-crossed ext.}}
&&&
\Vect^{Q}_{\Gamma}[\Z_2,\eps]
\end{tikzcd}
\end{equation*}

\autoref{thm_currentExtVsCrossedExt} explicitly produces this braided $G$-crossed extension and proves that $\eps = \bar{\eps}$. For the last claim, we note that the $G$-ribbon twist in the condensation by $(A, \eqi)$ is simply the $G$-ribbon twist in the base category $\cC$. By the formulae in \autoref{sec_TY}, this means that
\begin{equation*}
\frac{\eps}{\alpha} = \frac{\bar{\eps}}{\bar{\alpha}}.
\end{equation*}
By the definition of $\alpha$ and $\bar{\alpha}$ and \autoref{rem_gauss}, the inverse square of both sides gives
\begin{equation*}
\eps^{-1} \e(-\sign(\Gamma, Q^{-1/2}))
=
\bar{\eps}^{-1} \e(-\sign(\DiscGamma, \DiscQ^{-1/2})),
\end{equation*}
and the signatures of $(\DiscGamma, \DiscQ^{-1/2})$ and $(\Gamma, Q^{-1/2})=(I^\perp/I, \DiscQ^{-1/2}|_{I^\perp})$ agree. This is a nontrivial statement about Gauss sums, as found in \cite{CS99}. Hence, $\eps = \bar{\eps}$ and we find:
\begin{prop}\label{prop_TYtoTY}
The category of $\Z_2$-local modules over $(A,\eqi)$ in $\smash{\Vect^\DiscQ_\DiscGamma}[\Z_2,\bar\eps]$ is the braided $\Z_2$-crossed tensor category
\begin{equation}\label{eq: equation in proposition}
\bigl(\Vect_{\DiscGamma}^{\DiscQ}[\Z_2, \bar{\eps}]\bigr)_{(A, \eqi)}^{\loc}
\cong \Vect_{\Gamma}^{Q}[\Z_2, \eps]
\end{equation}
with the sign $\eps = \bar{\eps}$.
\end{prop}

The main point is that the category on the left-hand side of equation \eqref{eq: equation in proposition} can be computed directly without prior knowledge. To illustrate the method, we now compute explicitly the objects of $\bigl(\Vect_{\DiscGamma}^{\DiscQ}[\Z_2, \bar{\eps}]\bigr)_{(A, \eqi)}^{\loc}$.

For a full computation of all associators, braidings and tensor structures, the reader should consult the new examples in \autoref{sec_idea2_infTY}.

\begin{ex}
The $(A,\eqi)$-modules over the untwisted sector $\cC_1 = \cB$ correspond to the commutative algebra condensation we initially constructed in \autoref{sec_DiscForm}, and we denote its simple objects by $\C_{a}$ for cosets $a \in I^\perp/I = \Gamma$. Now, let us consider $(A, \eqi)$-modules over the twisted sector $\cC_g$, which is the category $\Vect$ generated by the simple object $\X$. A general object in this sector is denoted by $V\X$, where $V = \C^n$ for some $n$, representing the multiplicity space. Modules of $A$ in the twisted sector are objects $M = V\X$ with an action given by
\begin{equation*}
\rho_M\colon A \otimes V\X \to V\X.
\end{equation*}

These modules are in bijective correspondence with representations of the twisted group algebra $\C_\epsilon[I]$ on the multiplicity space $V$, as the Tambara-Yamagami associator on $\C_i \otimes \C_j \otimes \X$ is trivial. We denote this action by $e_i.\basis$ for $i \in I$ and $\basis \in V$.

The local $(A, \eqi)$-modules in the sense of \autoref{thm_currentExtVsCrossedExt} correspond to $\C_\epsilon[I]$-modules $V$ such that the two maps $A \otimes V\X \to V\X$ given by
{\allowdisplaybreaks
\begin{align*}
e_i \otimes \basis&\overset{\rho_M}{\longmapsto}e_i.\basis,\\[+2mm]
e_i \otimes \basis
&\overset{c_{A,M}}{\longmapsto} q(i)^{-1} \, \basis \otimes e_i\\
&\overset{c_{M,A}}{\longmapsto} q(i)^{-2} \, g_*e_i \otimes \basis\\
&\overset{\eqi_A}{\longmapsto} q(i)^{-2}\eta_g(i) \, e_{-i} \otimes \basis\\
&\overset{\rho_M}{\longmapsto} q(i)^{-2}\eta_g(i) \, e_{-i}.\basis
\end{align*}
}%
are equal for all $\basis \in V$. Since $Q(i) = q(i)^2 = 1$ is the twist on $A$, these correspond to modules where $e_{2i}$ acts by $\eta_g(i)\epsilon(i, -i)\epsilon(i, i)^{-1}$ for all $i \in I$.

For a group $I$ of odd order, this condition uniquely determines a $1$-dimensional representation of $I$. Since we assumed $\epsilon = \eta_g = 1$, this representation is the trivial one, with $V = \C^1$. Thus, we observe that as abelian categories, as asserted:
\begin{equation*}
\left(\cB[\Z_2, \eps]\right)_{(A, \eqi)}^\loc = \Vect_{\Gamma} \oplus \Vect.
\end{equation*}
\end{ex}


\section{\autoref{idea_2}: \texorpdfstring{$\Z_2$}{Z\_2}-Crossed Extensions for Even Groups}\label{sec_TYEven}

The approach from \autoref{sec_GExtensionVsCurrentExtension}, specifically \autoref{idea_2}, can be used to construct for $G=\langle g\rangle\cong\Z_2$ the braided $G$-crossed extension
\begin{equation*}
\Vect_\Gamma^Q[\Z_2,\eps]=\Vect_\Gamma\oplus\Vect_{\Gamma/2\Gamma}
\end{equation*}
of the (nondegenerate) braided fusion category $\smash{\Vect_\Gamma^Q}$ for a discriminant form $(\Gamma, Q)$ with a $\Z_2$-action given by $-\!\id$. Here, in contrast to \autoref{sec_TY}, the finite, abelian group $\Gamma$ may have \emph{even order}, leading to new examples of braided $\Z_2$-crossed tensor categories. By equivariantising, these also provide new examples of modular tensor categories.

The explicit definition of the braided $\Z_2$-crossed tensor categories $\smash{\Vect_\Gamma^Q[\Z_2,\eps]}$ is given in \cite{GLM24b}. This includes a description of the associators, $\Z_2$-braiding and $\Z_2$-ribbon structure. We only summarise these results here, but we explain in detail how \autoref{idea_2} produces the data that are used in \cite{GLM24b} to define these categories.


\subsection{Summary of Results}

Morally, we want to apply \autoref{idea_2} to the situation
\begin{equation*}
\begin{tikzcd}
\Vect_{\R^d}^\DiscQ
\arrow[rightsquigarrow]{d}{}
\arrow[hookrightarrow]{rrr}{\text{$\Z_2$-crossed ext.}}
&&&
\Vect_{\R^d}^\DiscQ \oplus \Vect
\arrow[rightsquigarrow]{d}{}
\\
\Vect_{\Gamma}^Q
\arrow[hookrightarrow]{rrr}{\text{$\Z_2$-crossed ext.}}
&&&
\Vect^{Q}_{\Gamma}[\Z_2,\eps]
\end{tikzcd}
\end{equation*}
where $(\Gamma,Q)$ corresponds to an isotropic subgroup (meaning an even lattice~$L$) in $(\DiscGamma,\DiscQ)=(\R^d,\DiscQ)$ as discussed in \autoref{sec_latticeDiscriminantForm}, in order to produce the braided $\Z_2$-crossed tensor category $\smash{\Vect_\Gamma^Q}[\Z_2,\eps]$ with all structures. However, for this we need an infinite version of a Tambara-Yamagami category associated with the infinite, abelian group $\R^d$ with automorphism $-\!\id$, which should be $\smash{\Vect_{\R^d}^\DiscQ}\oplus\Vect$ as abelian category, but for which we cannot define a tensor product in the usual sense.

Hence, we turn our approach around and use \autoref{idea_2} in a nonrigorous but very explicit way to determine in \autoref{sec_idea2_infTY} the data that \emph{should} define $\smash{\Vect_\Gamma^Q}[\Z_2,\eps]$. As several times in this article, we assume for simplicity that $\epsilon=1$ (cf.\ \autoref{sec_latticeDiscriminantForm}), which corresponds to the lattice $L$ being even in a stronger sense; see \autoref*{GLM2ass_strongeven} in \cite{GLM24b} and \autoref{sec:prel}. The general case is left for future work.

We then use this in \cite{GLM24b}, \autoref*{GLM2sec_evenTY}, as input for a rigorous (but without \autoref{sec_idea2_infTY} ad hoc seeming) definition of a braided $\Z_2$-crossed tensor category $\LM(\Gamma,\sigma,\omega,\delta,\eps\,|\,q,\alpha,\beta)$. There, we prove in detail that this is indeed a braided $\Z_2$-crossed extension of $\smash{\Vect_\Gamma^Q}$ and hence must coincide with $\smash{\Vect_\Gamma^Q}[\Z_2,\eps]$.

In \autoref*{GLM2sec_equiv} of \cite{GLM24b} we describe the $\Z_2$-equivariantisation $\smash{\Vect_\Gamma^Q}[\Z_2,\eps]\sslash\Z_2$, which is a modular tensor category. There, in particular, we compute the simple objects, fusion rules and modular data.

Moreover, in \autoref*{GLM2sec_latticedata} of \cite{GLM24b} we realise $(\Gamma,Q)$ explicitly as discriminant form $\Gamma=L^*\!/L$ of some even lattice $L$. We discuss how the data used to define the braided $\Z_2$-crossed tensor category $\LM(\Gamma,\sigma,\omega,\delta,\eps\,|\,q,\alpha,\beta)$ are obtained from lattice data. If $L$ is positive-definite, the braided $\Z_2$-crossed tensor category will in this way appear for the $\Z_2$-orbifold of the corresponding lattice \voa{}, which we shall discuss in \autoref{sec_VOAevenTV}.

\medskip

We briefly summarise the main result of \cite{GLM24b}, \autoref*{GLM2thm_evenTY}, the construction of the braided $\Z_2$-crossed tensor categories $\smash{\Vect_\Gamma^Q}[\Z_2,\eps]$ that generalise the Tambara-Yamagami case (cf.\ \autoref{thm_TYribbon}, \autoref{cor_TYpseudounitary} and \autoref{cor_TY}).
\begin{thm}\label{thm_evenTY_here}
Let $G=\langle g\rangle\cong\Z_2$. The data given in \autoref*{GLM2sec_evenTY} of \cite{GLM24b} define a braided $\Z_2$-crossed tensor category
\begin{equation*}
\LM(\Gamma,\sigma,\omega,\delta,\eps\,|\,q,\alpha,\beta)=\Vect_\Gamma\oplus\Vect_{\Gamma/2\Gamma},
\end{equation*}
which is a braided $\Z_2$-crossed extension of $\smash{\Vect_\Gamma^Q}$ for a discriminant form $(\Gamma,Q)$, with the categorical $\Z_2$-action being $g_*\C_a=\C_{-a}$ for $a\in\Gamma$ with tensor structure $g_*(\C_a\otimes\C_b)\to g_*\C_a\otimes g_*\C_b$ given by $\omega(a,b,-b)$ for $a,b\in\Gamma$. The ribbon twist yields positive quantum dimensions if and only if $\alpha\beta=\eps$.
\end{thm}


\subsection{Condensation of Infinite Tambara-Yamagami Category}\label{sec_idea2_infTY}

We now show how the ad-hoc definition in \cite{GLM24b}, \autoref*{GLM2sec_evenTY}, of the braided $\Z_2$-crossed tensor category $\LM(\Gamma,\sigma,\omega,\delta,\eps\,|\,q,\alpha,\beta)$ can be derived using the idea in \autoref{sec_ThmGExtensionVsCurrentExtension} (\autoref{idea_2}) and realising the discriminant form as $\Gamma=L^*\!/L$ for some even lattice $L$, which in turn can be embedded into its ambient space $L\otimes_\Z\R\cong\R^d$.
\begin{equation*}
\begin{tikzcd}
\Vect_{\R^d}^\DiscQ
\arrow[rightsquigarrow]{d}{}
\arrow[hookrightarrow]{rrr}{\text{$\Z_2$-crossed ext.}}
&&&
\Vect_{\R^d}^\DiscQ \oplus \Vect
\arrow[rightsquigarrow]{d}{}
\\
\Vect_{\Gamma}^Q
\arrow[hookrightarrow]{rrr}{\text{$\Z_2$-crossed ext.}}
&&&
\Vect^{Q}_{\Gamma}[\Z_2,\eps]
\end{tikzcd}
\end{equation*}
This derivation using an infinite category is not rigorous, but very explicit. It provides a motivation for the definition of $\smash{\Vect^Q_\Gamma}[\Z_2,\eps]=\LM(\Gamma,\sigma,\omega,\delta,\eps\,|\,q,\alpha,\beta)$ in \cite{GLM24b}.

\medskip

Recall from \autoref{sec_latticeDiscriminantForm} the infinite braided tensor category $\cB=\cC_1=\smash{\Vect_{\R^d}^{\DiscQ}}$ for a nondegenerate quadratic form $\DiscQ(v)=\e(\langle v,v\rangle/2)$ on $\R^d$, characterised by the abelian $3$-cocycle $(\DiscS,\DiscO)$ with $\DiscS(v,w)=\e(\langle v,w\rangle/2)$ for $v,w\in\R^d$ and $\DiscO=1$. Also following \autoref{sec_latticeDiscriminantForm}, we take an even lattice $L$ in $(\R^d,\DiscQ)$ and the corresponding commutative algebra $A=\C_{\epsilon}[L]$ in $\cB$ for some $\epsilon\colon L\times L\to\C^\times$ with $\d\epsilon=\DiscO=1$ (i.e.\ a $2$-cocycle) and $\epsilon(u,v)\epsilon(v,u)^{-1}=\DiscS(u,v)$ for $u,v\in L$. The corresponding condensation is the braided tensor category $\cB_A^{\loc}=\smash{\Vect_\Gamma^Q}$, with the concrete representing abelian $3$-cocycle $(\sigma,\omega)$. Denote again by $\hat{a}\in L^*$ a fixed choice of representatives for the $L$-cosets $a+L\in\Gamma$, and let $\smash{u(a,b)=\hat{a}+\hat{b}-\smash{\widehat{a+b}}}$ be the associated $2$-cocycle with values in~$L$.

We consider, as in \autoref{sec_DiscFormAut}, the action of $G=\langle g\rangle\cong\Z_2$ on $\R^d$ where the generator $g$ of $\Z_2$ acts by multiplication by $-1$. This gives rise to a strict action by strict tensor functors on $\cB$ because the action fixes the chosen abelian $3$-cocycle.

\medskip

We consider (not fully rigorously) the corresponding infinite Tambara-Yamagami category from \autoref{sec_TY}
\begin{equation*}
\text{``}\TY(\R^d,\DiscS,\bar{\eps}\,|\,\Discq,\bar{\alpha},\bar{\beta})\text{''},
\end{equation*}
which as an abelian category is of the form
\begin{align*}
\cC=\cC_1\oplus\cC_g=\Vect_{\R^d}^{\DiscQ}\oplus\Vect.
\end{align*}

\begin{prob}
Since $\Gamma=\R^d$ is infinite, the formulae of a Tambara-Yamagami category do not define a tensor category, in the usual sense: the tensor product $\X\otimes\X$ supposedly contains an infinite (continuous) direct sum of simple objects $\C_t$, and, even worse, the associator on $\X\otimes\X\otimes \X$ supposedly acts on infinite-dimensional multiplicity spaces, as some version of a Fourier transformation. It would be very interesting to rigorously define such a braided $G$-crossed tensor category, e.g., enriched over Hilbert spaces.
\end{prob}
Disregarding these problems, we now want to apply \autoref{idea_2} with \autoref{thm_currentExtVsCrossedExt} to construct the (finite) braided $\Z_2$-crossed tensor category $\LM(\Gamma,\sigma,\omega,\delta,\eps\,|\,q,\alpha,\beta)$ as the condensation of this category $\cC=\cC_1\oplus\cC_g$ by $A$.

As in \autoref{sec_DiscFormAut}, we note that the $\Z_2$-action preserves any lattice $L$. We assume that the lattice is strongly even in the sense of \autoref*{GLM2ass_strongeven} in \cite{GLM24b} (see \autoref{sec:prel}); so $\epsilon=1$ and we may choose the trivial $\Z_2$-equivariant structure $\eqi$ on $A$ by $\eta_g=1$.

We now compute the modules of the $\Z_2$-equivariant commutative algebra $A$ in the braided $\Z_2$-crossed tensor category $\cC=\cC_1\oplus\cC_g$. The computation starts formally identically to \autoref{sec_TYtoTY}, but now we explicitly compute associators and braidings.

Again, the modules of $A$ contained as objects in $\cC_1$, which is a tensor subcategory, simply recover $\cB_A$. The modules of $A$ contained as objects in $\cC_g$ are of the form $V\X$, where $V$ is the multiplicity space, and the action of $A$ corresponds to a linear representation of $\C_{\epsilon}[L]$ on $V$, again because the particular associator $\C_v\otimes \C_w\otimes \X$ in the Tambara-Yamagami category $\cC$ is trivial. The $G$-local modules with respect to the equivariant structure are again those, where
\begin{equation*}
e_i.\basis=\Discq(i)^{-1}\Discq(i)^{-1}\,e_{-i}.\basis
\end{equation*}
for all $i\in L$ and $\basis\in V$. Since we assumed $\epsilon=1$, the simple modules of $\C_{\epsilon}[L]$ are $1$-dimensional and given by characters $\chi\colon L\to \C^\times$, and they are local if and only if $\chi(i)=\chi(-i)$ since $\Discq(i)^2=\DiscQ(i)=1$. Hence, $G$-local modules are of the form $\X^\chi$, parametrised by characters $\chi\colon L/2L\to \C^\times$ and the overall category of $G$-local $A$-modules in $\cC$ is
\begin{equation*}
(\cC_1\oplus\cC_g)^\loc_A=\Vect_\Gamma\oplus\Vect_{\widehat{L/2L}}.
\end{equation*}

As discussed in detail in \cite{GLM24b}, \autoref*{GLM2sec_latticedata}, we may identify $\widehat{L/2L}\cong\Gamma/2\Gamma$ in the present setting.


\subsection*{Tensor Products and Identifications}

We now compute the tensor products $\otimes_A$ and explicit isomorphisms to the $A$-modules in the explicit form above.


\subsubsection*{Tensor product $\C_{\hat{a}+L}\otimes_A \C_{\hat{b}+L}\cong \C_{\widehat{a+b}+L}$:}

We repeat this computation from \autoref{sec_DiscForm} with $\epsilon=1$ and in the present language:
\begin{equation*}
\begin{tikzcd}[column sep=small]
\Bigl(\bigl(\bigoplus_{i\in L} \C_{\hat{a}+i}\bigr) \otimes \C_x\Bigr)\otimes\bigl(\bigoplus_{j\in L} \C_{\hat{b}+j}\bigr)
\ar{d}{\DiscS(\hat{a}+i,x)}
\ar{r}{\DiscO=1}
&
\bigl(\bigoplus_{i\in L} \C_{\hat{a}+i}\bigr) \otimes \Bigl(\C_x\otimes\bigl(\bigoplus_{j\in L}\C_{\hat{b}+j}\bigr)\Bigr)
\ar{d}{\rho}
\\
\Bigl(\C_x\otimes\bigl(\bigoplus_{a\in I} \C_{\hat{a}+i}\bigr)\Bigr)\otimes\bigl(\bigoplus_{j\in L} \C_{\hat{b}+j}\bigr)
\ar{r}{\rho}
&
\bigl(\bigoplus_{i\in L} \C_{\hat{a}+i}\bigr)\otimes \bigl(\bigoplus_{j\in L} \C_{\hat{b}+j}\bigr)
\end{tikzcd}
\end{equation*}
for $a,b\in \Gamma$ and for all $x\in L$, and where $\rho$ is the action of $A$ on the respective $A$-module. The equivalence relation on the coequaliser from $(e_{\hat{a}+i}\otimes e_x)\otimes e_{\hat{b}+j}$ is
\begin{equation*}
(e_{\hat{a}+i}\otimes e_{x+\hat{b}+j})
\;\sim\;
\DiscS(\hat{a}+i,x)(e_{\hat{a}+i+x}\otimes e_{\hat{b}+j}).
\end{equation*}
We identify the above coequaliser with the intended result
\begin{equation*}
\iota_{\C_{\hat{a}+L},\C_{\hat{b}+L}}\colon
\bigoplus_{i,j\in L}\C_{\hat{a}+i}\otimes \C_{\hat{b}+j}/\!\sim\;\,\overset{\sim}{\longrightarrow}\;\bigoplus_{k\in L}\C_{\widehat{a+b}+k}
\end{equation*}
by choosing as preimage of $e_{\widehat{a+b}+k}$ the $\,\sim\,$-representative $e_{\hat{a}+k-u(a,b)}\otimes e_{\hat{b}}$. This is also an $A$-module isomorphism as the action on the first factor is $e_x\otimes e_{\hat{a}+i}\to e_{\hat{a}+i+x}$ and the associator in $\cC$ is trivial. On an arbitrary element, it follows from the previous formula that the identification with $\C_{\widehat{a+b}+i}$ maps for all $a,b\in\Gamma$
\begin{equation*}
\boxed{
\iota_{\C_{\hat{a}+L},\C_{\hat{b}+L}}\colon
(e_{\hat{a}+i}\otimes e_{\hat{b}+j})\;\sim\;(e_{\hat{a}+i+j}\otimes e_{\hat{b}})\mapsto\DiscS(\hat{a}+i,j)e_{\widehat{a+b}+u(a,b)+i+j}.
}
\end{equation*}


\subsubsection*{Tensor product $\C_{\hat{a}+L}\otimes_A \X^\chi\cong \X^{\DiscS(\hat{a},\cdot)\chi}$:}

We consider
\begin{equation*}
\begin{tikzcd}
\bigl((\bigoplus_{i\in L} \C_{\hat{a}+i}) \otimes \C_x\bigr)\otimes \X^\chi
\ar{d}{\DiscS(\hat{a}+i,x)}
\ar{rr}{\id}
&&
(\bigoplus_{i\in L} \C_{\hat{a}+i}) \otimes(\C_x\otimes \X^\chi)
\ar{d}{\rho}
\\
\bigl(\C_x\otimes (\bigoplus_{a\in I} \C_{\hat{a}+i})\bigr)\otimes \X^\chi
\ar{rr}{\rho}
&&
(\bigoplus_{i\in L} \C_{\hat{a}+i})\otimes \X^\chi
\end{tikzcd}
\end{equation*}
for all $a\in\Gamma$, $\chi\in\widehat{L/2L}$ and $x\in L$. The equivalence relation on the coequaliser from $(e_{\hat{a}+i}\otimes e_x)\otimes \basis$ for $\basis$ a chosen basis vector in the $1$-dimensional multiplicity space of $\X^\chi$ is
\begin{equation*}
(e_{\hat{a}+i}\otimes \chi(x)\basis)
\;\sim\;
\DiscS(\hat{a}+i,x)(e_{\hat{a}+i+x}\otimes \basis).
\end{equation*}
In particular, for $x=-i$
\begin{equation*}
\boxed{
\iota_{\C_{\hat{a}+L},\X^\chi}\colon
(e_{\hat{a}+i}\otimes \basis)
\;\sim\;
\DiscS(\hat{a}+i,-i)\chi(i)(e_{\hat{a}}\otimes \basis)
\mapsto
\DiscS(\hat{a},i)^{-1}\chi(i) \, \basis,
}
\end{equation*}
noting that $\DiscS(i,-i)=1$ for $i\in L $, gives an identification with $\X^{\DiscS(\hat{a},\cdot)^{-1}\chi}$ since the action is
\begin{equation*}
e_i\otimes (e_{\hat{a}}\otimes \basis)
\mapsto
(e_i\otimes e_{\hat{a}})\otimes \basis
\mapsto
(e_{\hat{a}+i}\otimes \basis)
\sim \DiscS(\hat{a},i)^{-1}\chi(i)\, \basis.
\end{equation*}


\subsubsection*{Tensor product $\X^\chi\otimes_A\C_{\hat{a}+L}\cong\X^{\DiscS(\hat{a},\cdot)\chi}$:}

We consider
\begin{equation*}
\begin{tikzcd}
(\X^\chi \otimes \C_x)\otimes (\bigoplus_{i\in L} \C_{\hat{a}+i})
\ar{d}{\Discq(x)^{-1}}
\ar{rr}{\id}
&&
\X^\chi \otimes (\C_x\otimes \bigoplus_{i\in L} \C_{\hat{a}+i})
\ar{d}{\rho}
\\
(\C_x\otimes \X^\chi)\otimes (\bigoplus_{i\in L} \C_{\hat{a}+i})
\ar{rr}{\rho}
&&
X^\chi \otimes (\bigoplus_{i\in L} \C_{\hat{a}+i})
\end{tikzcd}
\end{equation*}
for all $a\in\Gamma$, $\chi\in\widehat{L/2L}$ and $x\in L$. The equivalence relation on the coequaliser from $(\basis\otimes e_x)\otimes e_{\hat{a}+i}$ for $\basis$ a basis vector in the $1$-dimensional multiplicity space of $\X^\chi$ is
\begin{equation*}
(\basis\otimes e_{\hat{a}+i+x})
\;\sim\;
\Discq(x)^{-1}(\chi(x)\basis\otimes e_{\hat{a}+i} ).
\end{equation*}
In particular, for $x=-i$ the formula
\begin{equation*}
\boxed{
\iota_{\X^\chi,\C_{\hat{a}+L}}\colon
(\basis\otimes e_{\hat{a}+i})
\;\sim\;
\Discq(i)\chi(i)(\basis\otimes e_{\hat{a}})
\mapsto \Discq(i)\chi(i) \basis
}
\end{equation*}
gives an identification with $\X^{\DiscS(\hat{a},\cdot)^{-1}\chi}$ since the action with the nontrivial associator of type $\C\X\C$ in $\cC$ is
\begin{equation*}
e_i\otimes (\basis\otimes e_{\hat{a}})
\mapsto
\DiscS(i,\hat{a})^{-1}
(e_i\otimes \basis)\otimes e_{\hat{a}}
\mapsto
\DiscS(i,\hat{a})^{-1}
\chi(i)(\basis\otimes e_{\hat{a}}).
\end{equation*}


\subsubsection*{Tensor product $\X^\chi\otimes_A\X^\phi\cong\bigoplus_{t,\;\DiscS(t,\cdot)=\chi\phi \Discq}\C_{t+L}$:}

We consider
\begin{equation*}
\begin{tikzcd}
(\X^\chi \otimes \C_x)\otimes \X^\phi
\ar{d}{\Discq(x)^{-1}}
\ar{rr}{\DiscS(x,t)}
&&
\X^\chi \otimes (\C_x\otimes \X^\phi)
\ar{d}{\rho}
\\
(\C_{-x}\otimes \X^\chi)\otimes \X^\phi
\ar{rr}{\rho}
&&
X^\chi \otimes \X^\phi
\end{tikzcd}
\end{equation*}
for all $\chi,\phi\in\widehat{L/2L}$ and $x\in L$, where $\X\otimes\X=\bigoplus_{t\in\R^d}\C_t$. The equivalence relation on the coequaliser from $(\basis\otimes e_x)\otimes \basisTwo$, for $\basis,\basisTwo$ basis vectors in the $1$-dimensional multiplicity spaces of $\X^\chi$ and $\X^\phi$, respectively, is
\begin{equation*}
\DiscS(x,t)(\basis\otimes \phi(x)\basisTwo)
\;\sim\;
\Discq(-x)\chi(-x)\basis\otimes \basisTwo.
\end{equation*}
Then the $\C_t$-term in $\X^\chi\otimes_A\X^\phi$ is nonzero for those $t\in \R^d$ with
\begin{equation*}
\DiscS(x,t)=\Discq(x)\chi(x)\phi(x)^{-1}.
\end{equation*}
In particular, $t\in L^*$ because the right-hand side vanishes for $x\in L$. Regarding the $A$-action, the associator in $\cC$ on $\C_a\otimes\X\otimes \X$ is trivial, and the action of $e_i\in A$ on $\C_t$ is
\begin{equation*}
e_i\otimes e_{\hat{t}+j}\mapsto \chi(i)e_{i+\hat{t}+j}
\end{equation*}
for $i,j\in L$. This is not the standard $A$-module structure on $\C_{\hat{t}+L}$; so we have to use a nontrivial identification
\begin{equation*}
\boxed{
\iota_{\X^{\chi},\X^{\phi}}\colon
e_{\hat{t}+k} \mapsto \chi(k)e_{\hat{t}+k}.}
\end{equation*}


\subsection*{Associators}

From these tensor products and identifications we can now compute the associators. A good additional check is that the respective morphism must be independent of the choice of representative (i.e.\ $i,j,k\in L$ below).


\subsubsection*{Associator $\C_{\hat{a}+L}\otimes_A \C_{\hat{b}+L} \otimes_A \C_{\hat{c}+L}$:}

We consider
\begin{equation*}
\begin{tikzcd}
(\C_{\hat{a}+i}\otimes \C_{\hat{b}+j}) \otimes \C_{\hat{c}+k}
\ar{d}{\sim}
\ar{r}{\id}
& \C_{\hat{a}+i}\otimes (\C_{\hat{b}+j} \otimes \C_{\hat{c}+k})
\ar{d}{\sim}
\\
(\C_{\hat{a}+i}\otimes_A \C_{\hat{b}+j}) \otimes_A \C_{\hat{c}+k}
\ar{d}{\DiscS(\hat{a}+i,j)}
& \C_{\hat{a}+i}\otimes_A (\C_{\hat{b}+j} \otimes_A \C_{\hat{c}+k})
\ar{d}{\DiscS(\hat{b}+j,k)}\\
\C_{\widehat{a+b}+i+j+u(a,b)} \otimes_A \C_{\hat{c}+k}
\ar{d}{\DiscS(\widehat{a+b}+i+j+u(a,b),k)}
& \C_{\hat{a}+k}\otimes_A \C_{\widehat{b+c}+j+k+u(b,c)}
\ar{d}{\DiscS(\hat{a}+i,j+k+u(b,c))}\\
\C_{\widehat{a+b+c}+i+j+k+u(a,b)+u(a+b,c)}
\arrow[dashed,r]{}{}
&\C_{\widehat{a+b+c}+i+j+k+u(a,b+c)+u(b,c)}
\end{tikzcd}
\end{equation*}
for chosen $a,b,c\in\Gamma$ and for all $i,j,k\in L$ corresponding to all direct summands of the tensor product. Recall that because $\epsilon=1$, it follows that $\DiscS(\cdot,u)$ is a homomorphism $\Gamma/2\Gamma\to \C^\times$ for $u\in L$, i.e.\ exponents $\pm 1$ are irrelevant and hence $\DiscS(i,j)=1$. Altogether, we get for this associator
\begin{align*}
&\DiscS(\widehat{a+b}+i+j+u(a,b),k)^{-1}
\DiscS(\hat{a}+i,j)^{-1}
\DiscS(\hat{b}+j,k)
\DiscS(\hat{a}+i,j+k+u(b,c))\\
&=\DiscS(\hat{a},k)^{-1}\DiscS(\hat{b},k)^{-1}\DiscS(\hat{a},j)^{-1}\DiscS(\hat{b},k)
\DiscS(\hat{a},j)\DiscS(\hat{a},k)\DiscS(\hat{a},u(b,c)),
\end{align*}
which simplifies to
\begin{equation*}
\boxed{
\omega(a,b,c)\coloneqq\DiscS(\hat{a},u(b,c)).
}
\end{equation*}
Of course this computation on $\cC_1$ coincides with the respective computation of the condensation of the discriminant form in \autoref{sec_DiscForm} for $\epsilon=1$.


\subsubsection*{Associator $\X^\chi\otimes_A \C_{\hat{a}+L} \otimes_A \C_{\hat{b}+L}$:}

We consider
\begin{equation*}
\begin{tikzcd}
(\X^\chi\otimes \C_{\hat{a}+i}) \otimes \C_{\hat{b}+j}
\ar{d}{\sim}
\ar{r}{\id}
& \X^\chi\otimes (\C_{\hat{a}+i} \otimes \C_{\hat{b}+j})
\ar{d}{\sim}
\\
(\X^{\chi} \otimes_A \C_{\hat{a}+i}) \otimes_A \C_{\hat{b}+j}
\ar{d}{\Discq(i)\chi(i)}
& \X^\chi\otimes_A (\C_{\hat{b}+i} \otimes_A \C_{\hat{b}+j})
\ar{d}{\DiscS(\hat{a}+i,j)}\\
\X^{\chi \DiscS(\hat{a},\cdot)} \otimes_A \C_{\hat{b}+j}
\ar{d}{\Discq(j)\chi(j)\DiscS(\hat{a},j)}
& \X^\chi \otimes_A \C_{\widehat{a+b}+i+j+u(a,b)}
\ar{d}{\Discq(i+j+u(a,b))\chi(i+j+u(a,b))}\\
\X^{\chi \DiscS(\hat{a}+\hat{b},\cdot)}
\arrow[dashed,r]{}{}
& \X^{\chi \DiscS(\hat{a}+\hat{b},\cdot)}
\end{tikzcd}
\end{equation*}
for $a,b\in\Gamma$, for all $i,j\in L$ and $\chi\in\widehat{L/2L}$. Recall that $\Discq$ on $L$ restricts to a group homomorphism $L/2L\to \C^\times$. Altogether, we obtain for this associator:
\begin{equation*}
\boxed{
\chi(u(a,b))\Discq(u(a,b)).
}
\end{equation*}


\subsubsection*{Associator $\C_{\hat{a}+L} \otimes_A \X^\chi \otimes_A \C_{\hat{b}+L}$:}

We then consider
\begin{equation*}
\begin{tikzcd}
(\C_{\hat{a}+i}\otimes \X^\chi) \otimes \C_{\hat{b}+j}
\ar{d}{\sim}
\ar{r}{\DiscS(\hat{a}+i,\hat{b}+j)}
&[2em] \C_{\hat{a}+i}\otimes (\X^\chi \otimes \C_{\hat{b}+j})
\ar{d}{\sim}
\\
(\C_{\hat{a}+i}\otimes_A \X^\chi) \otimes_A \C_{\hat{b}+j}
\ar{d}{\DiscS(\hat{a},i)^{-1}\chi(i)}
& \C_{\hat{a}+i}\otimes_A (\X^\chi \otimes_A \C_{\hat{b}+j})
\ar{d}{\Discq(j)\chi(j)}\\
\X^{\chi \DiscS(\hat{a},\cdot)} \otimes_A \C_{\hat{b}+j}
\ar{d}{\Discq(j)\chi(j)\DiscS(\hat{a},j)}
& \C_{\hat{a}+i} \otimes_A \X^{\chi \DiscS(\hat{b},\cdot)}
\ar{d}{\DiscS(\hat{a},i)^{-1}\chi(i)\DiscS(\hat{b},i)}\\
\X^{\chi \DiscS(\hat{a}+\hat{b},\cdot)}
\arrow[dashed,r]{}{}
& \X^{\chi \DiscS(\hat{a}+\hat{b},\cdot)}
\end{tikzcd}
\end{equation*}
for $a,b\in\Gamma$, $i,j\in L$ and $\chi\in\widehat{L/2L}$. Altogether, we get for this associator the same as in $\cC=\smash{\Vect_{\R^d}^{\DiscQ}}\oplus\Vect$ (for the chosen representatives), namely:
\begin{equation*}
\boxed{
\DiscS(\hat{a},\hat{b}).
}
\end{equation*}


\subsubsection*{Associator $\C_{\hat{a}+L} \otimes_A \C_{\hat{b}+L} \otimes_A \X^\chi$:}

We consider
\begin{equation*}
\begin{tikzcd}
(\C_{\hat{a}+i} \otimes \C_{\hat{b}+j}) \otimes \X^\chi
\ar{d}{\sim}
\ar{r}{\id}
&[2em] \C_{\hat{a}+i} \otimes (\C_{\hat{b}+j} \otimes \X^\chi)
\ar{d}{\sim}
\\
(\C_{\hat{a}+i} \otimes_A \C_{\hat{b}+j}) \otimes_A \X^\chi
\ar{d}{\DiscS(\hat{a}+i,j)}
&\C_{\hat{a}+i} \otimes_A (\C_{\hat{b}+j} \otimes_A \X^\chi)
\ar{d}{\DiscS(\hat{b},j)\chi(j)}\\
\C_{\widehat{a+b}+i+j-u(a,b)} \otimes_A \X^\chi
\ar{d}{\DiscS(\widehat{a+b},i+j-u(a,b))^{-1}\chi(i+j-u(a,b))}
& \C_{\hat{a}+i} \otimes_A \X^{\chi \DiscS(\hat{b},\cdot)}
\ar{d}{\DiscS(\hat{a},i)\,\chi(i)\DiscS(\hat{b},i)}\\
\X^{\chi \DiscS(\hat{a}+\hat{b},\cdot)}
\arrow[dashed,r]{}{}
& \X^{\chi \DiscS(\hat{a}+\hat{b},\cdot)}
\end{tikzcd}
\end{equation*}
for $a,b\in\Gamma$, for all$i,j\in L$ and $\chi\in\widehat{L/2L}$. Note that $\DiscS(\widehat{a+b},\cdot)=\DiscS(\hat{a}+\hat{b},\cdot)$ on $L$. Altogether, we obtain for this associator:
\begin{equation*}
\boxed{
\chi(u(a,b))\DiscS(\hat{a}+\hat{b},u(a,b)).
}
\end{equation*}


\subsubsection*{Associator $\X^\chi\otimes_A \X^\phi \otimes_A \C_{\hat{a}+L}$:}

We consider
\begin{equation*}
\begin{tikzcd}
&[-17pt]
(\X^\chi \otimes \X^\phi) \otimes \C_{\hat{a}+i}
\ar{d}{\sim}
\ar{r}{\id}
& \X^\chi \otimes (\X^\phi \otimes \C_{\hat{a}+i})
\ar{d}{\sim}
\\
&
(\X^\chi \otimes_A \X^\phi) \otimes_A \C_{\hat{a}+i}
\ar{d}{\chi(k)}
& \X^\chi \otimes_A (\X^\phi \otimes_A \C_{\hat{a}+i})
\ar{d}{\Discq(i)\phi(i)}\\
\;\;\,\bigoplus_{\hat{t}+k,\,\DiscS(\hat{t},\cdot)=\chi\phi \Discq}
\hspace{-1cm}
&
\C_{\hat{t}+k} \otimes_A \C_{\hat{a}+i}
\ar{d}{\DiscS(\hat{t}+k,i)}
& \X^\chi \otimes_A \X^{\phi \DiscS(\hat{a},\cdot)}
\ar{d}{\chi(k+i-u(t,a))}\\
\bigoplus_{\hat{t}+k,\,\DiscS(\hat{t},\cdot)=\chi\phi \Discq}
\hspace{-1cm}
&
\C_{\widehat{t+a}+k+i-u(t,a)}
\arrow[dashed,r]{}{}
& \C_{\widehat{t+a}+k+i-u(t,a)}
\end{tikzcd}
\end{equation*}
for $a\in\Gamma$, for all $i\in L$ and $\chi,\phi\in\widehat{L/2L}$, and in the final steps summing over all cosets $t$ and all elements $\hat{t}+k$ therein such that $\DiscS(\hat{t}+k,-)=\DiscS(\hat{t},-)$ coincides with $\chi\phi \Discq$ as a character on $L$. The right-hand side tensor product $\X^\chi \otimes_A \X^{\phi \DiscS(\hat{a},\cdot)}$ evaluates to some object indexed by $\hat{t}'+k'$ with $\DiscS(\hat{t}',\cdot)=\chi\phi \DiscS(\hat{a},\cdot)$, which for a morphism in $\cC$ has to match $\widehat{t+a}+k+i-u(t,a)$. Altogether, we get for the associator:
\begin{equation*}
\boxed{\chi(u(t,a)).}
\end{equation*}


\subsubsection*{Associator $\X^\chi\otimes_A \C_{\hat{a}+L} \otimes_A \X^\phi$:}

Then, we consider
\begin{equation*}
\begin{tikzcd}
&[-37pt]
(\X^\chi\otimes \C_{\hat{a}+i}) \otimes \X^\phi
\ar{d}{\sim}
\ar{r}{\DiscS(\hat{t}+k,\hat{a}+i)}
&[2em] \X^\chi\otimes (\C_{\hat{a}+i} \otimes \X^\phi)
\ar{d}{\sim}
\\
&
(\X^\chi\otimes_A \C_{\hat{a}+i}) \otimes_A \X^\phi
\ar{d}{\Discq(i)\chi(i)}
& \X^\chi\otimes_A (\C_{\hat{a}+i} \otimes_A \X^\phi)
\ar{d}{\DiscS(\hat{a},i)\phi(i)}\\
&
\X^{\chi\DiscS(\hat{a},\cdot)} \otimes_A \X^\phi
\ar{d}{\chi(k)\DiscS(\hat{a},k)}
& \X^\chi \otimes_A \X^{\phi\DiscS(\hat{a},\cdot)}
\ar{d}{\chi(k)}\\
\bigoplus_{\hat{t}+k,\,\DiscS(\hat{t},\cdot)=\chi\DiscS(\hat{a},\cdot)\phi \Discq}
\hspace{-1cm}
&
\C_{\hat{t}+k}
\arrow[dashed,r]{}{}
& \C_{\hat{t}+k}
\end{tikzcd}
\end{equation*}
for $a\in\Gamma$, for all $i\in L$ and $\chi,\phi\in\widehat{L/2L}$. The condition on $\hat{t}$ allows us to rewrite $\DiscS(\hat{t},i)=\chi(i)\phi(i)\Discq(i)$. Then, altogether we get for the associator:
\begin{equation*}
\boxed{
\DiscS(\hat{t},\hat{a}).
}
\end{equation*}


\subsubsection*{Associator $\C_{\hat{a}+L} \otimes_A \X^\chi\otimes_A \X^\phi $:}

We consider
\begin{equation*}
\begin{tikzcd}
&[-13pt]
(\C_{\hat{a}+i} \otimes \X^\chi)\otimes \X^\phi
\ar{d}{\sim}
\ar{r}{\id}
& \C_{\hat{a}+i} \otimes (\X^\chi\otimes \X^\phi)
\ar{d}{\sim}
\\
&
(\C_{\hat{a}+i} \otimes_A \X^\chi)\otimes_A \X^\phi
\ar{d}{\DiscS(\hat{a},i)\chi(i)}
& \C_{\hat{a}+i} \otimes_A (\X^\chi\otimes_A \X^\phi)
\ar{d}{\chi(k)}\\
\bigoplus_{\hat{t}+k,\,\DiscS(\hat{t},\cdot)=\chi\DiscS(\hat{a},\cdot)\phi \Discq}
\hspace{-1cm}
&
\X^{\chi\DiscS(\hat{a},\cdot)} \otimes_A \X^\phi
\ar{d}{\chi(i+k-u(a,t))\DiscS(\hat{a},i+k-u(a,t))}
& \C_{\hat{a}+i} \otimes_A \C_{\hat{t}+k}
\ar{d}{\DiscS(\hat{a}+i,k)}\\
\bigoplus_{\hat{t}+k,\,\DiscS(\hat{t},\cdot)=\chi\DiscS(\hat{a},\cdot)\phi \Discq}
\hspace{-1cm}
&
\C_{\widehat{a+t}+i+k-u(a,t)}
\arrow[dashed,r]{}{}
& \C_{\widehat{a+t}+i+k-u(a,t)}
\end{tikzcd}
\end{equation*}
for $a\in\Gamma$, for all $i\in L$ and $\chi,\phi\in\widehat{L/2L}$. Altogether, we obtain the associator:
\begin{equation*}
\boxed{
\chi(u(a,t))\DiscS(a,u(a,t)).
}
\end{equation*}


\subsubsection*{Associator $\X^\chi \otimes_A \X^\phi\otimes_A \X^\psi $:}

Finally, we consider
\begin{equation*}
\begin{tikzcd}
&[-30pt]
(\X^\chi \otimes \X^\phi)\otimes \X^\psi
\ar{d}{\sim}
\ar{r}{\substack{ \eps N^{-1/2} \\ \DiscS(\hat{t}+k,\hat{r}+l)^{-1}}}
&[3.5em] \X^\chi \otimes (\X^\phi\otimes \X^\psi)
\ar{d}{\sim}
\\
&
(\X^\chi \otimes_A \X^\phi)\otimes_A \X^\psi
\ar{d}{\chi(k)}
& \X^\chi \otimes_A (\X^\phi\otimes_A \X^\psi)
\ar{d}{\phi(l)}\\
\bigoplus_{\substack{
\hat{t}+k,\,\DiscS(\hat{t},\cdot)=\chi\phi \Discq \\
\hat{r}+l,\,\DiscS(\hat{r},\cdot)=\phi\psi \Discq }}
\hspace{-1cm}
&
\C_{\hat{t}+k} \otimes_A \X^\psi
\ar{d}{\DiscS(\hat{t},k)\psi(k)}
& \X^\chi \otimes_A \C_{\hat{r}+l}
\ar{d}{\Discq(l)\chi(l)}\\
\qquad\;\;\bigoplus_{\substack{
\hat{t}+k,\,\DiscS(\hat{t},\cdot)=\chi\phi \Discq \\
\hat{r}+l,\,\DiscS(\hat{r},\cdot)=\phi\psi \Discq }}
\hspace{-1cm}
&
\X^{\chi\phi\psi \Discq}
\arrow[dashed,r]{}{}
& \X^{\chi\phi\psi \Discq}
\end{tikzcd}
\end{equation*}
for $\chi,\phi,\psi\in\widehat{L/2L}$. The final character is $\DiscS(\hat{t},\cdot)\psi=\chi\phi\psi \Discq=\chi\DiscS(\hat{r},\cdot)$, and there are $l$- and $k$-dependent terms $\Discq(l)\phi(l)\chi(l)\DiscS(\hat{t},l)=1$ and $\chi(k)\psi(k)\DiscS(\hat{t},k)\DiscS(\hat{r},l)=1$, respectively, which cancel.

At this point, we re-encounter the nonrigorous aspect of our constructive approach in the present setting, namely the normalisation factor $|\Gamma|=|\R|=\infty$. On the other hand, there is a (constant) sum over $k$ and $l$ in $L$, which is again infinite. We simply propose at this point one that should take the altogether normalisation factor $|2\Gamma|$, so that this associator is given by
\begin{equation*}
\boxed{
\eps |2\Gamma|^{-1/2}\DiscS(\hat{t},\hat{r})^{-1}.
}
\end{equation*}
In \cite{GLM24b}, \autoref{GLM2sec_evenTY}, we take this as a definition and verify that it defines a consistent braided $\Z_2$-crossed extension.


\subsection*{Tensor Structure}

We now come to the tensor structures:
\begin{equation*}
\begin{tikzcd}
g_*M\otimes g_*N /\!\sim
\ar{r}{\iota_{g_*M,g_*N}}
\ar{d}{\tau_{M,N}^g=\id}
&
g_*M\otimes_A g_*N
 \arrow[d, dashrightarrow,"\tau^{g_A}_{M,N}"]
 \\
 g_*(M\otimes N) /\!\sim
\ar{r}{g_*\iota_{M,N}}
&
g_*(M\otimes_A N)
\end{tikzcd}
\end{equation*}
They arise from applying $g_*$ to the identifications $\iota_{M,N}$ above, which does in general not agree with the identification $\iota_{g_*M,g_*N}$. Note that we have assumed a strict action $\tau^g=\id$ on the base category $\cC_1$.

The first result on $\cC_1$ agrees of course with the action of the group on a condensation in \autoref{sec_DiscFormAut}.
{\allowdisplaybreaks
\begin{align*}
\tau^{g_A}_{\hat{a},\hat{b}}
&=\DiscS(\widehat{-a}-u(a,-a)-i,-u(b,-b)-j)^{-1}
\cdot\DiscS(\hat{a}+i,j) \\
&=\DiscS(a,u(b,-b))^{-1},\\[+2mm]
\tau^{g_A}_{\chi,\hat{a}}
&=(\Discq(-u(a,-a)-i)\chi(-u(a,-a)-i))^{-1}
\cdot (\Discq(i)\chi(i)\\
&=\Discq(u(a,-a))\chi(u(a,-a))),\\[+2mm]
\tau^{g_A}_{\hat{a},\chi}
&=(\DiscS(a,-u(a,-a)-i)^{-1}\chi(-u(a,-a)-i))^{-1}
\cdot(\DiscS(a,i)^{-1}\chi(i))\\
&=\DiscS(a,u(a,-a))\chi(u(a,-a)),\\[+2mm]
\tau^{g_A}_{\chi,\phi}(t)
&=\chi(-u(t,-t)-k)^{-1}\cdot \chi(k)\\
&=\chi(u(t,-t))
\end{align*}
}%
for all $a,b\in \Gamma$, $i,j\in L$ and $\chi,\phi\in\widehat{L/2L}$.


\subsection*{Braiding}

Then, we study the braiding:


\subsubsection*{Braiding $\C_{\hat{a}+L}\otimes_A \C_{\hat{b}+L}$:}

We consider
\begin{equation*}
\begin{tikzcd}
\C_{\hat{a}+i}\otimes \C_{\hat{b}+j}
\ar{r}{\DiscS(\hat{a}+i,\hat{b}+j)}
\ar{d}{\sim}
&[2em]
\C_{\hat{b}+L}\otimes \C_{\hat{a}+L}
\ar{d}{\sim}
\\
\C_{\hat{a}+L}\otimes_A \C_{\hat{b}+L}
\ar{d}{\DiscS(\hat{a}+i,j)}
&
\C_{\hat{b}+j}\otimes_A \C_{\hat{a}+i}
\ar{d}{\DiscS(\hat{b}+j,i)}\\
\C_{\widehat{a+b}+i+j-u(a,b)}
\arrow[dashed,r]{}{}
&\C_{\widehat{a+b}+i+j-u(a,b)}
\end{tikzcd}
\end{equation*}
for $a,b\in\Gamma$ and $i,j\in L$. Altogether, the braiding is simply the Tambara-Yamagami braiding on the representatives:
\begin{equation*}
\boxed{
\sigma(a,b)\coloneqq\DiscS(\hat{a},\hat{b}).
}
\end{equation*}


\subsubsection*{Braiding $\C_{\hat{a}+L}\otimes_A \X^\chi$:}

We consider
\begin{equation*}
\begin{tikzcd}
\C_{\hat{a}+i}\otimes \X^\chi
\ar{r}{\Discq(\hat{a}+i)^{-1}}
\ar{d}{\sim}
&[1em]
\X^\chi\otimes \C_{\hat{a}+i}
\ar{d}{\sim}
\\
\C_{\hat{a}+L}\otimes_A \X^\chi
\ar{d}{\DiscS(\hat{a},i)^{-1}\chi(i)}
&
\X^\chi\otimes_A \C_{\hat{a}+L}
\ar{d}{\Discq(i)\chi(i)}\\
\X^{\chi\DiscS(a,\cdot)}
\arrow[dashed,r]{}{}
&\X^{\chi\DiscS(a,\cdot)}
\end{tikzcd}
\end{equation*}
for $a\in\Gamma$, $i\in L$ and $\chi\in\widehat{L/2L}$. We can use equation~(\ref*{GLM2formula_halfQadditivity}) in \cite{GLM24b} to express $\Discq(\hat{a}+i)=\Discq(\hat{a})\Discq(i)\DiscS(\hat{a},i)$. Then altogether, the braiding is simply the Tambara-Yamagami braiding on the representative:
\begin{equation*}
\boxed{
\Discq(\hat{a})^{-1}.
}
\end{equation*}


\subsubsection*{Braiding $\X^\chi\otimes_A \C_{\hat{a}+L}$:}

We then consider
\begin{equation*}
\begin{tikzcd}
\X^\chi\otimes \C_{\hat{a}+i}
\ar{r}{\Discq(\hat{a}+i)^{-1}}
\ar{d}{\sim}
&[1.5em]
\C_{\widehat{-a}-u(a,-a)-i}\otimes \X^\chi
\ar{d}{\sim}
\\
\C_{\hat{a}+L}\otimes_A \X^\chi
\ar{d}{\Discq(i)\chi(i)}
&
\C_{\hat{a}+L}\otimes_A \X^\chi
\ar{d}{\DiscS(\widehat{-a},-u(a,-a)-i)^{-1}\chi(-u(a,-a)-i)}
\\
\X^{\chi\DiscS(a,\cdot)}
\arrow[dashed,r]{}{}
&\X^{\chi\DiscS(a,\cdot)}
\end{tikzcd}
\end{equation*}
for $a\in\Gamma$, $i\in L$ and $\chi\in\widehat{L/2L}$. Altogether, the braiding is the following modification of the Tambara-Yamagami braiding:
\begin{equation*}
    \boxed{
    \Discq(\hat{a})^{-1}\cdot \DiscS(a,u(a,-a))\chi(u(a,-a)).
    }
\end{equation*}


\subsubsection*{Braiding $\X^\chi\otimes_A \X^\phi$:}

Finally, we consider
\begin{equation*}
\begin{tikzcd}
&[-42pt]
\X^\chi\otimes \X^\phi
\ar{r}{\alpha \Discq(\hat{t}+k)}
\ar{d}{\sim}
&[1em]
\X^\phi\otimes \X^\chi
\ar{d}{\sim}
\\
&
\X^\chi\otimes_A \X^\phi
\ar{d}{\chi(k)}
&
\X^\phi\otimes \X^\chi
\ar{d}{\phi(k)}
\\
\bigoplus_{\hat{t}+k,\DiscS(t,\cdot)=\chi\phi \Discq}
&
\C_{\hat{t}+k}
\arrow[dashed,r]{}{}
&\C_{\hat{t}+k}
\end{tikzcd}
\end{equation*}
for $\chi,\phi\in\widehat{L/2L}$. The difference between $\Discq(\hat{t}+k)$ and $\Discq(\hat{t})$ is precisely caught by the difference between $\chi(k)$ and $\phi(k)$. Altogether, the braiding is simply the Tambara-Yamagami braiding on the representative:
\begin{equation*}
\boxed{
\alpha \,\Discq(\hat{a}).
}
\end{equation*}
Note that the normalisation factor $\alpha^2=\eps\e(-\rk(L)/8)$ as a square root of the signature is finite even for the groups $\R^d$, and our calculation predicts that it remains the same after condensation.


\subsection*{Twists}

We forgo the determination of the twists, as we have not developed the respective calculations and they can be given directly in the resulting category. We observe that the twist remains the same after condensation.


\begin{concl}\label{concl:TYeven}
We constructed (nonrigorously, but very explicitly) the condensation $(\cC_1\oplus\cC_g)^\loc_A$ of $\cC=\smash{\Vect_{\R^d}^{\DiscQ}}\oplus\Vect$. This produces the braided $\Z_2$-crossed tensor category
\begin{equation*}
\LM(\Gamma,\sigma,\omega,\delta,\eps\,|\,q,\alpha,\beta)
\end{equation*}
defined in \cite{GLM24b}, \autoref*{GLM2sec_evenTY}, with $\Gamma,\sigma,\omega,\delta,\eps$ and $q,\alpha,\beta$ given in terms of lattice data as in \autoref*{GLM2sec_latticedata} of \cite{GLM24b}.
\end{concl}


\section{\autoref{idea_1} for \VOA{}s}\label{sec_VOA}

The purpose of this section is to discuss \autoref{idea_1} in the context of \voa{}s. To this end, we first describe the precise setting in which the categorical notions treated in this work appear and then introduce the relevant examples, namely Heisenberg and lattice \voa{}s. Then, we show how \autoref{idea_1} can be used to describe lattice $\Z_2$-orbifolds of Tambara-Yamagami type. More general lattice $\Z_2$-orbifolds will appear in \autoref{sec_VOAIdea2}.


\subsection{\VOA{}s and \texorpdfstring{$G$}{G}-Orbifolds}\label{sec_VOAIntro}

A \emph{\voa{}}~$\V$ is, roughly speaking, a formal power series version of a unital, commutative, associative algebra with a derivation \cite{Bor86,FLM88}. The algebra structure on $\mathcal V$ in this sense is given by a bilinear map
\begin{align*}
Y(\cdot, z)\colon \; \mathcal \V \times \mathcal \V &\to \V((z)),\\
(v,w) &\mapsto Y(v, z)w
\end{align*}
with $Y(v, z)w$ a formal Laurent series in $z$ with coefficients in $\V$, subject to certain well-motivated axioms that encode the commutativity and associativity up to re-expansion. Standard mathematical textbooks on the subject are \cite{FBZ04,Kac90,FHL93,LL04}. Moreover, $\V$ is endowed with an action of the infinite-dimensional Virasoro Lie algebra generated by $L_n$, $n\in \Z$, that is compatible with the action of the Witt algebra on the formal variable $z$. In particular, $\V$ has a $\Z$-grading by eigenvalues of $L_0$, which acts semisimply on $\V$.

Similarly, one can define a \emph{\voa{} module} $\mathcal{M}$ over $\V$ whose module structure is defined by a bilinear map
\begin{equation*}
Y_{\mathcal{M}}(\cdot, z)\colon \; \mathcal \V \times \mathcal{M} \to \mathcal{M}((z)),
\end{equation*}
again subject to essentially the same axioms. (Ordinary) modules are also graded by $L_0$-eigenvalues, with $L_0$ acting also semisimply on $\V$, but the degrees are not necessarily in $\Z$. For $\V$-modules $\mathcal{M}$, $\mathcal{N}$ and $\mathcal{L}$ there is a version of a $\V$-balanced map, called \emph{intertwining operator}
\begin{equation*}
Y_{\mathcal{M,N}}^{\mathcal{L}}(\cdot, z)\colon \; \mathcal \mathcal{M} \times \mathcal{N} \to \mathcal{L}[\log(z)]\{z\},
\end{equation*}
where $\mathcal{L}[\log(z)]\{z\}$ denotes series involving arbitrary complex exponents and logarithms of the formal variable $z$, describing a multivalued function on $\C\setminus\{0\}$ with a regular singular point at $z=0$. The vector space of intertwining operators of this type is denoted by $\binom{\mathcal{M}\;\; \mathcal{N}}{\mathcal{L}}$.

\medskip

As is proved in the series of papers beginning with \cite{HL92,HL94,HL95}\nocite{HL92,HL94,HL95,HL95a,HL95b,Hua95b} and \cite{HLZ1}\nocite{HLZ1,HLZ2,HLZ3,HLZ4,HLZ5,HLZ6,HLZ7,HLZ8} and reviewed in \cite{HL13}, in good cases (e.g., when $\V$ is $C_2$-cofinite) the category $\Rep(\V)$ of (suitable) $\V$-modules can be endowed with a braided tensor category structure. The tensor product $\mathcal{M}\otimes\mathcal{N}$ (the universal object admitting an intertwining operator from $\mathcal{M}\times\mathcal{N}$, in analogy to commutative rings) is often also written as $\mathcal{M}\boxtimes\mathcal{N}$ or $\mathcal{M}\boxtimes_\V\mathcal{N}$. The braiding comes from analytically continuing $Y_{\mathcal{M,N}}^{\mathcal{L}}(\cdot, z)$ counterclockwise to $Y_{\mathcal{M,N}}^{\mathcal{L}}(\cdot, e^{\pi\i}z)$. So, the double braiding is nontrivial if and only if the intertwining operator is multivalued.

\medskip

Let $G$ be a finite group acting (faithfully) on a \voa{} $\V$ in a way that is compatible with $Y$. The subspace $\smash{\V^G}$ fixed by $G$ is again a \voa{}, and hence its category of modules $\smash{\Rep(\V^G)}$ is again a braided tensor category (provided that $\smash{\V^G}$ is sufficiently regular). Certainly, restricting a $\V$-module to $\smash{\V^G}$ produces a (typically not simple) $\smash{\V^G}$-module. However, this does not exhaust all $\smash{\V^G}$-modules. To overcome this, for every $g\in G$ one defines \emph{$g$-twisted $\V$-modules} $\mathcal{M}$ as vector spaces with a module structure
\begin{equation*}
Y_{\mathcal{M}}(\cdot, z)\colon \; \mathcal \V \times \mathcal{M} \to \mathcal{M}[\log(z)]\{z\}
\end{equation*}
involving now also multivalued functions, such that the multivaluedness for an analytic continuation $z\mapsto e^{2\pi\i} z$ around $z=0$ is controlled by $g$ as
\begin{equation*}
Y_{\mathcal{M}}(a,e^{2\pi\i} z)=Y_{\mathcal{M}}(g.a, z)
\end{equation*}
for $a\in V$. In particular, restricting to $\smash{\V^G}$ again yields an untwisted module.

It is expected that in good cases the category of $g$-twisted $\V$-modules for $g\in G$ forms a braided $G$-crossed tensor category (see \autoref{def_Gcrossed})
\begin{equation*}
\cC=\Rep^G(\V)=\bigoplus_{g\in G}\cC_g,
\end{equation*}
where $\cC_g=\smash{\Rep^{g}(\V)}$ corresponds to the $g$-twisted modules so that in particular $\cC_1=\Rep^1(\V)=\Rep(\V)$. Further, the equivariantisation $\cC\sslash G=\smash{\Rep(\V^G)}$ is supposed to be the braided tensor category of $\smash{\V^G}$-modules.
\begin{equation*}
\begin{tikzcd}
\Rep(\V)\arrow[hookrightarrow]{rrrr}{\text{$G$-crossed extension}}\arrow[dashed]{ddrrrr}{}&&&&\Rep^G(\V)\arrow{dd}{\text{$G$-equivariantisation}}\\\\
&&&&\Rep(\V^G)
\end{tikzcd}
\end{equation*}
Then, in particular, every $\smash{\V^G}$-module is a direct sum of $g$-twisted $\V$-modules for possibly several $g\in G$.

\medskip

We now say more precisely under which conditions the above statements hold. For simplicity, in this paper we restrict ourselves to very regular \voa{}s, namely those whose categories of untwisted modules are (finite and semisimple) modular tensor categories. However, we expect the methods of this paper to be applicable also in the setting where the (choice of) category of modules is only a braided tensor category (in particular, not necessarily semisimple or rigid).

A \voa{} is called \emph{\strat{}} if it is simple, rational, $C_2$-cofinite, self-contragredient and of CFT-type. Moreover, we say that a \voa{} satisfies the \emph{positivity condition} or \emph{$G$-positivity condition} if all irreducible modules or irreducible $g$-twisted modules for $g\in G$, respectively, except for the \voa{} itself have only positive $L_0$-eigenvalues.

Let $\V$ be a \strat{} \voa{}. Then Huang showed that the category $\cB=\Rep(\V)$ of (untwisted) $\V$-modules admits the structure of a \mtc{} (in the usual finite and semisimple sense) \cite{Hua08b}. If $\V$ also satisfies the positivity condition, then the quantum dimensions of all simple objects in $\cB$ are positive (in particular, $\cB$ is pseudo-unitary) and they coincide with the quantum dimensions obtained as certain limits of the module $q$-characters \cite{DJX13,DLN15}. Moreover, under this condition, the $\SMatrix$- and $\TMatrix$-matrix of $\V$ (defined via the modular invariance of $q$-characters) coincide with those defined for $\cB=\Rep(\V)$ after normalisation (see \cite{DLN15} for details).

Let $\V$ be a \voa{} and $G<\Aut(\V)$ a finite group of automorphisms of $\V$ (acting faithfully on $\V$). Suppose that the \fpvosa{} $\smash{\V^G}$ is \strat{} and satisfies the positivity condition (then it is not difficult to see that $\V$ is \strat{} and satisfies the $G$-positivity condition), so that $\smash{\Rep(\V^G)}$ and $\cB=\Rep(\V)$ are (pseudo-unitary) \mtcs{}. (The strong rationality of $\smash{\V^G}$ is implied, for instance, if $\V$ is \strat{} and $G$ is a solvable group \cite{Miy15,CM16}, see also \cite{McR21,McR21b}. In that situation, the positivity condition for $\smash{\V^G}$ follows from the $G$-positivity condition for $\V$ since all irreducible $\smash{\V^G}$-modules appear in a $g$-twisted $\V$-module for some $g\in G$ \cite{DRX17}).

Let $\cC=\smash{\Rep^G(\V)}=\bigoplus_{g\in G}\cC_g$ with $\cC_g=\Rep^{g}(\V)$ be the category of $g$-twisted $\V$-modules for all $g\in G$ (so that $\cB=\cC_1$). Under the stated assumptions, $\cC$ admits the structure of a braided $G$-crossed tensor (in fact, fusion) category extending $\cB$ and $\smash{\Rep(\V^G)}$ is the equivariantisation $\cC\sslash G$ of $\cC$ \cite{McR21b}.

\medskip

Recall that in the modular tensor category $\Rep(\V)$ or $\smash{\Rep(\V^G)}$ due to \cite{Hua08b} the ribbon twist is given by the unique eigenvalue of $e^{2\pi\i L_0}$ acting on each simple object. For completeness we also establish the meaning of the $G$-ribbon twist:
\begin{lem}\label{lem_GribbonVOA}
Let $\V$ be a \strat{} \voa{} and $G<\Aut(\V)$ finite such that $\smash{\V^G}$ is again \strat{} (e.g., if $G$ is solvable), so that the category $\smash{\Rep^G(\V)}$ of $g$-twisted $\V$-modules for $g\in G$ is a braided $G$-crossed fusion category and the modular tensor category category $\smash{\Rep(\V^G)}$ of modules of $\smash{\V^G}$ coincides with the $G$-equivariantisation $\smash{\Rep^G(\V)}\sslash G$.

Then there is a $G$-ribbon twist on $\smash{\Rep^G(\V)}$ that coincides for each simple object with one of the eigenvalues of $e^{2\pi\i L_0}$ on it.

This $G$-ribbon twist on $\smash{\Rep^G(\V)}$ yields only positive quantum dimensions if and only if the corresponding ribbon twist on $\smash{\Rep(\V^G)}$ has only positive quantum dimensions (e.g., when $\V$ is $G$-positive).
\end{lem}
\begin{proof}
The first assertion follows from the connection between the $G$-ribbon twist and the ribbon twist of the $G$-equivariantisation discussed in \autoref*{GLM2lem_ribbon} of \cite{GLM24b}, the second assertion follows from \autoref*{GLM2cor_pseudo} of \cite{GLM24b}.
\end{proof}


\subsection{Lattice \VOA{}s}\label{sec_LatticeVOA}

Well-known examples of \voa{}s are Heisenberg and lattice \voa{}s, which we review in the following.

\medskip

Let $\h=(\R^d,\DiscQ)$ be a real vector space equipped with a nondegenerate quadratic form $\DiscQ\colon\R^d\to\C^\times$ or, equivalently, with a nondegenerate symmetric bilinear form $\langle\cdot,\cdot\rangle\colon\R^d\times\R^d\to\R$ via $\DiscQ(v)=\e(\langle v,v\rangle/2)$. To $\h$ we associate the rank-$d$ Heisenberg \voa{} $\Heis$. To be precise, as we consider all \voa{}s over the base field $\C$, the \voa{} $\Heis$ only depends on the complexification $\h\otimes_\R\C\cong\C^d$. However, it is useful to retain the information about the real form $\h$ of $\h\otimes_\R\C$, which defines a real form of $\Heis$, in particular when considering a suitable category of $\Heis$-modules.

The category of (ordinary, completely reducible, with real ``momenta'') modules of $\Heis$ is the infinite pointed braided tensor category
\begin{equation*}
\Rep(\Heis)=\Vect_\h^\DiscQ
\end{equation*}
(see, e.g., \cite{CKLR19}). The irreducible modules are given by $\V_\alpha$ for $\alpha\in\h$. The usual definition of the twist $\theta_{\V_\alpha}=\e(h(\V_\alpha))=\DiscQ(\alpha)$ for $\alpha\in\h$ equips $\Rep(\Heis)$ with a (pseudo-unitary) ribbon structure. Here, $h(\V_\alpha)+\Z$ denotes the $\Z$-coset of $L_0$-eigenvalues of $\V_\alpha$.

\medskip

To a positive-definite, even lattice $L$ (with ambient space $\h=L\otimes_\Z\R$) we associate the lattice \voa{} $\V_L$ \cite{Bor86,FLM88}, which is \strat{} (in contrast to $\Heis$) and satisfies the positivity condition. As $\V_L$ is a conformal extension of $\Heis$, the category $\Rep(\V_L)$ of modules of the lattice \voa{} $\V_L$ can be thought of arising from $\Rep(\Heis)=\smash{\Vect_\h^\DiscQ}$ by an (infinite) condensation by a commutative algebra $A=\C_\epsilon[L]$ related to $L$, as described in \autoref{sec_latticeDiscriminantForm}, a fact that we shall use repeatedly in the following.

We describe the category of modules $\Rep(\V_L)$, which must be a \mtc{} by \cite{Hua08b}, in more detail. Let $(\Gamma,Q)\coloneqq L^*\!/L$ denote the discriminant form of $L$ (see \autoref{sec:prel}). The irreducible modules of $\V_L$ up to equivalence are $\V_{\alpha+L}$, $\alpha+L\in\Gamma$, and over $\Heis$ they decompose as $\V_{\alpha+L}=\bigoplus_{\beta\in \alpha+L} \V_{\beta}$ \cite{Don93}. In particular, the $L_0$-grading is in $
h(\V_{\alpha+L})+\Z=\langle\alpha,\alpha\rangle/2+\Z$ for $\alpha+L\in\Gamma$. As braided tensor category
\begin{equation*}
\Rep(\V_L)\cong\Vect_\Gamma^Q,
\end{equation*}
and choosing a representing abelian $3$-cocycle $(\sigma,\omega)$ corresponds to fixing a basis of the (here $1$-dimensional) spaces of $\V_L$-intertwining operators \cite{DL93}. Again, the usual definition of the twist $\theta_{\V_{\alpha+L}}=\e(h(\V_{\alpha+L}))=Q(\alpha+L)$ equips $\Rep(\V_L)$ with a (pseudo-unitary) ribbon structure (in fact, it is a \mtc{}).

\medskip

In the following sections, we shall consider examples of braided $G$-crossed tensor categories $\smash{\Rep^G(\V_L)}$ associated with lattice \voa{}s and finite automorphism groups $G<\Aut(\V_L)$, applying the methods developed in this text.

Recall that the definition of a lattice \voa{} $\V_L$ involves an (up to isomorphism irrelevant) choice of $2$-cocycle $\epsilon\colon L\times L\to\{\pm1\}$ satisfying $\epsilon(\alpha,\beta)/\epsilon(\beta,\alpha)=\e(\langle\alpha,\beta\rangle/2)$ for all $\alpha,\beta\in L$. This perfectly mirrors the construction described in \autoref{sec_latticeDiscriminantForm} of $\Rep(\V_L)=\smash{\Vect_\Gamma^Q}$ as a condensation by an infinite commutative algebra object $A=\C_\epsilon[L]$ in $\Rep(\Heis)=\smash{\Vect_\h^\DiscQ}$.

Now, given a lattice isometry $\auto\in\OO(L)$, its lifts $g=\phi_\eta(\auto)$ to $\Aut(\V_L)$ require a choice of function $\eta\colon L\to\C^\times$ satisfying $\eta(\alpha+\beta)\epsilon(\alpha,\beta)=\epsilon(\auto\alpha,\auto\beta)\eta(\alpha)\eta(\beta)$ for all $\alpha,\beta\in L$, with values in the torsion subgroup of $\C^\times$ if $g$ has finite order (see \cite{Lep85,FLM88,Bor92}). All finite-order automorphisms in $\Aut(\V_L)$, up to conjugacy, are of this form \cite{DN99,HM22}. Again, this mirrors the description in \autoref{sec_DiscFormAut} of the lifting of an automorphism of the commutative algebra $A=\C_\epsilon[L]$ to the condensation $\smash{\Vect_\Gamma^Q}$.

A lift $g=\phi_\eta(\auto)$ is called a \emph{standard lift} \cite{Lep85} if $\eta$ takes values in $\{\pm1\}$ and is trivial on $L^\auto$. Its order is either $|\auto|$ or $2|\auto|$. All standard lifts are conjugate in $\Aut(\V_L)$ \cite{EMS20a}. (On the other hand, it is always possible to find a lift $g=\phi_\eta(\auto)$ that has the same order as $\auto$, but it may not be a standard lift \cite{HM22}.)

\medskip

The twisted modules of $\V_L$ under standard lifts are described in \cite{DL96,BK04}, and using \cite{Li96} we can describe the twisted modules under arbitrary finite-order automorphisms (see, e.g., \cite{EMS20a,HM22}).

However, the point of this paper (see \autoref{idea_1}) is to use no (or only very little) explicit information on the twisted $\V_L$-modules. Rather, we use that a finite group $G<\Aut(\V_L)$ defines a categorical action on the untwisted modules $\Rep(\V_L)$ and then consider the possible braided $G$-crossed extensions, which are the candidates for the twisted modules $\smash{\Rep^G(\V_L)}$. At this point we only remark:
\begin{rem}\label{rem:latpos}
$\smash{\Rep^G(\V_L)}$ or equivalently (see \autoref{lem_GribbonVOA}) its equivariantisation $\smash{\Rep^G(\V_L)}\sslash G=\smash{\Rep(\V_L^G)}$ have only positive quantum dimensions (and, in particular, are always pseudo-unitary) because $\smash{\V_L^G}$ satisfies the positivity condition (see, e.g., \cite{Moe18,HM22}).
\end{rem}


\subsection{Fixed-Point Free Lattice Orbifolds of Order 2}\label{sec_VOALatticeMinusOne}

Continuing in the setting of the previous section, let $\auto\in\OO(L)$ be an isometry of $L$ of order~$2$, which induces an isometry of $\Gamma=L^*\!/L$, and assume that the latter has no nontrivial fixed-points, i.e.\ that $\Gamma^\auto=\{0+L\}$. Equivalently, $|\Gamma|$ is odd and $\auto$ acts on $\Gamma$ (but not necessarily on $L$) as multiplication by $-1$. Let $d_0$ and $d_1$ be the dimensions of the eigenspaces of $\auto$ on $\mathfrak{h}$ corresponding to the eigenvalues $\e(0/2)=1$ and $\e(1/2)=-1$, respectively. Then $d_0+d_1=\rk(L)$ and $d_0=\rk(L^\auto)$. Here, under the assumptions made on $\auto$, $d_0$ must be an integer multiple of $4$.

We consider a standard lift $g=\phi_\eta(\auto)\in\Aut(\V_L)$ and assume for simplicity that $\auto$ does not exhibit order doubling, i.e.\ that $\langle\alpha,\auto\alpha\rangle\in2\Z$ for all $\alpha\in L$. Then also $g$ has order~$2$ \cite{Bor92}. (Otherwise we could modify $g$ by an inner automorphism so that it has order~$2$. See, e.g., \cite{HM22}.) Let $G=\langle g\rangle\cong\Z_2$. This is the setting.

\medskip

The goal is to determine the braided $\Z_2$-crossed tensor category $\smash{\Rep^G(\V_L)}$ for this group $G\cong\Z_2$ using \autoref{idea_1} or \autoref{thm_ourENO_braided}, i.e.\ using as little explicit information about the $g$-twisted $\V_L$-modules as possible. Then, the equivariantisation is the orbifold category $\smash{\Rep(\V_L^G)}\cong\smash{\Rep^G(\V_L)}\sslash G$.

Indeed, based on the given $\Z_2$-action on the category $\Rep(\V_L)$ of untwisted modules, it follows from \autoref{cor_TY} and \autoref{thm_ourENO_braided} that $\smash{\Rep^G(\V_L)}$ is a braided $\Z_2$-crossed Tambara-Yamagami category $\smash{\Rep^G(\V_L)}\cong\smash{\Vect_\Gamma^Q}[\Z_2,\eps]$ for $\eps\in\{\pm1\}$. In particular, $\smash{\Rep^G(\V_L)}\cong\smash{\Vect_\Gamma\oplus\Vect}$, i.e.\ $\V_L$ has a unique irreducible $g$-twisted module up to isomorphism, which we call $\V_L(g)$ and identify with $\X$. This is already stated in \cite{BK04}, but here it follows purely from the theory of braided $G$-crossed extensions.

Finally, to fix the parameter $\eps\in\{\pm1\}$, we do in fact need some explicit information on the twisted sector. In this example, it suffices to consider the ribbon twist of $\V_L(g)$, which comes from the $L_0$-grading by \autoref{lem_GribbonVOA}. By \cite{DL96}, the unique irreducible $g$-twisted $\V_L$-module $\V_L(g)$ has $L_0$-grading in
\begin{equation}\label{eq_anomaly}
h(\V_L(g))+\frac{1}{2}\N=\frac{d_1}{16}+\frac{1}{2}\N,
\end{equation}
where $d_1/16$ is the (Heisenberg) vacuum anomaly. This also shows that $\V_L$ satisfies the $G$-positivity condition, which is true for arbitrary lattice \voa{}s and finite groups, as commented above. We obtain (cf.\ \cite{Bis18,EG23}):
\begin{thm}\label{thm_latTY}
Let $\V_L$ be the lattice \voa{} for a positive-definite, even lattice $L$ with $(\Gamma,Q)=L^*\!/L$ of odd order. Let $\Z_2\cong G=\langle g\rangle<\Aut(\V_L)$ where $g=\phi_\eta(\auto)$ is a standard lift of an isometry $\auto\in\OO(L)$ acting by $-\!\id$ on $\Gamma$.

Then, the category of modules $\smash{\Rep^G(\V_L)}$ is equivalent to the braided $\Z_2$-crossed extension of $\smash{\Vect_\Gamma^Q}$ with a (pseudo-unitary) $\Z_2$-ribbon structure defined in \autoref{cor_TY}
\begin{align*}
\Vect_\Gamma^Q[\Z_2,\eps]
&=
\TY(\Gamma,B_Q^{1/2},\eps\mid Q^{1/2}, \alpha,\eps/\alpha)
\end{align*}
with
\begin{equation*}
\eps=\e(-d_0/8)\Bigl(\frac{2}{|\Gamma|}\Bigr)\in\{\pm1\}
\end{equation*}
and where $\alpha$ is an (irrelevant) choice of square root of
\begin{equation*}
\alpha^2=\eps\,G(\Gamma,q^{-1})
=\eps\left(\frac{2}{|\Gamma|}\right)\e(-\rk(L)/8)=\e(-d_0/4-d_1/8).
\end{equation*}
Both solutions for $\alpha$ produce nontrivially equivalent $G$-crossed braidings (see \autoref{rem_alphaEquivalence}).
\end{thm}
\begin{proof}
It only remains to determine the sign $\eps\in\{\pm1\}$. As just explained, we do this by studying the $G$-ribbon twist in the $g$-twisted sector. By \autoref{thm_TYribbon}, the square of the twist of the simple object $\X$ in $\smash{\TY(\Gamma,B_Q^{1/2},\eps\mid Q^{1/2}, \alpha,\eps/\alpha)}$ is
\begin{align*}
\beta^2&=\eps^2/\alpha^2=1/\alpha^2=\bigl(\eps\, G(\Gamma,q^{-1})\bigr)^{-1}=\eps\left(\frac{2}{|\Gamma|}\right)\e(\rk(L)/8)\\
&=\eps\left(\frac{2}{|\Gamma|}\right)\e((d_0+d_1)/8)
\end{align*}
Identifying $\X$ with the twisted module $\V_L(g)$, the above expression must coincide by \autoref{lem_GribbonVOA} with the following expression involving the $L_0$-weight grading from equation~\eqref{eq_anomaly}: $\e(2h(\V_L(g)))=\e(d_1/8)$. This yields the asserted expression for $\eps$.
\end{proof}
The theorem describes the braided $\Z_2$-crossed tensor category of $G$-twisted $\V_L$-modules, including the braiding and the associators. This is in contrast to the results in \cite{ADL05}, where only the simple objects and fusion rules are determined (on the level of the equivariantisation), not the complete modular tensor category $\smash{\Rep(\V_L^G)}$.

\begin{rem}
The relation between the ribbon twist and the $L_0$-weights of the irreducible modules, which we assert in \autoref{lem_GribbonVOA} and use in the above proof, was to our knowledge not yet established in the literature in the $G$-crossed setting.
\end{rem}

In the \voa{} literature, it is traditionally more common to study the equivariantisation:
\begin{cor}\label{cor:latTY}
In the situation of \autoref{thm_latTY}, the modular tensor category $\smash{\Rep(\V_L^G)}\cong\smash{\Rep^G(\V_L)}\sslash G$ is given by the modular tensor category (partially) described in \autoref*{GLM2sec_equiv} of \cite{GLM24b} for the special case of $|\Gamma|$ odd.
\end{cor}
In \autoref*{GLM2sec_equiv} of \cite{GLM24b}, we describe the simple objects, fusion rules and modular data of the orbifold category $\smash{\Rep(\V_L^G)}$, but we could easily also write down all coherence data, as these follow immediately from the ones in $\smash{\Rep^G(\V_L)}$, which are described in detail in \autoref*{GLM2sec_evenTY} of \cite{GLM24b}. This equivariantisation (for the special case of $|\Gamma|$ odd) is also given in \cite{GNN09}. The fusion rules are in perfect agreement with \cite{ADL05}.

\begin{ex}
As a special case, suppose that $L$ is unimodular, i.e.\ that $|\Gamma|=1$. Then $\rk(L)=d_0+d_1\in8\N$. The lattice \voa{} $\V_L$ is holomorphic, i.e.\ its category of modules is $\Rep(\V_L)=\Vect$, and hence the equivariantisation $\smash{\Rep(\V_L^G)}$ of the braided $G$-crossed extension $\smash{\Rep^G(\V_L)}$ for any (solvable) finite group $G$ must be a twisted Drinfeld double $\cD_\omega(G)$ for some cocycle $\omega$. Indeed, this was conjectured in \cite{DVVV89,DPR90}, proved for cyclic $G$ in \cite{Moe16,EMS20a} and in full generality in \cite{DNR21b,GR24}.

On the other hand, for $G\cong\Z_2$, the above proposition identifies $\smash{\Rep^G(\V_L)}$ with $\TY(\Gamma,\sigma,\eps\,|\,q,\alpha,\eps/\alpha)$. Clearly, $\Gamma$, $\sigma$ and $q$ are trivial and
\begin{align*}
\eps&=\e(-d_0/8)=\e(d_1/8)=\e(2h(\V_L(g)))\in\{\pm1\},\\
\alpha&=\pm\eps^{1/2}=\pm\e(d_1/16)=\pm\e(h(\V_L(g))).
\end{align*}
In other words, the sign $\eps$ determines exactly which of the two possible cocycles $\omega$ in $\cD_\omega(G)$ is the correct one. This is in perfect agreement with the results in \cite{Moe16,EMS20a}.
\end{ex}


\section{\autoref{idea_2} for \VOA{}s}\label{sec_VOAIdea2}

We demonstrate the idea of computing the braided tensor category of an orbifold \voa{} (or a braided $G$-crossed extension) from an existing one using a condensation by a commutative algebra $A$, as described in \autoref{thm_currentExtVsCrossedExt} (see \autoref{idea_2}). In the \voa{} context, this amounts to computing the orbifold category $\smash{\Rep^G(\V)}$ from a known orbifold category $\smash{\Rep^G(\mathcal{W})}$, where $G$ is a group of automorphisms of $\V$ and $\V\supset\mathcal \mathcal{W}$ is a conformal extension preserved by $G$ i.e.\ such that $G$ restricts to automorphisms of $\mathcal{W}$.


\subsection{From Tambara-Yamagami to Tambara-Yamagami}\label{sec_VOAIdea2_TYtoTY}

As a first check of \autoref{idea_2} in the purely categorical setting, we showed in \autoref{prop_TYtoTY} that the $A$-condensation of the Tambara-Yamagami category
\begin{equation*}
\Vect_{\DiscGamma}^\DiscQ[\Z_2,\bar\eps]=\TY(\DiscGamma,B_\DiscQ^{1/2},\bar\eps|\DiscGamma|^{1/2} \mid \DiscQ^{1/2}, \bar\alpha,\bar\eps/\bar\alpha)\supset\Vect_\DiscGamma^\DiscQ,
\end{equation*}
a braided $\Z_2$-crossed extension of $\smash{\Vect_\DiscGamma^\DiscQ}$ with $|\DiscGamma|$ odd, with respect to the algebra $A=\C_\epsilon[I]$ defined by the $\DiscQ$-isotropic subgroup $I\subset\DiscGamma$ (fixed by $\Z_2$) is again of the form
\begin{equation*}
\Vect_\Gamma^Q[\Z_2,\eps]=\TY(\Gamma,B_Q^{1/2},\eps|\Gamma|^{1/2} \mid Q^{1/2},\alpha,\eps/\alpha)\supset\Vect_\Gamma^Q,
\end{equation*}
a braided $\Z_2$-crossed extension of $\smash{\Vect_\Gamma^Q}$, where $\Gamma=I^\perp/I$, $Q$ is the quadratic form induced by $\DiscQ$ on $I^\perp/I$, and where the sign choices $\bar\eps=\eps$ agree.

\medskip

On the \voa{} side (see \autoref{sec_VOALatticeMinusOne}), this corresponds, for example, to the situation of \voa{}s $\V_{\bar{L}}$ and $\V_L$ associated with two positive-definite, even lattices $\bar{L}\subset L$ of the same rank, so that
\begin{equation*}
\Rep(\V_{\bar{L}})=\Vect_\DiscGamma^\DiscQ\quad\text{and}\quad\Rep(\V_L)=\Vect_\Gamma^Q,
\end{equation*}
with discriminant forms $\DiscGamma=\bar{L}^*\!/\bar{L}$ and $\Gamma=L^*\!/L$ of odd order. Then the image of $L$ under the modulo-$\bar{L}$ projection map is a $\DiscQ$-isotropic subgroup $I=\bar{L}^*\!/\bar{L}$ with new discriminant form $\Gamma\cong I^\perp/I$.

Further, we consider a lattice involution $\nu\in\OO(L)$ restricting to $\OO(\bar{L})$ and satisfying that $\DiscGamma^\nu$ and $\Gamma^\nu$ are trivial, lifting in a suitable way to an involution $g$ of $\V_{\bar{L}}$ and of $\V_L$. For $G=\langle g\rangle\cong\Z_2$ we study the braided $G$-crossed extensions
\begin{equation*}
\Rep^G(\V_{\bar{L}})=\Vect_\DiscGamma^\DiscQ[\Z_2,\bar\eps]\quad\text{and}\quad\Rep^G(\V_L)=\Vect_\Gamma^Q[\Z_2,\eps].
\end{equation*}
Our categorical result, that the $A$-condensation produces the choice of braided $\Z_2$-crossed tensor category with the same sign $\bar\eps=\eps$, matches the \voa{} result in \autoref{thm_latTY}, that the choice of twist in the form of the sign $\bar\eps=\e(-d_0/8)\smash{(\frac{2}{|\DiscGamma|})}$ with $d_0=\rk(\bar{L}^\nu)$ and $|\DiscGamma|$ only depends on data that do not change when replacing $\bar{L}$ by the extension $L$. Indeed, the first factor only depends on the action of the automorphism $\nu$ on the ambient vector space $\h=\bar{L}\otimes_\Z\R=L\otimes_\Z\R$, i.e.\ the Heisenberg \voa{}, and not on the particular lattices $\bar{L}$ or~$L$, and the second factor does not change when replacing $\bar{L}$ by $L$ by elementary properties of the Kronecker symbol.


\subsection{Reflection Orbifolds of Lattice \VOA{}s}\label{sec_VOAevenTV}

We now come to the main result for \voa{}s, based on the categorical statement in \autoref{thm_evenTY_here} (from \cite{GLM24b}, based on \autoref{idea_2} here). We obtain a complete categorical answer for the braided $\Z_2$-crossed tensor category of the reflection orbifold of an arbitrary lattice \voa{} (with only a minor simplifying assumption), strengthening the results in \cite{ADL05}, where only the simple objects and fusion rules were determined on the level of the equivariantisation.

\medskip

Let $L$ be an even, positive-definite lattice of rank $d$ and $(\Gamma,Q)=L^*\!/L$ its discriminant form. Let $\V_L$ be the corresponding lattice \voa{}. For simplicity, we assume that $L$ satisfies the stronger evenness condition in \autoref*{GLM2ass_strongeven} of \cite{GLM24b} (see \autoref{sec:prel}). Then \autoref{sec_latticeDiscriminantForm} describes how
\begin{equation*}
\Rep(\V_L)=\Vect_\Gamma^Q
\end{equation*}
arises as condensation of $\smash{\Vect_{\R^d}^\DiscQ}$ by the commutative algebra $A=\C_\epsilon[L]$ with $\epsilon=1$, equipped with the representing abelian $3$-cocycle $\sigma(a,b)=\DiscS(\hat{a},\hat{b})=\e(\langle\hat{a},\hat{b})/2\rangle$ and $\omega(a,b,c)=\DiscS(\hat{a},u(b,c))$ for cosets $a,b\in\Gamma$, with choices of representatives $\hat{a},\hat{b}\in L^*$ and the $2$-cocycle $u(a,b)=\hat{a}+\hat{b}-\smash{\widehat{a+b}}$ on $\Gamma$ with values in $L$.

Now, let $\auto=-\id$ be the reflection isometry of $L$. It induces an isometry of the discriminant form $\Gamma=L^*\!/L$ that also acts by multiplication by $-1$, and this is precisely how any lift $g$ of $\auto$ will permute the simple objects of $\Rep(\V_L)=\smash{\Vect_\Gamma^Q}$. We remark that while the action of $\auto=-\id$ on $L$ has only the trivial fixed-point, the action on $\Gamma$, which may be of even order, may have nontrivial fixed points, meaning that we are no longer necessarily in the situation of \autoref{sec_VOALatticeMinusOne}.

As $\auto$ is a fixed-point free lattice isometry, all its lifts $g=\phi_\eta(\auto)$ to $\Aut(\V_L)$ are conjugate and again of order~$2$ (see, e.g., \cite{EMS20a}). Set $G\coloneqq\langle g\rangle\cong\Z_2$. In fact, \autoref{prop_ActionOnCondensation} describes the action of $G$ on $\smash{\Vect_\Gamma^Q}$ as follows, where $\epsilon=1$ allows us to take $\eta=1$: for all $a,b\in\Gamma$
\begin{align*}
g_{*}&\colon \C_a \longmapsto \C_{-a},\\[-1mm]
\tau^g&\colon g_{*}(\C_a \otimes_A \C_b)
\xrightarrow{\DiscS(\hat{a},-\hat{b}-\widehat{-b})}
g_{*}(\C_a) \otimes_A g_*(\C_b).
\end{align*}

We shall describe the (pseudo-unitary) braided $\Z_2$-crossed extension $\smash{\Rep^G(\V_L)}$ of $\Rep(\V_L)$ consisting of untwisted and $g$-twisted $\V_L$-modules, which as an abelian category will be of the form
\begin{equation*}
\Rep^G(\V_L)\cong\Vect_\Gamma\oplus\Vect_{\Gamma/2\Gamma}.
\end{equation*}
We observe that the braided tensor category $\Rep(\V_L)$ together with the $G$-action (including the tensor structure $\tau^g$) coincides with the untwisted sector $\cC_1$ of the braided $\Z_2$-crossed tensor categories $\cC=\LM(\Gamma,\sigma,\omega,\delta,\eps\,|\,q,\alpha,\eps/\alpha)$ we construct in \autoref*{GLM2sec_evenTY} of \cite{GLM24b} (see \autoref{thm_evenTY_here}), depending on a sign $\eps$. Hence, by the classification of braided $G$-crossed extensions (see \autoref{thm_ourENO_braided} and \autoref{idea_1}), the category $\smash{\Rep^G(\V_L)}$ is precisely given by one of these two categories. We only need to determine the correct sign $\eps$.
\begin{thm}\label{thm_latmain}
Let $L$ be a positive-definite lattice satisfying the stronger evenness condition in \autoref{sec:prel}. Then the (pseudo-unitary) $\Z_2$-crossed ribbon fusion category of possibly twisted modules of the lattice \voa{} $\V_L$ under (a lift $G\cong\Z_2$ of) the action of $\langle-\id\rangle\cong\Z_2$ on $L$ is given by
\begin{equation*}
\Rep^G(\V_L)\cong\LM(\Gamma,\sigma,\omega,\delta,\eps\,|\,q,\alpha,\eps/\alpha)
\end{equation*}
from \autoref*{GLM2sec_evenTY} in \cite{GLM24b} where $(\sigma,\omega)$ is as above, $(q,\delta)$ is as defined in \autoref*{GLM2sec_latticedata} of \cite{GLM24b} and
\begin{equation*}
\eps=+1\in\{\pm1\},
\end{equation*}
and where $\alpha$ is an (irrelevant) choice of square root of
\begin{equation*}
\alpha^2=\eps\,G_\delta(\Gamma,q^{-1})=\eps\,|2\Gamma|^{-1/2}\sum_{{a\in \delta}} q(a)^{-1}=\eps\,\e(-\rk(L)/8).
\end{equation*}
\end{thm}
\begin{proof}
It only remains to determine the sign $\eps$. We proceed analogously to the proof of \autoref{thm_latTY} and consider the square of the $G$-ribbon twist of any simple object $\X^{\bar{x}}$, $\bar{x}\in\Gamma/2\Gamma$, in the twisted sector. It satisfies
\begin{align*}
\beta^2&=\eps^2/\alpha^2=1/\alpha^2=(\eps\,G_\delta(\Gamma,q^{-1}))^{-1}=\eps\,\e(\rk(L)/8).
\end{align*}
We identify the $\X^{\bar{x}}$ with the irreducible $g$-twisted $\V_L$-modules. Since $g$ is a lift of the involution $-\!\id$ on $L$, all these twisted modules have the same $L_0$-weight grading, which is simply given by the vacuum anomaly $\rk(L)/16$ \cite{DL96}, as in equation~\eqref{eq_anomaly} but with $d_0=0$ and $d_1=\rk(L)$. Hence, by \autoref{lem_GribbonVOA}, $\beta^2$ must equal $\e(\rk(L)/8)$, which yields the assertion.
\end{proof}
Some remarks are warranted.

It should not be too difficult to drop the simplifying assumption on the even lattice~$L$. What changes is that one now has to keep track of the $2$-cocycle $\epsilon$ on $L$ and consequently the function $\eta$ on $L$. In particular, in the (nonrigorous) derivation of the braided $\Z_2$-crossed tensor category in \autoref{sec_idea2_infTY}, one needs to consider the condensation by the nontrivially twisted algebra $A=\C_\epsilon[L]$.

Note that on the level of simple objects and fusion rules, the above result agrees with the \voa{} result in \cite{ADL05}, but goes much further by explicitly stating, amongst other things, the associators and braiding (after passing to the equivariantisation, which is a straightforward computation).

Philosophically speaking, the argument given in \cite{ADL05} on the level of \voa{}s is actually somewhat similar to the (nonrigorous) derivation in \autoref{sec_idea2_infTY} as an infinite condensation of an infinite Tambara-Yamagami category $\smash{\Vect_{\R^d}}\oplus\Vect$, in the sense that the fusion rules for the fixed-point \voa{} $\smash{\V_L^G}$ are determined by viewing $\smash{\V_L^G}$ as an extension of the fixed points $\smash{\Heis^G}$ of the underlying Heisenberg \voa{} $\Heis$ with $\h=L\otimes_\Z\R\cong\R^d$.

The result in \autoref{thm_latmain} generalises \autoref{thm_latTY}. However, because of the simplifying assumption on the lattice $L$, namely that is of the form $L=\sqrt{2}K$ for an integral lattice $K$, the Kronecker symbol in the Gauss sum in \autoref{thm_latTY} does not appear. (Compare also \autoref*{GLM2conj_partialGauss2} to \autoref*{GLM2conj_partialGauss1} in \cite{GLM24b}.)

Finally, recall from \autoref{exm_counterexampleA1} (see also \autoref*{GLM2rem_minusoneaction} in \cite{GLM24b}) that it is crucial to not only look at the permutation action of $G$ on the simple objects in the untwisted sector, but to consider the full categorical $G$-action on the untwisted sector, including the tensor structure $\tau^g$. In particular, there are examples of $\Z_2$-orbifolds of lattice \voa{}s where $\Z_2$ acts as $-\!\id$ on $\Gamma$, i.e.\ where the corresponding braided $\Z_2$-crossed tensor category $\smash{\Rep^G(\V_L)}$ is $\smash{\Vect_\Gamma\oplus\Vect_{\Gamma/2\Gamma}}$ as abelian category, but where $\smash{\Rep^G(\V_L)}$ is not given by the braided $\Z_2$-crossed tensor category constructed in \autoref*{GLM2sec_evenTY} of \cite{GLM24b}. One indication of this would be if not all irreducible $g$-twisted modules had the same $L_0$-weights modulo $1/2$.

\medskip

We also state the corresponding result for the $\Z_2$-equivariantisation:
\begin{cor}\label{cor_latmain}
In the situation of \autoref{thm_latmain}, the modular tensor category $\smash{\Rep(\V_L^G)}\cong\smash{\Rep^G(\V_L)}\sslash G$ is given by the modular tensor category (partially) described in \autoref*{GLM2sec_equiv} of \cite{GLM24b} for $\eps=1$.
\end{cor}
In particular, the simple objects, fusion rules and modular data are given in \autoref*{GLM2prop_eqivsimple}, \autoref*{GLM2prop_equi_fusion} and \autoref*{GLM2prop_equivmodular} of \cite{GLM24b}, agreeing with \cite{ADL05}.


\subsection{General \texorpdfstring{$\Z_2$}{Z\_2}-Orbifolds of Lattice \VOA{}s}\label{sec_VOAIdea2_Els}

With the method in \autoref{thm_currentExtVsCrossedExt} (\autoref{idea_2}), we can also extend the knowledge of lattice orbifolds under fixed-point free automorphisms of order~$2$ (see \autoref{sec_VOALatticeMinusOne}) to determine the braided $G$-crossed tensor category $\smash{\Rep^{\Z_2}(\V_L)}$ of arbitrary $\Z_2$-orbifolds of lattice \voa{}s $\V_L$. The setup is parallel to the strategy employed in \cite{BE15,Els17} (on the level of simple objects and fusion rules), namely by reducing to the fixed-point free and the trivial case.

We leave the concrete description of these braided $\Z_2$-crossed tensor categories $\smash{\Rep^{\Z_2}(\V_L)}$ for future work, bur rather only describe the general idea.

\medskip

Let $L$ be a positive-definite, even lattice and consider again the ambient vector space $\h=L\otimes_\Z\R$. Let $\auto\in\OO(L)$ be an arbitrary automorphism of $L$ of order $2$ and consider a standard lift $g\in\Aut(\V_L)$, which we assume to be again of order~$2$ for simplicity (the extension to nonstandard lifts or order doubling is not difficult). Set $G=\langle g\rangle\cong\Z_2$. Then $\h$ orthogonally decomposes into $\auto$-eigenspaces
\begin{equation*}
\h=\h_+\oplus \h_-
\end{equation*}
with eigenvalues $\pm1$, respectively. We consider the intersection of $L$ with these eigenspaces
\begin{equation*}
L_\pm \coloneqq L\cap \h_\pm,\quad L_+\oplus L_- \subset L,
\end{equation*}
which are again even lattices, and their direct sum has finite index in $L$. We also consider the corresponding discriminant forms
\begin{equation*}
\Gamma_{\pm}\coloneqq L_\pm^* / L_\pm.
\end{equation*}
The image $I$ of $L$ in $\Gamma_+\oplus\Gamma_-$ is an isotropic subgroup and $I^\perp/I\cong L^*\!/L$ is the discriminant form of $L$. Hence, the category of untwisted modules $\Rep(\V_L)\cong\smash{\Vect_\Gamma^Q}$ is the condensation of $\smash{\Rep(\V_{L_+\oplus L_-})\cong\Vect_{\Gamma_+}^{Q_{+}}\boxtimes\Vect_{\Gamma_-}^{Q_-}}$ by an algebra object corresponding to $I$.

On the other hand, the action of $\auto$ on the lattice $L_+\oplus L_-$ splits: by design, it acts trivially on $L_+$ and by $-\!\id$ on $L_-$. Similarly, the lift $g\in\Aut(\V_L)$ restricts to $\V_{L_+\oplus L_-}\cong\V_{L_+}\otimes\V_{L_-}$ and acts as identity on $\V_{L_+}$ and as a standard lift of $-\!\id$ on $\V_{L_-}$. Hence, we obtain the trivial (and nonfaithful) action of $G$ on $\smash{\Vect_{\Gamma_+}^{Q_{+}}}$ and the (faithful) action on $\smash{\Vect_{\Gamma_-}^{Q_-}}$ assumed in \autoref{thm_latmain}. With the additional simplifying assumption that $L_-$ is strongly even, \autoref{thm_evenTY_here}, the main result of \cite{GLM24b}, then describes the braided $G$-crossed extensions of $\smash{\Vect_{\Gamma_-}^{Q_-}}$.

So, by \autoref{thm_currentExtVsCrossedExt} or \autoref{idea_2}, the braided $G$-crossed tensor category $\smash{\Rep^G(\V_L)}$ we are looking for can be obtained as the condensation of the braided $G$-crossed tensor category
\begin{align*}
\Rep^G(\V_{L_+\oplus L_-})&\cong
\bigl(\Vect_{\Gamma_+}\oplus\Vect_{\Gamma_+}\bigr)\boxtimes_{\Z_2}\LM(\Gamma_-,\ldots)\\
&\cong\bigl(\Vect_{\Gamma_+}^{Q_+}\boxtimes\Vect_{\Gamma_-}^{Q_-}\bigr)\oplus\bigl(\Vect_{\Gamma_+}\boxtimes\Vect\bigr)
\end{align*}
by the algebra object corresponding to $I$. Here, $\LM(\Gamma_-,\ldots)$ is the braided $\Z_2$-crossed extension of $\smash{\Vect_{\Gamma_-}^{Q_-}}$ from \autoref{thm_evenTY_here}. We summarise this in the following commuting diagram:
\begin{equation*}
\begin{tikzcd}
\Vect_{\Gamma_+}^{Q_{+}}\boxtimes\Vect_{\Gamma_-}^{Q_-}
\arrow[rightsquigarrow]{d}{}
\arrow[hookrightarrow]{rr}{\text{$\Z_2$-crossed ext.}}
&&
\Rep^G(\V_L)
\arrow[rightsquigarrow]{d}{}
\\
\Vect_\Gamma^Q
\arrow[hookrightarrow]{rr}{\text{$\Z_2$-crossed ext.}}
&&
\begin{tabular}{c}
$\bigl(\Vect_{\Gamma_+}\oplus\Vect_{\Gamma_+}\bigr)\boxtimes_{\Z_2}\LM(\Gamma_-,\ldots)$\\[+1mm]
$\cong\bigl(\Vect_{\Gamma_+}^{Q_+}\boxtimes\Vect_{\Gamma_-}^{Q_-}\bigr)\oplus\bigl(\Vect_{\Gamma_+}\boxtimes\Vect\bigr)$
\end{tabular}
\end{tikzcd}
\end{equation*}


\bibliographystyle{alpha_noseriescomma}
\bibliography{references}{}

\end{document}